\documentclass[11pt,a4paper]{article}

\usepackage{authblk}

\usepackage[margin=2.5cm]{geometry}
\usepackage{amsmath,amssymb,amsthm,mathtools,bm}
\usepackage{algorithm}
\usepackage{algorithmic}
\usepackage{enumitem}
\usepackage{hyperref}
\usepackage{xcolor}
\usepackage{booktabs}
\usepackage{graphicx}
\usepackage{tikz-cd}
\usepackage{tikz}
\usetikzlibrary{arrows.meta,positioning,fit,calc,decorations.pathreplacing}
\allowdisplaybreaks
\hypersetup{colorlinks=true, linkcolor=blue!70!black, citecolor=green!50!black, urlcolor=blue!60!black}

\theoremstyle{plain}
\newtheorem{theorem}{Theorem}[section]
\newtheorem{proposition}[theorem]{Proposition}
\newtheorem{lemma}[theorem]{Lemma}

\newtheorem{assumption}{Assumption}

\theoremstyle{definition}

\theoremstyle{remark}
\newtheorem{remark}[theorem]{Remark}

\DeclareMathOperator{\rank}{rank}
\DeclareMathOperator{\diag}{diag}

\newcommand{\R}{\mathbb{R}}
\newcommand{\X}{\mathcal{X}}
\newcommand{\Y}{\mathcal{Y}}

\newcommand{\D}{\mathcal{D}}
\newcommand{\A}{\mathcal{A}}
\newcommand{\Pmap}{\mathcal{P}}
\newcommand{\J}{\mathrm{J}}
\newcommand{\rvec}{\boldsymbol{\rho}}

\newcommand{\norm}[1]{\lVert #1\rVert}          
\newcommand{\ip}[2]{\langle #1,#2\rangle}
\newcommand{\ran}{\operatorname{ran}}

\newcommand{\vecop}{\operatorname{vec}}
\newcommand{\Tfwd}{\overset{\scriptscriptstyle\to}{\mathcal T}}
\newcommand{\Tbwd}{\overset{\scriptscriptstyle\leftarrow}{\mathcal T}}

\newcommand{\W}{\mathcal{W}}
\newcommand{\M}{\mathcal{M}}
\newcommand{\Fmap}{\mathcal{F}}
\newcommand{\Res}{\mathcal{R}}
\newcommand{\Cmap}{\mathcal{C}}
\newcommand{\diff}{\mathrm{D}}
\newcommand{\tr}{\tau}                       
\newcommand{\Vq}{\mathsf{V}}

\newcommand{\bigip}[2]{\bigl\langle #1,#2\bigr\rangle}
\newcommand{\Bigip}[2]{\Bigl\langle #1,#2\Bigr\rangle}
\newcommand{\dsp}{d_{x}}
\newcommand{\dpa}{d_{\xi}}
\newcommand{\Lat}{\Gamma}

\newcommand\restr[2]{\ensuremath{\left.#1\right|_{#2}}}
\newcommand{\Neff}{N_{\mathrm{eff}}}
\newcommand{\Bgram}{\Upsilon}
\newcommand{\Cmat}{\mathsf{C}}
\newcommand{\Kmat}{\mathsf{K}}

\newcommand{\RA}{R_{\mathcal A}}
\newcommand{\ND}{N_{\mathcal D}}
\newcommand{\Diag}{\operatorname{Diag}}

\newcommand{\had}{\circ}
\newcommand{\DGram}{\widehat{\mathrm{G}}}

\title{\bfseries Tensor-Train Compressed Separable PINNs:\\[4pt]
A Curvature-Aware Optimization Framework for\\ Parametric PDEs
in High Dimensions}

\author[1]{Denis Korolev\thanks{Corresponding author. E-mail: \href{mailto:korolev@wias-berlin.de}{\texttt{korolev@wias-berlin.de}}}}
\author[1]{Martin Eigel\thanks{E-mail: \href{mailto: eigel@wias-berlin.de}{\texttt{eigel@wias-berlin.de}}}}

\affil[1]{Weierstrass Institute for Applied Analysis and Stochastics (WIAS),
Anton-Wilhelm-Amo-Stra{\ss}e~39, 10117 Berlin, Germany}

\date{\today}

\begin{document}
\maketitle

\noindent
\begin{abstract}
In this work, we develop a second-order optimization framework for physics-informed neural networks (PINNs) applied to high-dimensional parametric partial differential equations (PDEs). The framework is built on the Gauss--Newton pullback metric, which provides an operator-informed notion of curvature in parameter space and connects the method to the broader family of natural gradient schemes. We show that, for coordinate-separable neural architectures and linear differential operators (or linearized operators in the nonlinear case) admitting a finite separable representation, the residual Jacobian inherits a structured separable factorization. This yields an exact compressed formulation of the Gauss--Newton step in a reduced space, without assembling the full residual Jacobian on the exponentially large tensor-product collocation grid. The dimension of the reduced space (the effective compressed dimension) is determined by the local collocation grid sizes, the separable operator structure, and the contraction pattern of the architecture, thereby avoiding dependence on the full tensor-product grid size and replacing dense linear algebra in parameter space by a substantially smaller structured problem. Within our framework, we investigate canonical polyadic and tensor-train parametrizations and derive their full algebraic characterization relevant to the Gauss--Newton method, including the structure of the residual Jacobian, the resulting compressed system, and its effective compressed dimension. Numerical experiments on high-dimensional PDEs, including parametric problems, demonstrate the high efficiency of the proposed compressed Gauss--Newton method, which achieves substantially lower errors than tensor-compressed first-order baselines with orders of magnitude fewer iterations and only a fraction of the computing time.
\end{abstract}




\bigskip

\section{Introduction}
\label{sec:introduction}


High-dimensional parametric PDEs arise in uncertainty quantification, quantum
chemistry, and PDE-constrained optimization with large parameter spaces
\cite{cohen2015approximation, garreis2017constrained, heinkenschloss2025optimization, khoromskaia2015tensor}.
In many such applications, the main computational challenge is not a single PDE
solve but the repeated evaluation of the parameter-to-solution map, for
instance to estimate quantities of interest or to optimize a design under
uncertainty
\cite{antil2023ttrisk, antil2025moreau, antil2025tensor, benner2016lowrank, garreis2017constrained, heinkenschloss2025optimization}.
Classical approaches address this challenge through sparse approximation,
reduced-order modeling, and low-rank tensor methods, exploiting structure of
the solution map to mitigate the curse of dimensionality
\cite{cohen2015approximation, KhoromskijSchwab2011,Nouy2015, quarteroni2015reduced}.
Neural approaches target the same object from a function-learning perspective: a
single network represents the solution for all parameters
\cite{bhattacharya2021model,kovachki2023neural}.
What distinguishes them from the low-rank methods above is therefore not the map
being approximated, but the way it is constructed. Neural operators are trained
on precomputed solution snapshots, whereas physics-informed neural solvers learn
directly from the governing equations and require no such data
\cite{kovachki2023neural,wang2021learning}. The computational burden is then shifted from repeated PDE solves to the
training problem itself. In high-dimensional parametric settings,
physics-informed objectives can be expensive and difficult to optimize to high
accuracy, making optimization a central bottleneck for neural PDE solvers
\cite{de2024operator, krishnapriyan2021failure, wang2022and}. Consistently,
neural surrogates for parametric PDEs reported so far have been less accurate
than classical alternatives or have required training times that offset the
savings from avoiding repeated PDE solves.
To bridge the gap between accurate neural optimization and computationally efficient high-dimensional approximation, we propose combining the natural gradient approach for neural solvers with tensor-structured neural ansatz classes.
The tensor structure provides a compact
representation of high-dimensional solution maps, while the natural gradient
geometry enables high accuracy curvature-aware second-order optimization. Our goal is therefore to
retain the structural efficiency that makes low-rank methods effective in high
dimensions
\cite{cohen2015approximation,khoromskaia2015tensor}
while making the optimization of the resulting neural models
computationally tractable and accurate.

We now briefly outline our setting and the main idea of the proposed approach. Let $V$ denote the spatial solution space of a parametric PDE,
and let $u(\cdot,\xi)\in V$ denote its solution for a parameter $\xi\in\Xi$. Here, we assume that
\begin{align*}
x=(x_1,\ldots,x_{d_x})\in\Omega\subset\R^{d_x},
\qquad
\xi=(\xi_1,\ldots,\xi_{d_\xi})\in\Xi\subset\R^{d_\xi},
\qquad
d=d_x+d_\xi,
\end{align*}
and that $d$ is sufficiently large, either due to large $d_x$, large $d_\xi$,
or both, making the respective PDE high-dimensional. It is then natural to regard the
parameter-to-solution map
\begin{equation}\label{eq:parametric_solution_map_intro}
\Xi\ni\xi\longmapsto u(\,\cdot\,,\xi)\in V
\end{equation}
as an element of the Bochner space
$U:=L^2(\Xi;V)$.  We consider a neural
parametrization and its associated model class
\[
\mathcal P:\Theta\subseteq\R^P\to U,
\qquad
\theta\longmapsto\mathcal P(\theta)=:u_\theta,
\qquad
\mathcal M:=\mathcal P(\Theta)
=\{u_\theta:\theta\in\Theta\},
\]
where $P$ is the dimensionality of the neural network parametrization. Our objective is to approximate \eqref{eq:parametric_solution_map_intro} within $\mathcal M$. The parameters are determined by minimizing the objective
\begin{equation}\label{eq:parametric_pinn_objective_intro}
\min_{\theta\in\Theta} \, \mathcal L(\theta)
:=
\frac12\bigl\|\mathcal F(u_\theta)\bigr\|_{\mathcal Y}^{2}
=
\frac12\bigl\|\mathcal A(u_\theta)-f\bigr\|_{\mathcal Y}^{2},
\end{equation}
where $\mathcal F(v):=\mathcal A(v)-f$ is the residual operator containing the space--parameter differential
operator $\mathcal A:U\to\mathcal Y$ associated with the parametric PDE, and $\mathcal Y$ is a suitable Bochner space over $\Xi$ with values in the
PDE residual space, e.g., $L^2(\Omega\times\Xi)$. 
In this way, a single neural network approximates the entire parameter-to-solution map. This formulation falls naturally within the class of physics-informed neural network methods (PINNs), which recast the solution of differential equations as an optimization problem by incorporating the governing equations and boundary conditions into a loss functional defined over a neural network model class \cite{raissi2019physics,shin2023error}.

Solving \eqref{eq:parametric_pinn_objective_intro} efficiently is hampered by two difficulties: the well-known challenges of PINN optimization \cite{de2024operator, wang2022and} and the underlying high dimensionality. To address the first, we follow \cite{muller2023achieving, muller2024optimization} and employ the (regularized, or damped) Gauss--Newton pullback metric to precondition the gradient descent step:
\begin{equation}\label{eq:intro_gn_step}
\theta_{k+1}
=
\theta_k
-
\eta_{k} \, \big[\mathcal{G}(\theta_k) + \mu \mathrm{I}_{P}\big]^{-1}
\nabla\mathcal L(\theta_k),
\qquad
\mathcal G(\theta)
=
\diff\Res(\theta)^*\diff\Res(\theta).
\end{equation}
Here,
$\diff\Res(\theta)
:=
\diff\Fmap(u_\theta)\circ\diff\Pmap(\theta):
\R^P\to\mathcal Y$
is the linearized residual map, $\mathcal G(\theta) \in \mathbb{R}^{P\times P}$ is the Gramian representing the pullback of the residual geometry to parameter space, and $\mu>0$ is the regularization parameter. The Gauss--Newton method yields an operator-informed notion of curvature in parameter space, which places it within the broader family of natural gradient methods \cite{amari1998natural, martens2020new, muller2023achieving,nurbekyan2023efficient}. This curvature-aware preconditioning approach has been shown to substantially improve accuracy and robustness in many applications \cite{guzman2026improving,jnini2026curvature,jnini2025gauss}, particularly compared with first-order methods, which are strongly affected by the ill-conditioning of PINN optimization \cite{de2024operator}. 

\begin{figure}[t]
\centering
\begin{tikzcd}[
    column sep=1.9cm,
    row sep=1.8cm,
    every label/.append style={font=\small}
]
\R^{P}
  \arrow[r, "\diff\Pmap(\theta)"]
  \arrow[d, equals]
  \arrow[rrr, bend left=24,
         start anchor=north, end anchor=north,
         "\mathcal G(\theta)"]
  \arrow[rr, bend right=20, "\diff\Res(\theta)"']
&
T_{u_\theta}\M\subset H
  \arrow[r, "\diff\Fmap(u_\theta)"]
&
\Y
  \arrow[r, "\diff\Res(\theta)^{*}"]
  \arrow[d, "\mathrm{ev}_{\X_\D}"]
&
\R^{P}
  \arrow[d, equals]
\\
\R^{P}
  \arrow[rr, "\J_{\rvec}(\theta)"']
  \arrow[rrr, bend right=24,
         start anchor=south, end anchor=south,
         "\DGram(\theta)"']
&
{}
&
\R^{N_\D}
  \arrow[r, "N_\D^{-1}\J_{\rvec}(\theta)^{\top}"']
&
\R^{P}
\end{tikzcd}
\caption{Continuous Gauss--Newton geometry (top row in the diagram) and its collocation
realization (bottom row) obtained by collocation evaluation $\mathrm{ev}_{\X_\D}$ and equal-weight quadrature on the tensor grid $\X_\D$ of cardinality $\mathcal{N}_{_\D}$.}
\label{fig:pullback_collocation}
\end{figure}
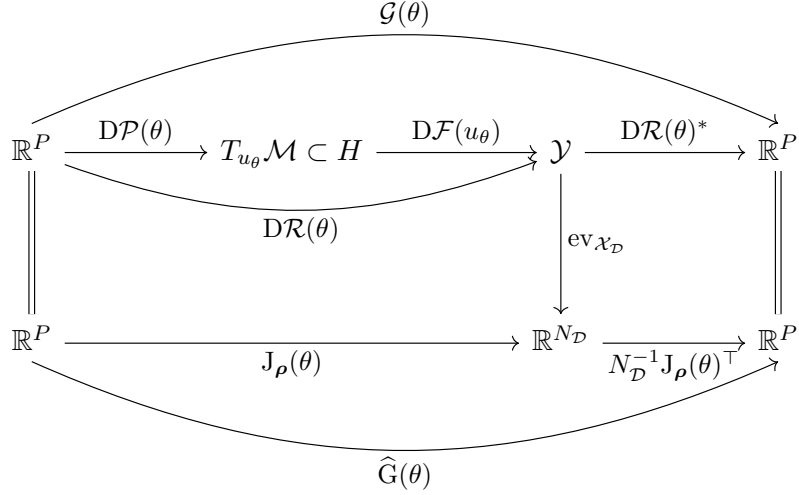

In numerical practice, the continuous gradient $\nabla\mathcal L(\theta)$
and Gramian $\mathcal G(\theta)$ entering the Gauss--Newton step
\eqref{eq:intro_gn_step} are discretized by collocation on a tensor-product
grid $\X_\D$ over $\Omega\times\Xi$, yielding
\begin{equation}\label{eq:intro_collocated_gn}
\nabla\widehat{\mathcal L}(\theta)
=
\frac{1}{N_\D}\mathrm{J}_{\rvec}(\theta)^\top \rvec(\theta),
\qquad
\DGram(\theta)
=
\frac{1}{N_\D}
\mathrm{J}_{\rvec}(\theta)^\top \mathrm{J}_{\rvec}(\theta),
\end{equation}
where $\rvec(\theta)\in\R^{N_\D}$ is the collocated residual vector and
$\mathrm{J}_{\rvec}(\theta)\in\R^{N_\D\times P}$ is its Jacobian,
obtained by collocating the residual push-forward $\diff\Res(\theta)$;
see Fig.~\ref{fig:pullback_collocation}. This yields the discrete damped
Gauss--Newton system to be solved at each optimization iteration:
\begin{equation}\label{eq:intro_damped_gn_system}
\bigl(\DGram(\theta)+\mu I_P\bigr)\delta\theta
=
\nabla\widehat{\mathcal L}(\theta).
\end{equation}
Since $\DGram(\theta)$ is generally dense and of size $P\times P$, solving
\eqref{eq:intro_damped_gn_system} becomes computationally demanding for large
parameter counts, and several approaches have been proposed to reduce this
cost. Kronecker-factored approximate curvature (KFAC) methods
\cite{dangel2025kronecker,dangel2024kronecker,martens2015optimizing}
replace $\DGram(\theta)$ by a Kronecker-factored approximation, which yields
an efficient inversion procedure. This technique applies to rather general
neural architectures and has performed well in second-order PINN optimization
(see, e.g., \cite{dangel2024kronecker}), but it solves an approximate
Gauss--Newton system rather than the exact one, and the corresponding
Kronecker factors are still assembled over the collocation set. Alternatively, Woodbury-type reformulations
\cite{guzman2026improving,jnini2025dual} solve the system \eqref{eq:intro_damped_gn_system} exactly by replacing the $P\times P$ parameter-space problem with an equivalent
sample-space formulation, whose dimension equals the size of the collocation
set. In both cases, the cost thus remains tied to the number of
collocation points. On a tensor-product grid with $n$ points per coordinate,
this number is $N_\D=n^d$ and grows exponentially with the dimension.
One is then left with two limiting options: accept an infeasible
optimization cost or resort to sparse collocation, which may degrade
generalization over $\Omega\times\Xi$ and, with it, the reliability of
online predictions for parametric PDEs.

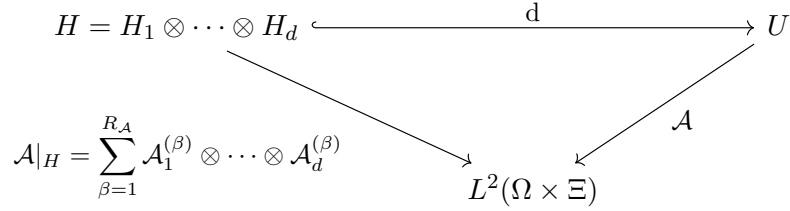
\begin{figure}[H]
\centering
\begin{tikzcd}[
    column sep=1.9cm,
    row sep=1.5cm,
    every label/.append style={font=\small}
]
H=H_1\otimes\cdots\otimes H_d
  \arrow[rr, hook, "\mathrm{d}"]
  \arrow[dr, "\A|_H=\displaystyle\sum_{\beta=1}^{R_\A}\A_1^{(\beta)}\otimes\cdots\otimes\A_d^{(\beta)}"']
& &
U
  \arrow[dl, "\A"]
\\
&
L^2(\Omega\times\Xi)
&
\end{tikzcd}
\caption{Functional setting of the structured PDE formulation. The PDE operator $\A$ is defined on $U$, while its restriction to the densely embedded tensor-product space $H$ admits a separable representation; the diagram commutes under the embedding $H\hookrightarrow U$.}
\label{fig:tensor_operator_structure}
\end{figure}

A classical way to mitigate the curse of dimensionality is to exploit tensor
structure by separating the dependence on individual coordinates
\cite{bachmayr2023low,hackbusch2012tensor}. We adopt this viewpoint for neural
representations of \eqref{eq:parametric_solution_map_intro}. To this end, we use the Hilbert tensor-product space
$H=\bigotimes_{m=1}^{d}H_m$, obtained as the completion of the algebraic
tensor product of the univariate Hilbert spaces $H_m$ with respect to the
canonical Hilbert tensor norm.  We assume that $H$ is continuously and densely embedded in the
solution space $U$, and that the restriction of the differential operator
(or its linearization in the nonlinear case) $\A$ to $H$ can be written as
a sum of $R_{\A}$ tensor products of univariate operators
$\A_m^{(\beta)}$ acting on the corresponding factor spaces $H_m$; see
Fig.~\ref{fig:tensor_operator_structure}.
By the triangle inequality, this setting yields, for any $v,w\in H$, the approximation error decomposition
\begin{align}\label{eq:triangle_approximation}
\lVert u - u_{\theta}\rVert_{U}
\leq
\lVert u - v \rVert_{U}
+
C_{H\hookrightarrow U}\lVert v-w\rVert_{H}
+
C_{H\hookrightarrow U}\lVert w-u_{\theta}\rVert_{H},
\end{align}
where $C_{H\hookrightarrow U}$ is the embedding constant. The first term in \eqref{eq:triangle_approximation} can be made arbitrarily small due to the density of $H$ in $U$
by choosing a suitable $v\in H$. The premise is that $v$ admits an accurate low-rank tensor approximation
$w$ built from univariate factors that reflect the separation of coordinates, keeping the second term small at moderate tensor complexity. These factors are subsequently approximated by neural networks to obtain
$u_\theta$ through the optimization problem
\eqref{eq:parametric_pinn_objective_intro}, while sufficiently expressive univariate networks render the third term in \eqref{eq:triangle_approximation}  small under the universal approximation hypothesis.

As an example of such a structured neural tensor class, consider the rank-$R$
separable ansatz with the parameter decomposition
$\theta=(\theta^1,\ldots,\theta^d)$, where
$\theta^m\in\R^{P_m}$ denotes the trainable parameters of the $m$-th
univariate factor:
\begin{equation}\label{eq:spinn_reduced_form_intro}
u_\theta(x,\xi)
=
\sum_{\alpha=1}^{R}
\phi_\theta^\alpha(x)c_\theta^\alpha(\xi)
=
\sum_{\alpha=1}^{R}
\prod_{m=1}^{d_x}u_m^\alpha(x_m;\theta^m)
\prod_{m=d_x+1}^{d}u_m^\alpha(\xi_{m-d_x};\theta^m).
\end{equation}
Here, the spatial subnetworks learn the modes $\phi_\theta^\alpha$, while
the parametric subnetworks learn the corresponding coefficient maps
$c_\theta^\alpha$. Moreover, with $n$ sample points in each coordinate, the separation in \eqref{eq:spinn_reduced_form_intro} reduces the tensor-grid representation from $n^d$ full-grid values to only $\mathcal O(Rdn)$ factor values.  Physics-informed neural optimization with the ansatz \eqref{eq:spinn_reduced_form_intro} is also known as the
separable PINN approach \cite{cho2023separable} and has shown promise
in several applications \cite{es2025separable,oh2025separable}.
However, its underlying canonical polyadic (CP) format may become
inefficient in high dimensions, since complex coordinate interactions
can require a large global canonical rank $R$.
An alternative is the tensor-train (TT) format
\cite{oseledets2011tensor}:
\begin{equation}\label{eq:tt_intro}
\begin{aligned}
u_\theta(x,\xi)
={}&
G_1(x_1;\theta^1)\cdots
G_{d_x}(x_{d_x};\theta^{d_x})
G_{d_x+1}(\xi_1;\theta^{d_x+1})\cdots
G_d(\xi_{d_\xi};\theta^d).
\end{aligned}
\end{equation}
Here, each $G_m(\,\cdot\,;\theta^m)$ is a univariate matrix-valued core of size $r_{m-1}\times r_m$, with $r_0=r_d=1$.
The cores are coupled through the local TT-ranks $r_1,\ldots,r_{d-1}$, allowing TT to capture coordinate interactions while remaining efficient when these ranks are moderate. Beyond their success in numerical computations in quantum physics, tensor trains have in recent years also been shown to compress high-dimensional parametric PDEs as they appear in the field of Uncertainty Quantification (UQ).
Early works include~\cite{KhoromskijSchwab2011,Nouy2015}.
Subsequently, adaptive stochastic FEM were developed in~\cite{EigelMarschallPfefferSchneider2020, eigel2017adaptive} and nonlinear adaptive weighted least squares approaches in~\cite{EigelFarchminHeidenreichTrunschke2023, EigelSchneider2018}.
A thorough overview of low-rank approximations of PDEs is given in~\cite{bachmayr2023low}; see also~\cite{BachmayrCohenDahmen2017}.
A common UQ benchmark, the affine Darcy problem, is also considered in Section~\ref{sec:numerics}.

\begin{remark} We note that \eqref{eq:spinn_reduced_form_intro} resembles a classical reduced-basis representation \cite{haasdonk2017reduced,quarteroni2015reduced}. This class of methods
constructs spatial modes $\phi_{\mathrm{RB}}^\alpha$ from high-fidelity
solution snapshots $u_h(\cdot,\xi^{(j)})$ at selected parameter values
and approximate
\[
u(x,\xi) \approx u_{\mathrm{RB}}(x,\xi)
=
\sum_{\alpha=1}^{R}
\phi_{\mathrm{RB}}^\alpha(x)c_{\mathrm{RB}}^\alpha(\xi).
\]
For each $\xi \in \Xi$, the coefficients $c_{\mathrm{RB}}^1(\xi),\ldots,c_{\mathrm{RB}}^R(\xi)
$ are then determined by solving a reduced algebraic system, which is obtained by projecting the PDE onto the linear space spanned by the snapshots. This requires sampling the parameter domain to identify representative
snapshot locations. As the dimension of $\Xi$ grows, sufficiently exploring
the parameter space becomes increasingly difficult, even with adaptive or
greedy sampling. Moreover, each selected parameter value requires a
high-fidelity PDE solve, which can make the offline stage expensive. In contrast, \eqref{eq:spinn_reduced_form_intro} treats $\xi$ as additional input variables and learns both the spatial modes $\phi_\theta^\alpha$ and the coefficient maps $c_\theta^\alpha$ jointly through PINN optimization, without requiring solution snapshots or a reduced solve for each parameter. However, both approaches rely on the same underlying \emph{manifold hypothesis}: sufficient regularity of the parameter-to-solution map with respect to $\xi$ yields a solution set
$\mathcal M_{\mathrm{sol}}=\{u(\cdot,\xi):\xi\in\Xi\}\subset V$
that admits an efficient low-rank approximation.
\end{remark}

The key structural observation of this work is that the coordinate-separable representations
\eqref{eq:spinn_reduced_form_intro} and \eqref{eq:tt_intro}
induce structure not only in the solution ansatz but also in its
derivatives with respect to the neural network parameters.
To exploit this structure in Gauss--Newton optimization, the PDE operator $\mathcal{A}$
(or its linearization in the nonlinear case) must also be separable with
respect to the same coordinates, as outlined above. The connection between the ansatz and the Gauss--Newton Gramian then follows from three structural properties:

\begin{enumerate}[label=(S\arabic*),leftmargin=*,itemsep=1pt,topsep=2pt]
\item\label{it:S1}
the architecture represents $u_\theta$ through a fixed multilinear
contraction of univariate factors;

\item\label{it:S2}
separability survives linearization, so that the parametrization push-forward $\diff\Pmap(\theta)$
is assembled from univariate quantities;

\item\label{it:S3}
the differential operator $\mathcal{A}$ in the linearized residual $\diff\Res(\theta)$
preserves separability, so the residual push-forward remains separable and its Gauss--Newton Gramian can be assembled by low-dimensional contractions.
\end{enumerate}

\definecolor{fzA}{RGB}{0,114,178}
\definecolor{bloodyred}{RGB}{150,0,0}

\tikzset{
  fz frame/.style={draw=black!70, line width=0.8pt},
  fz sep/.style={draw=black!40, line width=0.4pt},
  fz lbl/.style={font=\normalsize},
  fz in/.style={font=\small},
  fz note/.style={font=\footnotesize, text=bloodyred},
  fz brace/.style={
    decorate,
    decoration={brace, amplitude=4pt, raise=1.5pt},
    draw=black!50,
    line width=0.5pt
  },
}

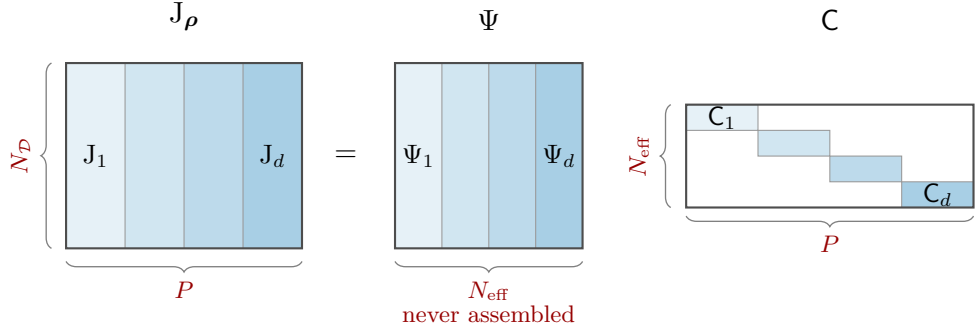
\begin{figure}
\centering
\begin{tikzpicture}[x=1cm, y=1cm]

\def\H{2.45}

\begin{scope}[shift={(0,0)}]

\def\JX{0}
\def\JW{0.78}

\foreach \k/\tint in {1/10,2/18,3/26,4/34}{
  \fill[fzA!\tint]
    (\JX+\JW*\k-\JW,0)
    rectangle
    (\JX+\JW*\k,\H);}
\foreach \k in {1,2,3}{
  \draw[fz sep]
    (\JX+\JW*\k,0)
    --
    (\JX+\JW*\k,\H);}
\draw[fz frame] (\JX,0) rectangle (\JX+4*\JW,\H);

\node[fz in] at (\JX+0.5*\JW,0.5*\H) {$\mathrm J_1$};
\node[fz in] at (\JX+3.5*\JW,0.5*\H) {$\mathrm J_d$};
\node[fz lbl, anchor=south] at (\JX+2*\JW,\H+0.30) {$\mathrm J_{\rvec}$};

\draw[fz brace] (\JX+4*\JW,-0.12) -- (\JX,-0.12);
\node[fz note, anchor=north] at (\JX+2*\JW,-0.30) {$P$};

\draw[fz brace] (\JX-0.14,0) -- (\JX-0.14,\H);
\node[fz note, rotate=90, anchor=south] at (\JX-0.34,0.5*\H) {$\ND$};

\node[fz lbl] at (3.70,0.5*\H) {$=$};

\def\PsiX{4.35}
\def\PsiW{0.62}

\foreach \k/\tint in {1/10,2/18,3/26,4/34}{
  \fill[fzA!\tint]
    (\PsiX+\PsiW*\k-\PsiW,0)
    rectangle
    (\PsiX+\PsiW*\k,\H);}
\foreach \k in {1,2,3}{
  \draw[fz sep]
    (\PsiX+\PsiW*\k,0)
    --
    (\PsiX+\PsiW*\k,\H);}
\draw[fz frame]
  (\PsiX,0)
  rectangle
  (\PsiX+4*\PsiW,\H);

\node[fz in] at (\PsiX+0.5*\PsiW,0.5*\H) {$\Psi_1$};
\node[fz in] at (\PsiX+3.5*\PsiW,0.5*\H) {$\Psi_d$};
\node[fz lbl, anchor=south] at (\PsiX+2*\PsiW,\H+0.30) {$\Psi$};

\draw[fz brace]
  (\PsiX+4*\PsiW,-0.12)
  --
  (\PsiX,-0.12);

\node[fz note, anchor=north, align=center] at
  (\PsiX+2*\PsiW,-0.30)
  {$\Neff$\\[-2pt] never assembled};

\def\CX{8.20}
\def\CW{0.95}
\def\Ct{1.90}
\def\CH{0.34}

\foreach \k/\tint in {1/10,2/18,3/26,4/34}{
  \fill[fzA!\tint]
    (\CX+\CW*\k-\CW,\Ct-\CH*\k)
    rectangle
    (\CX+\CW*\k,\Ct-\CH*\k+\CH);
  \draw[fz sep]
    (\CX+\CW*\k-\CW,\Ct-\CH*\k)
    rectangle
    (\CX+\CW*\k,\Ct-\CH*\k+\CH);}

\draw[fz frame]
  (\CX,\Ct-4*\CH)
  rectangle
  (\CX+4*\CW,\Ct);

\node[fz in] at (\CX+0.5*\CW,\Ct-0.17) {$\Cmat_1$};
\node[fz in] at (\CX+3.5*\CW,\Ct-1.19) {$\Cmat_d$};
\node[fz lbl, anchor=south] at (\CX+2*\CW,\H+0.30) {$\Cmat$};

\draw[fz brace]
  (\CX+4*\CW,\Ct-1.42)
  --
  (\CX,\Ct-1.42);

\node[fz note, anchor=north] at
  (\CX+2*\CW,\Ct-1.60)
  {$P$};

\draw[fz brace]
  (\CX-0.15,\Ct-4*\CH)
  --
  (\CX-0.15,\Ct);

\node[fz note, rotate=90, anchor=south] at
  (\CX-0.36,\Ct-2*\CH)
  {$\Neff$};


\end{scope}

\end{tikzpicture}

\caption{Structure exploited by the compressed Gauss--Newton step. The residual Jacobian admits a coordinate-wise factorization in which the local parameter derivatives are collected in the block-diagonal factor $\Cmat$, while $\Psi$ contains the remaining factor evaluations and is never assembled explicitly.}
\label{fig:intro_factorization}
\end{figure}

Furthermore, the collocation discretization does not destroy separability but
rather inherits it from the residual push-forward. For both CP and TT
architectures, this yields the exact factorization of the residual Jacobian
\begin{equation}\label{eq:intro_J_factorization}
\mathrm{J}_{\rvec}(\theta)
=
\Psi\Cmat,
\qquad
\Psi\in\R^{N_\D\times\Neff},
\qquad
\Cmat\in\R^{\Neff\times P},
\end{equation}
where $\Neff$ is what we call the \textit{effective compressed dimension}. This
factorization comes from the coordinate-wise structure of the residual
push-forward: differentiating $\Res(\theta)$ with respect to a parameter block $\theta^s$
acts only on the corresponding $s$-th univariate factor. The resulting derivative evaluations at the collocation points are collected in $\Cmat_s$, while $\Psi_s$ contains the remaining separable factor evaluations on the collocation grid. This yields $\mathrm J_s=\Psi_s\Cmat_s$ for each coordinate, and stacking these blocks
gives \eqref{eq:intro_J_factorization}, as illustrated in Fig.~\ref{fig:intro_factorization}. Consequently, the
discrete Gauss--Newton matrix admits the exact factorization
\begin{equation}\label{eq:intro_G_factorization}
\DGram(\theta)
=
\Cmat^\top\Bgram\Cmat,
\qquad
\Bgram
:=
\frac{1}{N_\D}\Psi^\top\Psi
\in\R^{\Neff\times\Neff},
\end{equation}
and, using the push-through identity, the solution of the Gauss--Newton system
\eqref{eq:intro_damped_gn_system} can be obtained from the compressed linear
system
\begin{equation}\label{eq:intro_compressed_gn_system}
\bigl(\Bgram\Cmat\Cmat^\top+\mu I_{\Neff}\bigr)y
=
\frac{1}{N_\D}\Psi^\top\rvec(\theta),
\qquad
\delta\theta=\Cmat^\top y.
\end{equation}
The solve in parameter space is thus replaced by a solve in the
$\Neff$-dimensional compressed residual space, which in the regimes of interest
is substantially smaller. Neither $\Psi$ nor the tensor-product grid of size
$N_\D$ is ever formed: the compressed quantities $\Bgram$ and
$\Psi^\top\rvec(\theta)$ are assembled by tensor contractions over the univariate
grids of sizes $n_s$, $s=1,\ldots,d$.

What makes our compressed Gauss--Newton strategy particularly attractive in high dimensions is that the effective compressed dimension grows additively across coordinates
\begin{equation}\label{eq:intro_neff_slotwise}
\Neff=\sum_{s=1}^{d}N_s,
\qquad
N_s=\RA\,k_s\,n_s,
\end{equation}
where $n_s$ is the univariate grid size and $k_s$ counts the number of features of the $s$-th univariate factor: $k_s=R$ for a rank-$R$ separable PINN and
$k_s=r_{s-1}r_s$ for a TT-PINN, which gives
\begin{equation}\label{eq:intro_neff_cp_tt}
\Neff^{\mathrm{CP}}
=
R \RA\sum_{s=1}^{d}n_s,
\qquad
\Neff^{\mathrm{TT}}
=
\RA\sum_{s=1}^{d}r_{s-1}r_s n_s.
\end{equation}
This has three immediate consequences:
\begin{itemize}
\item First, $\Neff$ is independent of the number
of trainable parameters used to represent the univariate factors. Thus, the
underlying networks may be widened or deepened without enlarging the
compressed linear system; consequently, the ratio $P/\Neff$, and hence the
benefit of the compressed solve, increases with the expressivity of the
univariate parametrizations.
\item Second, when the univariate grid sizes, local ranks, and operator separation rank are kept fixed as the dimension increases, $\Neff$ grows linearly in $d$, in sharp contrast to the tensor-grid size $N_\D=\prod_s n_s$.
\item Third, the local nature of $k_s$ is what distinguishes the two formats: TT is particularly advantageous here when the solution requires a large global CP rank but admits moderate local TT ranks.
\end{itemize}
Moreover, we exploit additional structure shared by the univariate operators in
the separable expansion: in each coordinate, many operator terms reuse the same
small set of univariate differential operations, such as the identity,
first derivative, or second derivative, differing only in their coefficient
functions. Grouping these common local differential channels replaces the
factor $\RA$ in \eqref{eq:intro_neff_slotwise} by the number $L_s$ of distinct
channels, which for second-order operators is at most three per coordinate and
independent of how many separable terms the operator has. We show that this
structure is available for a broad class of finite-order differential operators
with separable coefficients. Under this compression, our numerical experiments
demonstrate high accuracy in approximating the parameter-to-solution map
\eqref{eq:parametric_solution_map_intro} at very low computational cost.

Summarizing, the main contributions of this work are as follows:
\begin{itemize}[leftmargin=*,itemsep=2pt,topsep=3pt]

\item \textbf{Abstract framework.}
We develop a structured function-space framework for low-rank approximation of parametric PDEs using neural networks with separable structure. Keeping the tensor contraction pattern of the approximation class general, we formulate separability assumptions on the architecture and PDE operator and derive the separable structure of the continuous Gauss--Newton Gramian and loss gradient from the residual push-forward. We further show that tensor-product collocation preserves this structure, yielding exact algebraic factorizations of the residual Jacobian, Gramian, and loss gradient, as well as an efficient compressed formulation of the Gauss--Newton step.

\item \textbf{Algebraic characterization of the compressed Gauss--Newton system.}
We derive explicit representations of the Jacobian, compressed Gramian, and loss gradient required in practice for both CP and TT formats, characterize their effective compressed dimensions, and develop efficient assembly procedures.

\item \textbf{Applications to high-dimensional PDEs.}
We apply the developed framework to high-dimensional (parametric) PDEs, explicitly identifying the corresponding function spaces, separable operator structure, and compressed matrices required for computation. The numerical experiments show that the resulting tensor-train compressed Gauss--Newton method achieves high accuracy on these problems within seconds.
\end{itemize}

The rest of the paper is organized as follows.
Section~\ref{sec:formulation} develops the abstract function-space framework
and derives the structured Gauss--Newton factorization for a general contraction pattern.
Section~\ref{sec:woodbury} introduces the exact compressed Gauss--Newton solver. Section~\ref{sec:architectures} specializes the framework to two specific
contraction patterns corresponding to the CP and TT formats and derives their
algebraic characterization.
Section~\ref{sec:numerics} presents the model problems, including the parametric Darcy equations with coefficients of varying parametric complexity, together with their analysis and the corresponding numerical experiments.
Section~\ref{sec:conclusion} concludes with a summary and an outlook.

\section{Mathematical formulation}
\label{sec:formulation}

A point of the extended domain
$\D=\Omega\times\Xi$ is written as $z=(x,\xi)\in\D$ with $x\in\Omega$ and
$\xi\in\Xi$, and its scalar components are numbered consecutively:
\begin{equation}\label{eq:coordinates}
z=(z_1,\dots,z_d),
\qquad
z_m=x_m\ \ (m\le\dsp),
\qquad
z_{\dsp+j}=\xi_j\ \ (j\le\dpa).
\end{equation}
We refer to $m\le\dsp$ as \emph{spatial} slots and to $m>\dsp$ as \emph{parametric} slots. All statements in this work are formulated in terms of the running index $m$; where the distinction between spatial and parametric slots matters, we state it explicitly.

A linear map $\Vq:\R^{q}\to H$ with finite-dimensional domain and values in a Hilbert space $H$ is represented as a \emph{quasi-matrix}, whose columns are functions in $H$:
\begin{equation}\label{eq:quasimatrix}
\Vq=\begin{bmatrix}\varphi_1&\cdots&\varphi_q\end{bmatrix},
\qquad
\varphi_i=\Vq e_i\in H,
\qquad
\Vq v=\sum_{i=1}^{q}v_i\varphi_i .
\end{equation}
Its Hilbert adjoint $\Vq^{*}:H\to\R^{q}$ and the products with another quasi-matrix
$\W=[\psi_1\ \cdots\ \psi_p]$ or with a single element $g\in H$ are then given by
\begin{equation}\label{eq:quasimatrix_adjoint}
\bigl(\Vq^{*}g\bigr)_i=\ip{\varphi_i}{g}_H\in\R,
\qquad
\bigl(\Vq^{*}\W\bigr)_{ij}=\ip{\varphi_i}{\psi_j}_H\in\R,
\qquad
\bigl(g^{*}\W\bigr)_{j}=\ip{g}{\psi_j}_H .
\end{equation}
Thus $\Vq^{*}g\in\R^{q}$ is a column vector, $g^{*}\W\in\R^{1\times p}$ is a row
vector, and $\Vq^{*}\W\in\R^{q\times p}$ is a matrix.  Quasi-matrices obey the
usual rules of block matrix algebra with functions in place of columns.

Throughout, we adopt the notational conventions that are standard in the tensor-train literature; cf. \cite{corona2017tensor, dolgov2014alternating, oseledets2011tensor}. A tuple of indices 
\[
(i_1,\ldots,i_k),
\qquad
i_\ell=1,\ldots,n_\ell,
\] 
is referred to as a \emph{multi-index}. Tensor entries are accessed by parentheses rather than by subscripts, while the subscripts remain available for enumerating objects. Thus,
$\mathcal{H}(i_1,\ldots,i_d)$ is an entry of the tensor $\mathcal H$ and
$\mathsf A(i,a)$ an entry of the matrix $\mathsf A$. A colon denotes a full
index range, as in $\mathsf A(:,a)$, and $\mathcal H(:,\ldots,:,a)$ is the
corresponding tensor slice. For matrices (or vectors) $\mathsf A_1,\ldots,\mathsf A_q$ with the same
number of columns, we denote their vertical concatenation by
\begin{equation}\label{eq:col_notation}
\operatorname{col}\bigl(
\mathsf A_1,\ldots,\mathsf A_q
\bigr)
:=
\begin{bmatrix}
\mathsf A_1\\
\vdots\\
\mathsf A_q
\end{bmatrix},
\end{equation}
but sometimes we keep the vertical concatenation explicit. Groups of indices that are merged into a single row
or column index of a matrix are written with an overbar,
\begin{equation}\label{eq:multiindex}
\overline{i_1i_2\ldots i_k}
:=
i_k+(i_{k-1}-1)n_k+\cdots+(i_1-1)n_2n_3\cdots n_k
\in\{1,\ldots,n_1n_2\cdots n_k\},
\end{equation}
corresponding to the position of $(i_1,\ldots,i_k)$ in lexicographic order, with the
leftmost index varying slowest. We write $\vecop$ for the vectorization associated with
\eqref{eq:multiindex}, so that $\vecop(\mathsf A)(\overline{ij})=\mathsf A(i,j)$. The ordering \eqref{eq:multiindex} is compatible with the Kronecker product:
\begin{equation}\label{eq:kron_convention}
(\mathsf A\otimes\mathsf B)
\bigl(\overline{i_1i_2},\overline{j_1j_2}\bigr)
=
\mathsf A(i_1,j_1)\,\mathsf B(i_2,j_2).
\end{equation}
A colon inside a merged index selects a whole block: for a matrix with row index $\overline{i_sa_s}$ and column index $\overline{i_ta_t}$, we write $\mathsf A(\overline{i_s:},\overline{i_t:})$ for its $(i_s,i_t)$ block. 

Finally, $L^p(\Omega)$, $H^1(\Omega)$, $H_0^1(\Omega)$, $H^k(\Omega)$, $W^{k,p}(\Omega)$, etc., denote the standard Lebesgue and Sobolev spaces; see, e.g.,~\cite{adams2003sobolev}.

\subsection{Separable parametric PDEs}
\label{sub:assumptions}

Let $\Omega\subset\R^{\dsp}$ be a bounded spatial domain with Lipschitz boundary
$\partial\Omega$, and let $\Xi\subset\R^{\dpa}$ be the parameter set. Let $V$ be a Hilbert space on $\Omega$ for which the spatial trace operator $\tr:V\to L^{2}(\partial\Omega)$ is bounded. For every  $\xi\in\Xi$, we consider the boundary value problem
\begin{align}\label{eq:parametric_pde}
\A(\xi)u(\cdot,\xi)=f(\cdot,\xi)
\quad\text{in }\Omega, \qquad
u(\cdot,\xi)=g(\cdot,\xi)
\quad\text{on }\partial\Omega.
\end{align}
We assume that \eqref{eq:parametric_pde} is well-posed for every $\xi\in\Xi$ with solution $u(\cdot,\xi)\in V$. The family of solutions defines the \emph{parameter-to-solution map}
\begin{align}\label{eq:parameter_to_solution_map}
u:\Xi\to V,
\qquad
\xi\mapsto u(\cdot,\xi).
\end{align}
We regard $u$ as an element of the Hilbert--Bochner space $U: = L^{2}(\Xi;V)$, which is defined as the space of Bochner-measurable maps $v:\Xi\to V$ satisfying
\begin{align}\label{eq:bochner_norm}
\norm{v}_{U}^{2}
:=
\int_{\Xi}\norm{v(\cdot,\xi)}_{V}^{2}\,\mathrm{d}\xi
<\infty.
\end{align}
We assume that the family $\{\A(\xi)\}_{\xi\in\Xi}$ induces the bounded space--parameter operator
\begin{align*}
\A: L^{2}(\Xi;V)\to L^{2}\bigl(\Xi;L^{2}(\Omega)\bigr),
\qquad
(\A v)(\cdot,\xi):=\A(\xi)v(\cdot,\xi).
\end{align*}
Likewise, the family $\{\tr(\xi)\}_{\xi\in\Xi}$ induces the bounded space--parameter trace operator
\begin{align*}
\tr: L^{2}(\Xi;V)\to L^{2}\bigl(\Xi;L^{2}(\partial\Omega)\bigr),
\qquad
(\tr v)(\cdot,\xi):=\tr\bigl(v(\cdot,\xi)\bigr).
\end{align*}
For the interior and boundary residuals, we introduce the Hilbert space
\begin{align*}
\mathcal Y
:=
L^{2}\bigl(\Xi;L^{2}(\Omega)\bigr)
\times
L^{2}\bigl(\Xi;L^{2}(\partial\Omega)\bigr),
\end{align*}
which is equipped with the canonical product inner product. 

Within the framework of the extended coordinates \eqref{eq:coordinates}, it is convenient to treat the spatial and parametric variables on the same footing. For $\D=\Omega\times\Xi$ and the lateral boundary
$\Lat:=\partial\Omega\times\Xi$, we employ the standard isometric
identifications
\begin{align*}
L^{2}\bigl(\Xi;L^{2}(\Omega)\bigr)\cong L^{2}(\D),
\qquad
L^{2}\bigl(\Xi;L^{2}(\partial\Omega)\bigr)\cong L^{2}(\Lat).
\end{align*}
The corresponding Bochner norms induce the following norms on $L^{2}(\D)$ and $L^{2}(\Lat)$:
\begin{align*}
\norm{v}_{L^{2}(\D)}^{2}
=\int_{\Xi}\!\int_{\Omega}|v(x,\xi)|^{2}\,\mathrm{d}x\,\mathrm{d}\xi,
\qquad
\norm{w}_{L^{2}(\Lat)}^{2}
=\int_{\Xi}\!\int_{\partial\Omega}|w(x,\xi)|^{2}\,\mathrm{d}x\,\mathrm{d}\xi,
\end{align*}
so that $\mathcal Y\cong L^{2}(\D)\times L^{2}(\Lat)$ with the norm
\begin{align}\label{eq:Y_norm}
\norm{(v,w)}_{\mathcal Y}^{2}
=\norm{v}_{L^{2}(\D)}^{2}+\norm{w}_{L^{2}(\Lat)}^{2}.
\end{align}
Within this setting, finding $u\in U$ corresponds to solving the problem
\begin{align}\label{eq:bochner_pde}
\A u=f \quad\text{in } L^{2}(\D),
\qquad
\tr u=g \quad\text{in } L^{2}(\Lat).
\end{align}
Thus, the entire parametric family \eqref{eq:parametric_pde} is represented by a single problem on $\mathcal{D}$. 

We proceed by stating the structural assumptions on the problem \eqref{eq:bochner_pde}. The first concerns the geometry of the extended domain.
\begin{assumption}[Separable domain]\label{ass:domain}
The spatial domain and the parameter set are Cartesian products of bounded
intervals:
\begin{equation}\label{eq:factor_domains}
\Omega=\Omega_1\times\cdots\times\Omega_{\dsp},
\qquad
\Xi=\Xi_1\times\cdots\times\Xi_{\dpa}.
\end{equation}
The extended domain factorizes accordingly:
\begin{equation}\label{eq:product_domain}
\D=\D_1\times\cdots\times\D_d\subset\R^{d},
\qquad
\D_m=(z_m^{-},z_m^{+})\subset\R,
\qquad m=1,\dots,d,
\end{equation}
where $\D_m=\Omega_m$ for $m\le\dsp$ and $\D_{\dsp+j}=\Xi_j$ for $j\le\dpa$.
\end{assumption}
\noindent
The product structure of $\D$ in Assumption~\ref{ass:domain} yields the
canonical Hilbert space identification
\begin{align}\label{eq:hilbert_tensor}
L^{2}(\D)
\cong
L^{2}(\D_1)\,\otimes\,\cdots\,\otimes\,L^{2}(\D_d)
\end{align}
where $\otimes$ is the Hilbert tensor product. Under this canonical identification, an elementary tensor 
\begin{align*}
v_1\otimes\cdots\otimes v_d, \quad v_m\in L^2(\D_m)
\end{align*}
corresponds to the separable (or the so-called rank-one) function
\begin{align}\label{eq: separable function}
v(z_1,\ldots,z_d)=\prod_{m=1}^{d}v_m(z_m)
\end{align}
Thanks to Fubini's theorem, the inner product of two elementary tensors factorizes into univariate ones
\begin{equation}\label{eq:crossnorm}
\ip{v_1\otimes\cdots\otimes v_d}{w_1\otimes\cdots\otimes w_d}_{L^{2}(\D)}
=
\prod_{m=1}^{d}\ip{v_m}{w_m}_{L^{2}(\D_m)}.
\end{equation}
The second structural hypothesis concerns the tensor-product functional setting
underlying our approximation scheme and the PDE operator.
\begin{assumption}[Separable differential operator]
\label{ass:operator}
There exist univariate Hilbert spaces
$H_m\hookrightarrow L^2(\D_m)$, $m=1,\ldots,d$, such that the following properties hold:
\begin{enumerate}

\item The Hilbert tensor product $H:=H_1\otimes\cdots\otimes H_d$ with the norm $\|v\|_H:=\sqrt{\langle v,v\rangle_H}$ satisfies the crossnorm property
\begin{align}\label{eq:H_crossnorm}
\left\|
v_1\otimes\cdots\otimes v_d
\right\|_H
=
\prod_{m=1}^{d}\|v_m\|_{H_m}.
\end{align}

\item The space $H$ is continuously and densely embedded in $U$, i.e.,
\begin{align}\label{eq:H_dense_embedding}
H\overset{\mathrm d}{\hookrightarrow} U,
\qquad
\overline{H}^{\,\|\cdot\|_U}=U.
\end{align}

\item The restriction of $\A$ to $H$ admits the separable representation
\begin{align}\label{eq:separable_operator}
\left.\A\right|_H
=
\sum_{\beta=1}^{R_\A}
\A_1^{(\beta)}\otimes\cdots\otimes\A_d^{(\beta)},
\end{align}
where $R_\A\in\mathbb N$ and
$\A_m^{(\beta)}:H_m\to L^2(\D_m)$ are bounded linear operators.
\end{enumerate}
With a slight abuse of notation, we continue to denote this restriction by
$\A$ whenever it is clear from the context.
\end{assumption}
\noindent
The consequence of Assumption \ref{ass:operator} is that $\A:H \rightarrow L^{2}(\mathcal{D})$ maps separable functions to finite sums of separable functions. Indeed, by the definition of the tensor product of operators, we get
\begin{align}\label{eq: separable operator}
\bigl(\A_1^{(\beta)}\otimes\cdots\otimes\A_d^{(\beta)}\bigr)
(v_1\otimes\cdots\otimes v_d)=(\A_1^{(\beta)}v_1)\otimes\cdots\otimes(\A_d^{(\beta)}v_d)
\end{align}
so that for a rank-one function \eqref{eq: separable function} we obtain
\begin{equation}\label{eq:operator_action}
\A v(z_1,\dots,z_d)
=
\sum_{\beta=1}^{R_{\A}}\prod_{m=1}^{d}\bigl(\A_m^{(\beta)}v_m\bigr)(z_m).
\end{equation}
Therefore, $\A$ increases the separation rank of its argument by at most the factor
$R_{\A}$ and preserves the tensor product structure.

Many differential operators admit a representation in terms of a small common family of (local) differential operators. As a matter of highlighting this (very useful) fact, the following assumption imposes additional local structure on the
univariate operator factors in Assumption~\ref{ass:operator}.
\begin{assumption}[Local differential-channel representation]
\label{ass:local_differential_channels}
Let Assumption~\ref{ass:operator} hold. For every $s=1,\ldots,d$, the
univariate operator family
$\{\A_s^{(\beta)}\}_{\beta=1}^{R_{\A}}$ admits a common representation
\begin{equation}\label{eq:local_differential_channel_representation}
\A_s^{(\beta)}v
=
\sum_{\ell\in\Lambda_s}
b_{s,\ell}^{(\beta)}\,\mathcal D_{s}^{\ell}v,
\qquad
v\in H_s,
\qquad
\beta=1,\ldots,R_{\A},
\end{equation}
where $\Lambda_s$ is a finite index set with $L_s:=|\Lambda_s|$,
$\mathcal D_{s}^{\ell}:H_s\to L^2(\D_s)$ are bounded linear differential
operators, and
$b_{s,\ell}^{(\beta)}\in C(\overline{\D_s})$ are coefficient functions.
\end{assumption}

It is now worth taking a moment to discuss the choice of the univariate spaces $H_s$ in
Assumption~\ref{ass:operator}, particularly in view of
\eqref{eq:local_differential_channel_representation}. Observe that
multiplication by
$b_{s,\ell}^{(\beta)}\in C(\overline{\D_s})\subset L^\infty(\D_s)$
is bounded on $L^2(\D_s)$, so the regularity of $H_s$ is determined by
the differential channels $\mathcal D_s^\ell$. In many applications,
including our parametric examples, these channels are
$\mathcal D_s^\ell=\partial_{z_s}^\ell$ with
\begin{align*}
\Lambda_s\subset\{0,\ldots,q\}\quad(s\le\dsp),
\qquad
\Lambda_s=\{0\}\quad(s>\dsp).
\end{align*}
Accordingly, we may take $H_s=H^q(\Omega_s)$ in the spatial directions
and $H_s=L^2(\D_s)$ in the parametric directions. Tensorization of the
spatial spaces then yields the Sobolev space of dominating mixed smoothness
\cite{hackbusch2012tensor,hochmuth2000tensor}:
\begin{align*}
\bigotimes_{m=1}^{\dsp} H^q(\Omega_m)
\cong
H_{\mathrm{mix}}^q(\Omega)
:=
\left\{
v\in L^2(\Omega):
D^\alpha v\in L^2(\Omega)
\ \text{for all } \alpha\in\mathbb N_0^{\dsp}
\text{ with } \|\alpha\|_\infty\leq q
\right\},
\end{align*}
equipped with the norm
\begin{align}\label{eq:mixed_norm}
\|v\|_{H_{\mathrm{mix}}^q(\Omega)}^2
:=
\sum_{\|\alpha\|_\infty\leq q}
\|D^\alpha v\|_{L^2(\Omega)}^2.
\end{align}
Accordingly, in the case of $ u(\cdot, \xi) \in V=H^q(\Omega)$ with $\xi \in \Xi$, a natural tensor space is given by
\begin{align}\label{eq:tensor_space_mixed}
H=
\left(\bigotimes_{m=1}^{\dsp}H^q(\Omega_m)\right)
\otimes
\left(\bigotimes_{j=1}^{\dpa}L^2(\Xi_j)\right)\cong
L^2\bigl(\Xi;H_{\mathrm{mix}}^q(\Omega)\bigr).
\end{align} 
Observe that the mixed norm \eqref{eq:mixed_norm} is stronger than the standard $H^q(\Omega)$
norm, as every derivative entering the $H^q(\Omega)$ norm also appears in the
$H_{\mathrm{mix}}^q(\Omega)$ norm, while the latter additionally contains
mixed derivatives of higher total order. As a result, we obtain the continuous embedding
$H_{\mathrm{mix}}^q(\Omega)\hookrightarrow H^q(\Omega)$, i.e., $\lVert v \rVert_{H^q(\Omega)} \leq C_{\mathrm{mix}}\lVert v \rVert_{H_{\mathrm{mix}}^q(\Omega)}$ for all $ v\in H_{\mathrm{mix}}^q(\Omega)$. 

It remains to establish the corresponding embedding of the Bochner spaces in our setting. From the embedding $H_{\mathrm{mix}}^q(\Omega)\hookrightarrow H^q(\Omega)$ we immediately obtain the estimate
\begin{align*}
\|v\|_{L^2(\Xi;H^q(\Omega))}^2
=
\int_\Xi
\|v(\xi)\|_{H^q(\Omega)}^2\,\mathrm d\xi
\leq
C_{\mathrm{mix}}^2
\int_\Xi
\|v(\xi)\|_{H_{\mathrm{mix}}^q(\Omega)}^2\,\mathrm d\xi
=
C_{\mathrm{mix}}^2
\|v\|_{L^2(\Xi;H_{\mathrm{mix}}^q(\Omega))}^2
\end{align*}
for all $v \in L^{2}(\Xi, H_{\mathrm{mix}}^q(\Omega))$, which then renders the embedding 
\begin{align*}
L^2\big(\Xi;H_{\mathrm{mix}}^q(\Omega)\big)\hookrightarrow L^2
\big(\Xi;H^q(\Omega)\big)    
\end{align*}
continuous, as well. Consequently, for $V=H^q(\Omega)$ and
$U=L^2(\Xi;V)$, the tensor space \eqref{eq:tensor_space_mixed} is
continuously embedded in $U$. For the product domains in Assumption~\ref{ass:domain},
$H_{\mathrm{mix}}^q(\Omega)$ is dense in $H^q(\Omega)$, since
$C^\infty(\overline{\Omega})\subset H_{\mathrm{mix}}^q(\Omega)$ and is
dense in $H^q(\Omega)$. Therefore, the embedding $H\hookrightarrow U$
is also dense, as required by Assumption~\ref{ass:operator}. Observe that if the parameter-to-solution map
\eqref{eq:parameter_to_solution_map} possesses the corresponding spatial mixed
regularity, then restricting the approximation to $H$ introduces no additional
regularity barrier, makes the tensor-product structure directly available, and
eliminates the first term in \eqref{eq:triangle_approximation}.

As a matter of convenience, we use the following proposition to illustrate the dominating mixed smoothness property while also verifying Assumption~\ref{ass:operator} for the Poisson problem.
\begin{proposition}
\label{prop:poisson_separable_structure}
Let Assumption~\ref{ass:domain} hold with $d:=\dsp$ and $\D:=\Omega$. Consider
\begin{align}\label{eq:poisson_operator_abstract}
\A:U\to L^2(\Omega),
\qquad
\A v:=-\Delta_x v,
\qquad
U:=H^2(\Omega),
\qquad
H:=\bigotimes_{m=1}^{d}H_m,
\end{align}
where $H_m:=H^2(\Omega_m)$, and $U$ and $H$ are equipped with the norms
\begin{align}\label{eq:poisson_U_H_norms}
\|v\|_U^2
&:=
\sum_{|\alpha|_1\leq 2}
\|D^\alpha v\|_{L^2(\Omega)}^2,
\qquad
\|v\|_H^2
:=
\sum_{\alpha\in\{0,1,2\}^d}
\|D^\alpha v\|_{L^2(\Omega)}^2.
\end{align}
Then Assumption~\ref{ass:operator} is satisfied, and $R_\A=d$. In particular,
\begin{align}\label{eq:poisson_operator_separable}
\left.\A\right|_H
=
-\sum_{\beta=1}^{d}
I^{\otimes(\beta-1)}
\otimes
\partial_{x_\beta}^2
\otimes
I^{\otimes(d-\beta)}.
\end{align}
where $I_m:H^2(\Omega_m)\hookrightarrow L^2(\Omega_m)$
denotes the canonical embedding.
\end{proposition}
\begin{proof}
We verify the three properties in the order of
Assumption~\ref{ass:operator}.

\textup{(i)} Let $v=v_1\otimes\cdots\otimes v_d$ with $v_m\in H^2(\Omega_m)$ be an elementary tensor. For every multi-index $\alpha\in\mathbb N_0^d$ with
$\|\alpha\|_\infty\leq2$, equivalently
$\alpha\in\{0,1,2\}^d$, under the canonical identification we obtain
\begin{align*}
D^\alpha v(x)
=
\prod_{m=1}^{d}
\partial_{x_m}^{\alpha_m}v_m(x_m).
\end{align*}
By Fubini's theorem,
\begin{align}\label{eq:Fubini_mixed_regularity}
\|D^\alpha v\|_{L^2(\Omega)}^2
=
\int_{\Omega}
\prod_{m=1}^{d}
\left|
\partial_{x_m}^{\alpha_m}v_m(x_m)
\right|^2
\,\mathrm dx=
\prod_{m=1}^{d}
\left\|
\partial_{x_m}^{\alpha_m}v_m
\right\|_{L^2(\Omega_m)}^2.
\end{align}
Therefore,
\begin{align*}
\|v\|_H^2
=
\sum_{\alpha_1=0}^{2}\cdots
\sum_{\alpha_d=0}^{2}
\prod_{m=1}^{d}
\left\|
\partial_{x_m}^{\alpha_m}v_m
\right\|_{L^2(\Omega_m)}^2
=
\prod_{m=1}^{d}
\left(
\sum_{j=0}^{2}
\|\partial_{x_m}^{j}v_m\|_{L^2(\Omega_m)}^2
\right)
=
\prod_{m=1}^{d}
\|v_m\|_{H_m}^2.
\end{align*}
Taking square roots proves the crossnorm property \eqref{eq:H_crossnorm}.

\textup{(ii)} As explicitly visible in \eqref{eq:Fubini_mixed_regularity}, the dominating
mixed smoothness of $H$ controls derivatives $D^\alpha v$ with
$\|\alpha\|_\infty\leq2$, while $U=H^2(\Omega)$ requires only those with
$|\alpha|_1\leq2$. Since $|\alpha|_1\leq2$ implies
$\|\alpha\|_\infty\leq2$, every derivative entering the $U$-norm also enters the $H$-norm. Hence
\begin{align*}
\|v\|_U^2
=
\sum_{|\alpha|_1\leq2}
\|D^\alpha v\|_{L^2(\Omega)}^2
\leq
\sum_{\alpha\in\{0,1,2\}^d}
\|D^\alpha v\|_{L^2(\Omega)}^2
=
\|v\|_H^2,
\end{align*}
and thus $H\hookrightarrow U$ continuously. By Assumption~\ref{ass:domain}, $\overline{C^\infty(\overline{\Omega})}^{\,\|\cdot\|_U}
=
H^2(\Omega)=U$. In addition, $C^\infty(\overline{\Omega})\subset H$ as every smooth function possesses all mixed derivatives occurring in
the $H$-norm. Thus, $U
=
\overline{C^\infty(\overline{\Omega})}^{\,\|\cdot\|_U}
\subset
\overline{H}^{\,\|\cdot\|_U}
\subset
U$, which implies $\overline{H}^{\,\|\cdot\|_U}=U$ and proves the dense embedding relation \eqref{eq:H_dense_embedding}.

\textup{(iii)}
First, $\A:U\to L^2(\Omega)$ is bounded, since $\|\A v\|_{L^2(\Omega)}
\leq \sqrt{d}\,\|v\|_U$. For a rank-one function $v(x_1,\ldots,x_d)
=
\prod_{m=1}^{d}v_m(x_m)$, we have
\begin{align}
\A v(x)
=
-\sum_{\beta=1}^{d}
v_\beta''(x_\beta)
\prod_{\substack{m=1\\ m\neq\beta}}^{d}
v_m(x_m),
\end{align}
which coincides with the action of
\eqref{eq:poisson_operator_separable}.
Since $\partial_{x_m}^2:H^2(\Omega_m)\to L^2(\Omega_m)$ and
$I_m:H^2(\Omega_m)\to L^2(\Omega_m)$ are bounded, the identity extends by
density from the algebraic tensor product to all of $H$. Therefore,
\eqref{eq:poisson_operator_separable} holds with $R_\A=d$.
\end{proof}
\noindent
We note that Assumption~\ref{ass:local_differential_channels} is satisfied by a broad class of finite-order differential operators.

\begin{proposition}[Finite-order differential operators with separable coefficients]
\label{prop:lc_scope_finite_order}
Let Assumption \ref{ass:domain} hold. Assume that
$\mathcal I\subset\mathbb N_0^d$ is nonempty and finite. Consider the bounded differential operator
\begin{equation}\label{eq:lc_scope_differential_operator}
\A:U\to L^2(\D), \qquad \A v
=
\sum_{\nu\in\mathcal I}a_\nu(z)\,\partial_z^\nu v,
\quad
\partial_z^\nu
:=
\partial_{z_1}^{\nu_1}\cdots\partial_{z_d}^{\nu_d},
\end{equation}
with variable coefficients satisfying the separable representation
\begin{equation}\label{eq:lc_scope_separated_coefficients}
a_\nu(z)
=
\sum_{j=1}^{R_\nu}\prod_{m=1}^{d}a_{\nu,j,m}(z_m),
\qquad
a_{\nu,j,m}\in C(\overline{\D_m}),
\qquad R_\nu\in\mathbb N.
\end{equation}
Set $q_s:=\max_{\nu\in\mathcal I}\nu_s$, and let $H_s$ be Hilbert
spaces such that
\begin{equation*}
H:=H_1\otimes\cdots\otimes H_d\overset{\mathrm d}{\hookrightarrow} U
\end{equation*}
and $\partial_{z_s}^{\ell}:H_s\to L^2(\D_s)$ is bounded
for $\ell=0,\ldots,q_s$, with $\partial_{z_s}^0=I$.

Then Assumptions~\ref{ass:operator}
and~\ref{ass:local_differential_channels} are satisfied. More precisely,
the following statements hold.

\begin{itemize}

\item[\textup{(i)}]
For every $\nu\in\mathcal I$, $j=1,\ldots,R_\nu$, and
$s=1,\ldots,d$, the univariate operator
\begin{equation}\label{eq:lc_scope_univariate_factors}
\A_s^{(\nu,j)}h
:=
a_{\nu,j,s}\,\partial_{z_s}^{\nu_s}h,
\qquad h\in H_s,
\end{equation}
defines a bounded linear map
$\A_s^{(\nu,j)}:H_s\to L^2(\D_s)$.

\item[\textup{(ii)}]
The restriction of the differential operator $\mathcal{A}$ to $H$ admits the separable representation
\begin{equation}\label{eq:lc_scope_operator_tensorization}
\left.\A\right|_H
=
\sum_{\nu\in\mathcal I}\sum_{j=1}^{R_\nu}
\A_1^{(\nu,j)}\otimes\cdots\otimes\A_d^{(\nu,j)}.
\end{equation}
Thus, one may take
$R_\A=\sum_{\nu\in\mathcal I}R_\nu$. This representation need not be minimal.

\item[\textup{(iii)}]
Upon identifying $\beta\equiv(\nu,j)$, the local differential-channel representation
\eqref{eq:local_differential_channel_representation} holds with
\begin{equation}\label{eq:lc_scope_explicit_channels}
\begin{aligned}
\Lambda_s&:=\{\nu_s:\nu\in\mathcal I\},
&\qquad
\mathcal D_{s}^{\ell}&:=\partial_{z_s}^{\ell},\\
b_{s,\ell}^{(\nu,j)}
&:=a_{\nu,j,s}\delta_{\ell\nu_s},
&
L_s:=|\Lambda_s|&\le q_s+1.
\end{aligned}
\end{equation}
\end{itemize}
\end{proposition}

\begin{proof}
\noindent\textup{(i)}
Since the operators
$\partial_{z_s}^{\ell}:H_s\to L^2(\D_s)$ are bounded for
$\ell=0,\ldots,q_s$ and $q_s<\infty$, there exists a constant
\begin{equation*}
C_s
:=
\max_{\ell=0,\ldots,q_s}
\bigl\|
\partial_{z_s}^{\ell}
\bigr\|_{\mathcal L(H_s,L^2(\D_s))}
<\infty
\end{equation*}
such that $\|\partial_{z_s}^{\ell}h\|_{L^2(\D_s)}
\le
C_s\|h\|_{H_s}$ for all $h\in H_s$ and $\ell=0,\ldots,q_s$.
Moreover, $a_{\nu,j,s}\in C(\overline{\D_s})$ is bounded since
$\D_s$ is a bounded interval. Hence
\begin{equation*}
\begin{aligned}
\|\A_s^{(\nu,j)}h\|_{L^2(\D_s)}
\le
C_s\|a_{\nu,j,s}\|_{L^\infty(\D_s)}\|h\|_{H_s}.
\end{aligned}
\end{equation*}
Thus every univariate operator factor
$\A_s^{(\nu,j)}:H_s\to L^2(\D_s)$ is bounded.

\medskip
\noindent\textup{(ii)}
For an elementary tensor $v=v_1\otimes\cdots\otimes v_d$ with
$v_m\in H_m$, we have
\begin{equation}\label{eq:lc_scope_derivative_tensor}
\partial_z^\nu v(z)
=
\prod_{m=1}^{d}\partial_{z_m}^{\nu_m}v_m(z_m).
\end{equation}
Using \eqref{eq:lc_scope_separated_coefficients} and collecting the factors
coordinate-wise, we obtain
\begin{align*}
a_\nu(z)\partial_z^\nu v(z)
&=
\left(
\sum_{j=1}^{R_\nu}\prod_{m=1}^{d}a_{\nu,j,m}(z_m)
\right)
\left(
\prod_{m=1}^{d}\partial_{z_m}^{\nu_m}v_m(z_m)
\right)
=
\sum_{j=1}^{R_\nu}
\prod_{m=1}^{d}
\left(
a_{\nu,j,m}(z_m)\partial_{z_m}^{\nu_m}v_m(z_m)
\right).
\end{align*}
Summing the above expression over $\nu$ yields
\begin{align*}
\A v=
\left[
\sum_{\nu\in\mathcal I}\sum_{j=1}^{R_\nu}
\A_1^{(\nu,j)}\otimes\cdots\otimes\A_d^{(\nu,j)}
\right]v,
\end{align*}
which proves \eqref{eq:lc_scope_operator_tensorization} for every
elementary tensor, and hence by linearity for every finite linear combination
of elementary tensors.

By \eqref{eq:hilbert_tensor} and part~\textup{(i)}, the right-hand side of
\eqref{eq:lc_scope_operator_tensorization} defines a bounded operator
between $H$ and $L^{2}(\mathcal{D})$ satisfying
\begin{align*}
\left\|
\sum_{\nu\in\mathcal I}\sum_{j=1}^{R_\nu}
\A_1^{(\nu,j)}\otimes\cdots\otimes\A_d^{(\nu,j)}
\right\|_{\mathcal L(H,L^2(\D))}
&\le
\sum_{\nu\in\mathcal I}\sum_{j=1}^{R_\nu}
\prod_{m=1}^{d}
\|\A_m^{(\nu,j)}\|_{\mathcal L(H_m,L^2(\D_m))}
\\
&\le
\sum_{\nu\in\mathcal I}\sum_{j=1}^{R_\nu}
\prod_{m=1}^{d}
\left(
C_m\|a_{\nu,j,m}\|_{L^\infty(\D_m)}
\right)
<\infty.
\end{align*}
Since $H\hookrightarrow U$ and $\A:U\to L^2(\D)$ is bounded,
the restriction $\left.\A\right|_H:H\to L^2(\D)$ is bounded as well.
The two bounded operators coincide on the finite linear combinations of
elementary tensors, which are dense in $H$, and therefore coincide on all
of $H$. Thus, one may take
$R_\A=\sum_{\nu\in\mathcal I}R_\nu$.

The crossnorm property \eqref{eq:H_crossnorm} follows from the Hilbert
tensor-product construction, while the dense embedding
\eqref{eq:H_dense_embedding} holds by assumption. Together with
\eqref{eq:lc_scope_operator_tensorization}, this verifies
Assumption~\ref{ass:operator}.

\medskip
\noindent\textup{(iii)}
For fixed $s$ and $(\nu,j)$, we have $\nu_s\in\Lambda_s$.
The choices in \eqref{eq:lc_scope_explicit_channels} yield
\begin{align*}
\sum_{\ell\in\Lambda_s}
b_{s,\ell}^{(\nu,j)}\mathcal D_{s}^{\ell}h
&=
\sum_{\ell\in\Lambda_s}
a_{\nu,j,s}\delta_{\ell\nu_s}\partial_{z_s}^{\ell}h
=
a_{\nu,j,s}\partial_{z_s}^{\nu_s}h
=
\A_s^{(\nu,j)}h.
\end{align*}
Moreover, each $\mathcal D_{s}^{\ell}:H_s\to L^2(\D_s)$ is bounded by
assumption and each
$b_{s,\ell}^{(\nu,j)}\in C(\overline{\D_s})$.
Hence \eqref{eq:local_differential_channel_representation} holds.
Finally, $\Lambda_s\subset\{0,\ldots,q_s\}$ implies
$L_s\le q_s+1$, which completes the proof.
\end{proof}
We make the following remark concerning the trace operator in the formulation \eqref{eq:bochner_pde}.
\begin{remark}\label{rem:trace}
Under Assumption~\ref{ass:domain} the lateral boundary $\Lat$
consists of the $2\dsp$ faces
$\Lat_m^{-}=\D_1\times\cdots\times\{z_m^{-}\}\times\cdots\times\D_d$
and $\Lat_m^{+}$ defined analogously, with $m\le\dsp$; the parametric variables contribute no faces,
since \eqref{eq:parametric_pde} imposes no condition on $\partial\Xi$. Each face
is again a product domain, so that
\[
L^{2}(\Lat_m^{\pm})
\cong
\bigotimes_{k\neq m}L^{2}(\D_k).
\]
Whenever the endpoint evaluation
$\varepsilon_{z_m^{\pm}}:H_m\to\R$,
$\varepsilon_{z_m^{\pm}}v:=v(z_m^{\pm})$, is bounded,
the restriction of the trace to $H$ admits the rank-one tensor-product representation
\begin{equation}\label{eq:trace_separable}
\left.\tr_m^{\pm}\right|_H
=
I\otimes\cdots\otimes
\varepsilon_{z_m^{\pm}}
\otimes\cdots\otimes I.
\end{equation}
Thus, the trace operator has the same tensor-product structure as the operator $\A$, but with separation rank one and codomain
$L^2(\Lat_m^\pm)$ rather than $L^2(\D)$.
\end{remark}

\subsection{Residual minimization over separable neural network classes}
\label{sub:parametrization}

We seek to approximate the parameter-to-solution map $u \in U$ solving \eqref{eq:bochner_pde} by a neural network. The respective PINN formulation \cite{raissi2019physics} is then built from three ingredients. The first is the \emph{parametrization}
\begin{equation}\label{eq:parametrization}
\Pmap:\R^{P}\to U,
\qquad
\theta\mapsto\Pmap(\theta)=u_\theta,
\end{equation}
which specifies the neural architecture. We have the following structural assumption.
\begin{assumption}[Separable architecture]\label{ass:architecture}
The parameter vector splits into coordinate blocks
\begin{equation}\label{eq:parameter_blocks}
\theta=(\theta^{1},\dots,\theta^{d}),
\qquad
\theta^{m}\in\R^{P_m},
\qquad
P=\sum_{m=1}^{d}P_m.
\end{equation}
For each coordinate $m$, there exists a differentiable univariate parametrization
\begin{equation}\label{univariate networks}
\Phi_m^{k_m}:\R^{P_m}\to\W_m^{k_m},
\qquad
\Phi_m^{k_m}(\theta^{m})
=
\bigl(w_m^{1},\dots,w_m^{k_m}\bigr),
\qquad
k_m\in\mathbb N,
\end{equation}
where $\W_m^k$ denotes the product Hilbert
space\footnote{The superscript $k$ in $\W_m^k$ denotes the number of
components, not the Sobolev order.}:
\begin{equation}\label{eq:local_tuple_space}
\W_m^k
:=
\bigl\{(w_m^1,\ldots,w_m^k):
w_m^{a_m}\in H_m,\ a_m=1,\ldots,k\bigr\},
\qquad
\|w_m\|_{\W_m^k}^2
:=
\sum_{a_m=1}^k\|w_m^{a_m}\|_{H_m}^2.
\end{equation}
Moreover, there exists a fixed contraction pattern
$c\in\R^{k_1\times\cdots\times k_d}$ and a multilinear map
\begin{align}\label{eq:multilinear_contraction}
\Cmap_{\mathrm c}:\W_1^{k_1}\times\cdots\times\W_d^{k_d}\to H,
\qquad
\Cmap_{\mathrm c}(w_1,\dots,w_d)
:=
\sum_{a_1=1}^{k_1}\cdots\sum_{a_d=1}^{k_d}
c_{a_1\ldots a_d}\,
w_1^{a_1}\otimes\cdots\otimes w_d^{a_d}.
\end{align}
The parametrization is given by
\begin{equation}\label{eq:multilinear}
\Pmap(\theta)
=
\Cmap_{\mathrm c}\bigl(
\Phi_1^{k_1}(\theta^{1}),\dots,\Phi_d^{k_d}(\theta^{d})
\bigr),
\qquad
\theta\in\R^{P}.
\end{equation}
We write
$\nu:=\#\bigl\{(a_1,\dots,a_d):c_{a_1\ldots a_d}\neq0\bigr\}$
for the number of nonzero terms in the contraction.
\end{assumption}
\noindent
\noindent
To distinguish the restriction imposed by $c$ from the choice of local
networks in the architecture, we define the ambient contraction class
\begin{equation}\label{eq:ambient_contraction_class}
\mathfrak T_{\mathrm c}
:=
\Cmap_{\mathrm c}\bigl(
\W_1^{k_1}\times\cdots\times\W_d^{k_d}
\bigr)
\subset H.
\end{equation}
For fixed univariate parametrizations, the neural model class satisfies
\begin{equation}\label{eq:fixed_and_growing_model_classes}
\M=\ran\Pmap
=
\Cmap_{\mathrm c}\bigl(
\ran\Phi_1^{k_1}\times\cdots\times\ran\Phi_d^{k_d}
\bigr)
\subset\mathfrak T_{\mathrm c}\subset H.
\end{equation}
Neither \eqref{eq:ambient_contraction_class} nor
\eqref{eq:fixed_and_growing_model_classes} need be a linear subspace of $H$. We remind that the inclusion
$\M\subset\mathfrak T_{\mathrm c}\subset H$
is tightly related to the last two approximation steps in
\eqref{eq:triangle_approximation}. The second term in \eqref{eq:triangle_approximation} reflects the structural
low-rank approximation of $v\in H$ by the prescribed contraction class
$\mathfrak T_{\mathrm c}$, based on the hypothesis that the target admits
an accurate representation for moderate $k_m$.
The third term in \eqref{eq:triangle_approximation} is the
classical neural approximation error of $w\in\mathfrak T_{\mathrm c}$ by
$u_\theta\in\M$. It can be made arbitrarily small provided the univariate neural
classes are sufficiently expressive, as shown in the next proposition.  

\begin{proposition}[Universal Approximation]
\label{prop:architecture_universal_approximation}
Let Assumption~\ref{ass:architecture} hold. Suppose that for every $w_m\in\W_m^{k_m}$, where $m=1,\ldots,d$, and
$\delta>0$, there exist 
$\Phi_m^{k_m}$ of the form \eqref{univariate networks} and parameters
$\theta^m$ such that $\|w_m-\Phi_m^{k_m}(\theta^m)\|_{\W_m^{k_m}}<\delta.$
Then for every $w\in\mathfrak T_{\mathrm c}$ and $\varepsilon>0$,
there exists $u_\theta=\Pmap(\theta)\in\mathfrak T_{\mathrm c}$ such that
$\|w-u_\theta\|_H<\varepsilon$.
\end{proposition}
\begin{proof}
The triangle inequality, the crossnorm property,
and the Cauchy--Schwarz inequality give
\begin{align*}
\|\Cmap_{\mathrm c}(w_1,\ldots,w_d)\|_H
&\leq
\sum_{a_1,\ldots,a_d}
|c_{a_1\ldots a_d}|
\prod_{m=1}^d\|w_m^{a_m}\|_{H_m}
\\
&\leq
\left(
\sum_{a_1,\ldots,a_d}|c_{a_1\ldots a_d}|^2
\right)^{1/2}
\left(
\sum_{a_1,\ldots,a_d}
\prod_{m=1}^d\|w_m^{a_m}\|_{H_m}^2
\right)^{1/2}
\\
&=
\|c\|_{\mathrm F}
\prod_{m=1}^d
\left(
\sum_{a_m=1}^{k_m}\|w_m^{a_m}\|_{H_m}^2
\right)^{1/2}
=
\|c\|_{\mathrm F}
\prod_{m=1}^d\|w_m\|_{\W_m^{k_m}}.
\end{align*}
Multilinearity gives the telescoping identity
\begin{align*}
\Cmap_{\mathrm c}(w_1,\ldots,w_d)
-
\Cmap_{\mathrm c}(\widetilde w_1,\ldots,\widetilde w_d)
=
\sum_{s=1}^d
\Cmap_{\mathrm c}\bigl(
\widetilde w_1,\ldots,\widetilde w_{s-1},
w_s-\widetilde w_s,
w_{s+1},\ldots,w_d
\bigr).
\end{align*}
Applying the preceding bound to each term yields
\begin{equation}\label{eq:contraction_approximation_bound}
\begin{aligned}
&\|\Cmap_{\mathrm c}(w_1,\ldots,w_d)
-
\Cmap_{\mathrm c}(\widetilde w_1,\ldots,\widetilde w_d)\|_H
\\
&\qquad\leq
\|c\|_{\mathrm F}\sum_{s=1}^d
\left(
\prod_{m<s}\|\widetilde w_m\|_{\W_m^{k_m}}
\right)
\|w_s-\widetilde w_s\|_{\W_s^{k_s}}
\left(
\prod_{m>s}\|w_m\|_{\W_m^{k_m}}
\right).
\end{aligned}
\end{equation}
Let $w=\Cmap_{\mathrm c}(w_1,\ldots,w_d)\in\mathfrak T_{\mathrm c}$.
By our assumption, for any $\delta>0$ and each $m$, we can choose
$\Phi_m^{k_m}$ and $\theta^m$ such that $\|w_m-\widetilde w_m\|_{\W_m^{k_m}}<\delta$ holds for $\widetilde w_m:=\Phi_m^{k_m}(\theta^m)$. Therefore,
$\|\widetilde w_m\|_{\W_m^{k_m}}
\leq \|w_m\|_{\W_m^{k_m}}+1$.
Hence, \eqref{eq:contraction_approximation_bound} yields
\begin{align*}
\|w-u_\theta\|_H
\leq
\delta\,\|c\|_{\mathrm F}
\sum_{s=1}^d
\prod_{m\neq s}
\left(\|w_m\|_{\W_m^{k_m}}+1\right),
\end{align*}
where
$u_\theta=\Cmap_{\mathrm c}(\widetilde w_1,\ldots,\widetilde w_d)
=\Pmap(\theta)$.
Choosing $\delta>0$ sufficiently small proves the claim.
\end{proof}
\noindent
Recall that a variety of universal approximation results
(see, e.g., \cite{de2021approximation,guhring2020error, hornik1991approximation})
provide conditions on neural architectures ensuring that the approximation assumption of
Proposition~\ref{prop:architecture_universal_approximation} is satisfied. We also note that the above result applies directly to the CP and TT neural
classes introduced in the subsequent sections, once the corresponding
contraction patterns are specified.  


We proceed to formulate the PINN problem. Having introduced the architecture and the neural class \eqref{eq:fixed_and_growing_model_classes}, the second key ingredient is the \emph{residual operator}
\begin{equation}\label{eq:residual_operator}
\Fmap: U \to\Y,
\qquad
\Fmap(u):=
\begin{bmatrix}
\A u-f\\[2pt]
\sqrt{\lambda_{\mathrm b}}\,(\tr u-g)
\end{bmatrix},
\end{equation}
which encodes the PDE and the boundary condition. Here $\lambda_{\mathrm b}>0$ weights
the boundary residual and $\Y$ is equipped with the inner product which agrees with \eqref{eq:Y_norm}. The
parameter-to-residual map is the composition
\begin{equation}\label{eq:residual_map}
\Res:=\Fmap\circ\Pmap:\R^{P}\to\Y,
\qquad
\Res(\theta)=
\begin{bmatrix}
\A u_\theta-f\\[2pt]
\sqrt{\lambda_{\mathrm b}}\,(\tr u_\theta-g)
\end{bmatrix}.
\end{equation}
We note that if $\A$ is nonlinear, its occurrence in the subsequent Gauss--Newton linearization is understood as the Fr\'echet derivative $\diff\A(u_\theta)$ evaluated at the current state $u_\theta$; cf. \cite{jnini2025gauss}.

The third ingredient is the \emph{least-squares objective}, given by the
squared residual norm
\begin{equation}\label{eq:continuous_loss}
\mathfrak L(\theta)
=\tfrac12\norm{\Res(\theta)}_{\Y}^{2}
=\tfrac12\norm{\A u_\theta-f}_{L^{2}(\D)}^{2}
+\tfrac{\lambda_{\mathrm b}}{2}\norm{\tr u_\theta-g}_{L^{2}(\Lat)}^{2}.
\end{equation}
Observe that minimizing \eqref{eq:continuous_loss} over $\theta \in \mathbb{R}^{P}$ fits the whole family
\eqref{eq:parametric_pde} at once.


Boundary conditions may also be imposed exactly instead of adding an additional objective term, which we adopt in the main derivations for clarity of exposition; cf. \cite{sukumar2022exact}. Let $\bar u_\theta$ denote the raw, unconstrained output of the neural ansatz. We then define
\begin{equation}\label{eq:hard_constraint}
u_\theta(z)=\ell(z)\,\bar u_\theta(z)+g_{\mathrm b}(z),
\qquad
\restr{\ell}{\Lat}=0,
\qquad
\ell>0\ \text{in }\D,
\qquad
\restr{g_{\mathrm b}}{\Lat}=g,
\end{equation}
Then $u_\theta$ satisfies the prescribed boundary condition exactly, so that the boundary residual in \eqref{eq:residual_map} vanishes and may be omitted, yielding $\Y=L^{2}(\D)$.  In the separable setting, the cutoff and
the lift are also chosen separable:
\begin{equation}\label{eq:separable_cutoff}
\ell(z)=\prod_{m=1}^{\dsp}\ell_m(z_m),
\quad
\restr{\ell_m}{\partial\D_m}=0,
\qquad
g_{\mathrm b}(z)=\sum_{\beta=1}^{R_g}\prod_{m=1}^{d}g_{\mathrm b,m}^{(\beta)}(z_m),
\quad R_g<\infty,
\end{equation}
where each satisfies $\ell_m \in C^{\infty}(\overline{\mathcal{D}_{m}})$. Therefore, such hard boundary constraint preserve the separable structure of the ansatz. However, keeping the boundary residual in \eqref{eq:residual_map} is equally admissible, since by
Remark~\ref{rem:trace} the boundary blocks are of the same separable type. Nevertheless, incorporating the boundary loss makes the practically relevant construction considerably more technical, and we defer the corresponding details to the appendix.

\subsection{Separable Gauss--Newton geometry: continuous case}
\label{sub:pushforward}

Equipped with the separable neural network classes, we are now ready to derive the properties of the Gauss--Newton geometry. First, observe that the parametrization \eqref{eq:parametrization} is a nonlinear mapping, whereas its differential is the linear push-forward map
\begin{equation}\label{eq:pushforward}
\diff\Pmap(\theta):\R^{P}\to H,
\qquad
\diff\Pmap(\theta)[v]=\sum_{j=1}^{P}v_j \,\partial_{\theta_j}u_\theta,
\end{equation}
which we regard as the quasi-matrix
$\diff\Pmap(\theta)=[\partial_{\theta_1}u_\theta\ \cdots\ \partial_{\theta_P}u_\theta]$
in the sense of \eqref{eq:quasimatrix}. Its range
$T_{u_{\theta}}\M\subset H$ is the so-called generalized tangent space of $\M$ at $u_\theta$ and is a linear subspace of $H$, whose dimension can be strictly smaller than $P$ due to the lack of injectivity of \eqref{eq:parametrization} yielding rank deficiency of $\diff\Pmap(\theta)$;  cf. \cite{muller2023achieving}.

The parameter grouping \eqref{eq:parameter_blocks} induces a block-column
partition of \eqref{eq:pushforward}.  For a coordinate direction $s$ we define the
\emph{block push-forward}
\begin{equation}\label{eq:block_pushforward}
\diff_{\theta^{s}}\Pmap(\theta)
:=
\begin{bmatrix}
\partial_{\theta^{s}_{1}}u_\theta & \cdots & \partial_{\theta^{s}_{P_s}}u_\theta
\end{bmatrix},
\qquad
\diff\Pmap(\theta)
=
\begin{bmatrix}
\diff_{\theta^{1}}\Pmap(\theta) & \cdots & \diff_{\theta^{d}}\Pmap(\theta)
\end{bmatrix}.
\end{equation}
For a linear operator $\A$, the chain rule applied to \eqref{eq:residual_map}
yields the \emph{residual push-forward}
\begin{equation}\label{eq:residual_pushforward}
\diff\Res(\theta)
=
\diff\Fmap(u_\theta)\circ\diff\Pmap(\theta)
=\A\,\diff\Pmap(\theta)
:\R^{P}\to\Y .
\end{equation}
This is where the compatibility \ref{it:S3} between the architecture and the operator enters. It ensures that applying $\A$ to the columns of $\diff_{\theta^{s}}\Pmap(\theta)$, and hence to those of $\diff\Pmap(\theta)$, preserves their representation as finite sums of elementary tensors.
\begin{proposition}[Separability of the residual push-forward]
\label{prop:separable_pushforward}
Let Assumptions~\ref{ass:domain}, \ref{ass:operator}, and
\ref{ass:architecture} hold.  Then, for every
coordinate direction $s$ and every $j\le P_s$, it holds
\begin{equation}\label{eq:separable_pushforward}
\A\,\partial_{\theta^{s}_{j}}u_\theta
=
\sum_{\beta=1}^{R_{\A}}\ \sum_{a_1,\ldots,a_d}c_{a_1\ldots a_d}\;
\bigl(\A_1^{(\beta)}w_1^{a_1}\bigr)\otimes\cdots\otimes
\bigl(\A_s^{(\beta)}\partial_{\theta^{s}_{j}}w_s^{a_s}\bigr)\otimes\cdots\otimes
\bigl(\A_d^{(\beta)}w_d^{a_d}\bigr).
\end{equation}
\end{proposition}
\begin{proof}
For each elementary tensor in the representation of
$u_\theta$, the product rule gives
\[
\partial_{\theta_j^s}
\bigl(
w_1^{a_1}\otimes\cdots\otimes w_d^{a_d}
\bigr)
=
\sum_{m=1}^{d}
w_1^{a_1}\otimes\cdots\otimes
\bigl(\partial_{\theta_j^s}w_m^{a_m}\bigr)
\otimes\cdots\otimes w_d^{a_d}.
\]
Since $w_m^{a_m}$ depends only on $\theta^m$,
we have $\partial_{\theta_j^s}w_m^{a_m}=0$ for $m\neq s$. Hence only the
$s$-th term survives, and we obtain
\begin{equation}\label{eq:multilinear_leibniz}
\partial_{\theta_j^s}u_\theta
=
\sum_{a_1,\ldots,a_d}c_{a_1\ldots a_d}\,
w_1^{a_1}\otimes\cdots\otimes
\bigl(\partial_{\theta_j^s}w_s^{a_s}\bigr)
\otimes\cdots\otimes w_d^{a_d}.
\end{equation}
As $\partial_{\theta_j^s}u_\theta\in H$ and the
restriction of $\A$ to $H$ admits the separable representation
\eqref{eq:separable_operator}, applying $\A$ termwise to
\eqref{eq:multilinear_leibniz} yields
\eqref{eq:separable_pushforward}. Moreover, since $c$ has $\nu$ nonzero entries, the latter contains at most $\nu R_{\A}$ elementary tensors.
\end{proof}
\noindent
For a fixed coordinate $s$, after applying the linear PDE operator, the corresponding block of the residual push-forward is given by
\begin{align}\label{eq:block_residual_pushforward}
\A\,\diff_{\theta^s}\Pmap(\theta)
=
\begin{bmatrix}
\A\partial_{\theta_1^s}u_\theta & \cdots & \A\partial_{\theta_{P_s}^s}u_\theta
\end{bmatrix}.
\end{align}
Accordingly, the full residual push-forward has the block structure:
\begin{align}\label{eq:residual_pushforward_final}
\mathcal{A} \diff \Pmap(\theta) = \begin{bmatrix}
\A\,\diff_{\theta^1}\Pmap(\theta) & \cdots & \A\,\diff_{\theta^d}\Pmap(\theta)
\end{bmatrix}.
\end{align}
The Gauss--Newton Gramian inherits the separable structure of the residual push-forward. Indeed, it follows directly from its construction
\begin{equation}\label{eq:continuous_gramian}
\mathcal G(\theta):=\diff\Res(\theta)^{*}\diff\Res(\theta)\in\R^{P\times P},
\qquad
\mathcal G(\theta)_{ij}
=
\bigip{\A\partial_{\theta_i}u_\theta}{\A\partial_{\theta_j}u_\theta}_{L^{2}(\D)},
\end{equation}
which is \eqref{eq:quasimatrix_adjoint} applied to the quasi-matrix
$\diff\Res(\theta)=\A\,\diff\Pmap(\theta)$. We have the following.
\begin{theorem}
[Separability of the Gramian]\label{cor:separable_gramian}
Let Assumptions~\ref{ass:domain}, \ref{ass:operator}, and
\ref{ass:architecture} hold and let $i\le P_s$
and $j\le P_t$ index parameters in the coordinate blocks $s$ and $t$.  Then
\begin{equation}\label{eq:gramian_expansion}
\mathcal G(\theta)_{ij}
=
\sum_{\beta,\gamma=1}^{R_{\A}}\ \sum_{a,a'}c_ac_{a'}
\prod_{m=1}^{d}
\Bigip{\A_m^{(\beta)}\pi_m^{s,i}(a)}{\A_m^{(\gamma)}\pi_m^{t,j}(a')}_{L^{2}(\D_m)},
\ \ 
\pi_m^{s,i}(a):=
\begin{cases}
\partial_{\theta^{s}_{i}}w_s^{a_s}, & m=s,\\[2pt]
w_m^{a_m}, & m\neq s.
\end{cases}
\end{equation}
Thus $\mathcal G(\theta)_{ij}$ is a sum of at most $\nu^{2}R_{\A}^{2}$ terms,
each a product of $d$ univariate inner products.
\end{theorem}

\begin{proof}
Inserting the representation \eqref{eq:separable_pushforward} for
$\A\partial_{\theta_i^s}u_\theta$ and
$\A\partial_{\theta_j^t}u_\theta$ into \eqref{eq:continuous_gramian} gives
\begin{align*}
\mathcal G(\theta)_{ij}
&=
\sum_{\beta,\gamma=1}^{R_{\A}}\sum_{a,a'}c_ac_{a'}
\Biggl\langle
\bigotimes_{m=1}^{d}\A_m^{(\beta)}\pi_m^{s,i}(a),
\bigotimes_{m=1}^{d}\A_m^{(\gamma)}\pi_m^{t,j}(a')
\Biggr\rangle_{L^{2}(\D)}.
\end{align*}
Applying the tensor-product identity \eqref{eq:crossnorm} yields \eqref{eq:gramian_expansion}. Since $a$ and $a'$ each range over at most
$\nu$ nonzero entries of the contraction pattern, the expansion \eqref{eq:gramian_expansion} contains at most $\nu^{2}R_{\A}^{2}$ terms.
\end{proof}
\noindent
Theorem~\ref{cor:separable_gramian} is stated without committing to a specific contraction pattern, which will be specified in later sections for particular architectural choices.

\subsection{Separable Gauss--Newton geometry: discretization}
We next discretize the Gauss--Newton pullback and the associated quantities. To begin with, we exploit the separability of the domain in Assumption~\ref{ass:domain}, and define a collocation grid compatible with its tensor-product structure. For each
coordinate direction $m$, we sample $n_m$ points
\begin{equation}\label{eq:factor_grids}
\X_m=\{z_{m,1},\dots,z_{m,n_m}\}\subset\D_m,
\end{equation}
and form the tensor-product collocation grid:
\begin{equation}\label{eq:tensor grid}
\X_\D
=
\X_1\times\cdots\times\X_d,
\qquad
N_\D=\prod_{m=1}^{d}n_m.
\end{equation}
If the boundary term in \eqref{eq:continuous_loss} is included, the lateral
boundary $\Lat$ must be sampled as well. According to Remark~\ref{rem:trace},
$\Lat$ consists of the $2\dsp$ product faces $\Lat_b^\sigma$,
$b=1,\ldots,\dsp$, $\sigma\in\{-,+\}$, which we discretize by
\begin{equation}\label{eq:boundary_face_grid}
\X_{\Lat_b^\sigma}
=
\X_1\times\cdots\times
\{z_b^\sigma\}
\times\cdots\times\X_d,
\qquad
N_{b,\sigma}
:=
|\X_{\Lat_b^\sigma}|
=
\prod_{m\neq b}n_m.
\end{equation}
Consequently, the boundary discretization contains
\begin{align*}
N_{\partial\Omega}
=
\sum_{b=1}^{\dsp}\sum_{\sigma\in\{-,+\}}N_{b,\sigma}
=
2\sum_{b=1}^{\dsp}\prod_{m\neq b}n_m
\end{align*}
face-grid points. A tensor-grid point in \eqref{eq:tensor grid} is indexed by a multi-index
\begin{align*}
i=(i_1,\ldots,i_d),
\qquad
i_m\in\{1,\ldots,n_m\},
\quad
m=1,\ldots,d.
\end{align*}
Whenever a grid point is addressed by a single linear index, as in
$z_1,\ldots,z_{N_\D}$ below, that index is understood to be the merged
multi-index $\overline{i_1\ldots i_d}$ of \eqref{eq:multiindex}

\definecolor{cgA}{RGB}{0,114,178}
\definecolor{cgB}{RGB}{213,94,0}
\definecolor{cgbnd}{RGB}{0,158,115}
\tikzset{
  cg domain/.style={draw=black!70, line width=0.6pt, fill=black!3},
  cg guide/.style={draw=black!25, line width=0.35pt, densely dotted},
  cg node/.style={circle, fill=black!85, inner sep=0pt, minimum size=3.6pt},
  cg ghost/.style={circle, draw=black!30, fill=white, inner sep=0pt, minimum size=3.2pt, line width=0.4pt},
  cg facA/.style={circle, fill=cgA, inner sep=0pt, minimum size=3.4pt},
  cg facB/.style={circle, fill=cgB, inner sep=0pt, minimum size=3.4pt},
  cg bnd/.style={rectangle, fill=cgbnd, inner sep=0pt, minimum size=4.4pt},
  cg lbl/.style={font=\small},
}
\newcommand{\cgXone}{0.165,0.882,1.762,2.332,2.797,3.588,4.263}
\newcommand{\cgXtwo}{0.178,0.984,1.361,2.061,2.706,3.275}
\newcommand{\cgW}{4.6}
\newcommand{\cgH}{3.6}

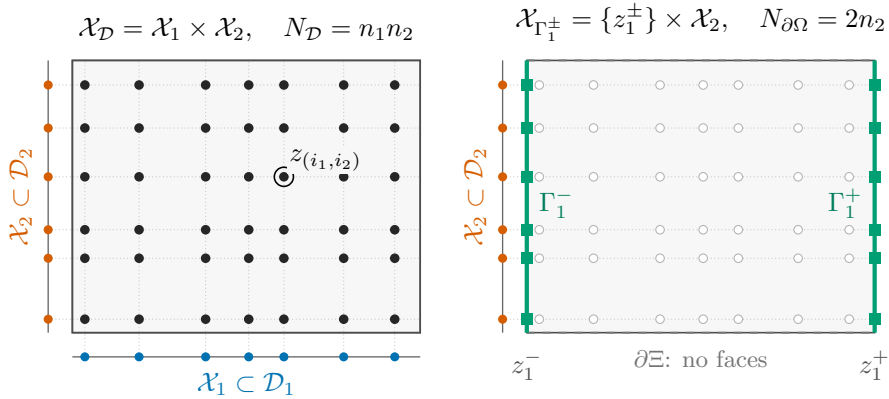
\begin{figure}[H]
\centering
\begin{tikzpicture}[x=1cm, y=1cm, baseline={(0,0)}]
  \fill[cg domain] (0,0) rectangle (\cgW,\cgH);
  \foreach \a in \cgXone {
    \draw[cg guide] (\a,-0.32) -- (\a,\cgH);
  }
  \foreach \b in \cgXtwo {
    \draw[cg guide] (-0.32,\b) -- (\cgW,\b);
  }
  \draw[cg domain, fill=none] (0,0) rectangle (\cgW,\cgH);

  \draw[black!60, line width=0.5pt] (0,-0.32) -- (\cgW,-0.32);
  \draw[black!60, line width=0.5pt] (-0.32,0) -- (-0.32,\cgH);
  \foreach \a in \cgXone {
    \node[cg facA] at (\a,-0.32) {};
  }
  \foreach \b in \cgXtwo {
    \node[cg facB] at (-0.32,\b) {};
  }

  \foreach \a in \cgXone {
    \foreach \b in \cgXtwo {
      \node[cg node] at (\a,\b) {};
    }
  }

  \node[cg lbl, text=cgA, anchor=north]
    at (0.5*\cgW,-0.40) {$\X_1\subset\D_1$};
  \node[cg lbl, text=cgB, anchor=south, rotate=90]
    at (-0.40,0.5*\cgH) {$\X_2\subset\D_2$};
  \node[cg lbl, anchor=south]
    at (0.5*\cgW,\cgH+0.12)
    {$\X_\D=\X_1\times\X_2,\quad N_\D=n_1n_2$};

  \draw[black, line width=0.6pt] (2.797,2.061) circle (3.6pt);
  \node[cg lbl, anchor=south west, fill=black!3, inner sep=1pt]
    at (2.83,2.10) {$z_{(i_1,i_2)}$};
\end{tikzpicture}\hspace{0.4cm}%
\begin{tikzpicture}[x=1cm, y=1cm, baseline={(0,0)}]
  \fill[cg domain] (0,0) rectangle (\cgW,\cgH);
  \foreach \b in \cgXtwo {
    \draw[cg guide] (-0.32,\b) -- (\cgW,\b);
  }

  \draw[black!55, line width=0.7pt, dashed] (0,0) -- (\cgW,0);
  \draw[black!55, line width=0.7pt, dashed] (0,\cgH) -- (\cgW,\cgH);

  \draw[cgbnd, line width=1.6pt] (0,0) -- (0,\cgH);
  \draw[cgbnd, line width=1.6pt] (\cgW,0) -- (\cgW,\cgH);

  \foreach \a in \cgXone {
    \foreach \b in \cgXtwo {
      \node[cg ghost] at (\a,\b) {};
    }
  }

  \draw[black!60, line width=0.5pt] (-0.32,0) -- (-0.32,\cgH);
  \foreach \b in \cgXtwo {
    \node[cg facB] at (-0.32,\b) {};
    \node[cg bnd] at (0,\b) {};
    \node[cg bnd] at (\cgW,\b) {};
  }

  \node[cg lbl, text=cgB, anchor=south, rotate=90]
    at (-0.40,0.5*\cgH) {$\X_2\subset\D_2$};
  \node[cg lbl, text=cgbnd!80!black, anchor=west,
        fill=black!3, inner sep=1pt]
    at (0.12,1.711) {$\Lat_1^-$};
  \node[cg lbl, text=cgbnd!80!black, anchor=east,
        fill=black!3, inner sep=1pt]
    at (\cgW-0.12,1.711) {$\Lat_1^+$};
  \node[cg lbl, anchor=north, black!70]
    at (0,-0.1) {$z_1^-$};
  \node[cg lbl, anchor=north, black!70]
    at (\cgW,-0.1) {$z_1^+$};
  \node[cg lbl, anchor=north, black!55, font=\footnotesize]
    at (0.5*\cgW,-0.1) {$\partial\Xi$: no faces};
  \node[cg lbl, anchor=south]
    at (0.5*\cgW,\cgH+0.12)
    {$\X_{\Lat_1^\pm}=\{z_1^\pm\}\times\X_2,\quad N_{\partial\Omega}=2n_2$};
\end{tikzpicture}
\caption{Collocation grids for $\dsp=\dpa=1$, $z=(x_1,\xi_1)$.
Left: tensor-product grid $\X_\D=\X_1\times\X_2$.
Right: lateral boundary grids
$\X_{\Lat_1^\pm}=\{z_1^\pm\}\times\X_2$.
No boundary conditions are imposed on $\partial\Xi$ (dashed).}
\label{fig:collocation_grids}
\end{figure}
\noindent
Observe that explicitly forming objects on $\X_\D$ (see Fig.~\ref{fig:collocation_grids}) quickly becomes infeasible in high dimensions. To mitigate this curse of dimensionality in the assembly, we exploit the separable structure.

Under the hard constraint \eqref{eq:hard_constraint}, the boundary component
of the residual map \eqref{eq:residual_map} vanishes. Hence, the
collocated\footnote{Collocation is represented by the evaluation
$\mathrm{ev}_{\X_\D}:C(\overline{\D})\to\R^{N_\D}$ with
$\bigl(\mathrm{ev}_{\X_\D}v\bigr)_i:=v(z_i)$, where $z_i\in\X_\D$.}
parameter-to-residual map is given by
\begin{equation}\label{eq:discrete_residual}
\rvec(\theta)
:=
\bigl[
(\A u_\theta-f)(z_1),\ldots,
(\A u_\theta-f)(z_{N_\D})
\bigr]^\top .
\end{equation}
We assume throughout that $f$ admits pointwise evaluation, as is standard in the PINN setting. By separability of the ansatz and the operator, $\A u_\theta$ preserves a finite separable representation; whenever $f$ is given in the same format, so does $\rvec(\theta)$. Discretizing the
continuous loss \eqref{eq:continuous_loss} on $\X_\D$ by equal-weight
collocation averages gives the discrete PINN loss
\begin{equation}\label{eq:discrete_loss}
\widehat{\mathcal L}(\theta)
=
\frac{1}{2N_\D}\|\rvec(\theta)\|_2^2.
\end{equation}
\begin{remark}
Throughout the discretization, we use equal-weight Monte-Carlo
collocation. In particular, for uniformly sampled points 
$\{z_i\}_{i=1}^N\subset\D$, we have
\begin{align*}
\int_\D h(z)\,\mathrm dz
\approx
\frac{|\D|}{N}\sum_{i=1}^N h(z_i).
\end{align*}
As is standard in PINN formulations, we omit the constant measure factor
$|\D|$ and use the empirical average instead. In fact, the latter can be
viewed as approximating the normalized integral
$|\D|^{-1}\int_\D h(z)\,\mathrm dz$. 
\end{remark}
Evaluating the residual push-forward \eqref{eq:residual_pushforward_final} column-wise on
$\X_\D$ turns its quasi-matrix representation into the block Jacobian:
\begin{equation}\label{eq:discrete_block_jacobian}
\mathrm J_{\rvec}
=
\begin{bmatrix}
\mathrm J_1 & \cdots & \mathrm J_d
\end{bmatrix},
\qquad
\mathrm J_s\in\R^{N_\D\times P_s},
\end{equation}
where each block is obtained by collocating the corresponding block push-forward in \eqref{eq:block_residual_pushforward}:
\begin{align}\label{eq:collocated_jacobian_block}
\mathrm J_s
=
\begin{bmatrix}
(\A\partial_{\theta_1^s}u_\theta)(z_1)
&
\cdots
&
(\A\partial_{\theta_{P_s}^s}u_\theta)(z_1)
\\
\vdots
&
&
\vdots
\\
(\A\partial_{\theta_1^s}u_\theta)(z_{N_\D})
&
\cdots
&
(\A\partial_{\theta_{P_s}^s}u_\theta)(z_{N_\D})
\end{bmatrix}.
\end{align}
Under collocation, the continuous loss gradient is approximated by replacing the
residual and residual push-forward with their discrete counterparts. This gives
the discrete gradient
\begin{equation}\label{eq:discrete_gradient}
\nabla\widehat{\mathcal L}(\theta)
=
\frac{1}{N_\D}
\mathrm J_{\rvec}^{\top}\rvec(\theta).
\end{equation}
The discrete Gauss--Newton Gramian is given by
\begin{equation}\label{eq:discrete_gramian}
\DGram(\theta)
=
\frac{1}{N_\D}\mathrm J_{\rvec}^{\top}\mathrm J_{\rvec}
=
\begin{bmatrix}
\DGram_{11} & \cdots & \DGram_{1d}\\
\vdots & \ddots & \vdots\\
\DGram_{d1} & \cdots & \DGram_{dd}
\end{bmatrix},
\qquad
\DGram_{st}
:=
\frac{1}{N_\D}\mathrm J_s^{\top}\mathrm J_t
\in\R^{P_s\times P_t},
\end{equation}
which is the equal-weight quadrature approximation of the continuous Gramian \eqref{eq:continuous_gramian}.


To expose the algebraic structure of the blocks
\eqref{eq:collocated_jacobian_block}, we evaluate
\eqref{eq:separable_pushforward} at a grid point $z_i\in\X_\D$.
Each elementary tensor then reduces to a product of univariate factor
evaluations:
\begin{align}
\bigl(\A\partial_{\theta_j^s}u_\theta\bigr)(z_i)
&=
\sum_{\beta=1}^{\RA}
\sum_{a_1,\ldots,a_d}
 c_{a_1\ldots a_d}
\left[
\prod_{m\ne s}
\bigl(\A_m^{(\beta)}w_m^{a_m}\bigr)(z_{m,i_m})
\right]
\bigl(\A_s^{(\beta)}\partial_{\theta_j^s}w_s^{a_s}\bigr)(z_{s,i_s}).
\label{eq:en-eval-entry}
\end{align}
The key observation here is that the expression in brackets  contains all dependence on the inactive grid indices
$i_{-s}$ and is independent of the parameter index $j$. Conversely, the dependence
on $j$ is confined to the local sensitivity coefficients associated with the active $s$-th coordinate. 

We now exploit the above separation to factorize
\eqref{eq:collocated_jacobian_block}. For a fixed coordinate $s\in\{1,\ldots,d\}$, we distinguish the
active grid and local component indices
\begin{align*}
i_s\in\{1,\ldots,n_s\},
\qquad
a_s\in\{1,\ldots,k_s\},
\end{align*}
from the corresponding inactive multi-indices
\begin{align*}
i_{-s}
:=
(i_1,\ldots,i_{s-1},i_{s+1},\ldots,i_d),
\qquad
a_{-s}
:=
(a_1,\ldots,a_{s-1},a_{s+1},\ldots,a_d).
\end{align*}
We use the shorthands: 
\begin{align*}
\sum_{a_{-s}}
:=
\sum_{a_1=1}^{k_1}\cdots
\sum_{a_{s-1}=1}^{k_{s-1}}
\sum_{a_{s+1}=1}^{k_{s+1}}\cdots
\sum_{a_d=1}^{k_d}, \qquad N_{-s}:=\prod_{m\ne s}n_m, \quad N_\D=n_sN_{-s}.
\end{align*}
We merge the active sample index and the active local index into the
single index
\begin{equation}\label{eq:compressed_row_index}
\overline{i_sa_s}\in\{1,\ldots,n_sk_s\},
\qquad
i_s=1,\ldots,n_s,
\quad
a_s=1,\ldots,k_s,
\end{equation}
ordered lexicographically according to \eqref{eq:multiindex}. Finally, let
$\Pi_s\in\R^{N_\D\times N_\D}$ denote the permutation matrix that maps
the active-first ordering $\overline{i_si_{-s}}$ to the natural tensor-grid
ordering $\overline{i_1\ldots i_d}$ of \eqref{eq:multiindex}: for any
$v\in\R^{N_\D}$, it holds
\begin{align}\label{permutation}
(\Pi_s v)\bigl(\overline{i_1\ldots i_d}\bigr)
=
v\bigl(\overline{i_si_{-s}}\bigr).
\end{align}

The following lemma provides the main technical ingredient for our approach. It shows that, after collocation of \eqref{eq:block_residual_pushforward}, each residual Jacobian block \eqref{eq:collocated_jacobian_block} admits a tensor representation. A suitable matricization, grouping its active modes against the inactive ones, recovers the standard matrix form appearing in the Gauss--Newton step and yields a useful separable factorization.
\begin{lemma}\label{lem:collocated_jacobian_factorization}
Let Assumptions~\ref{ass:domain}, \ref{ass:operator}, and
\ref{ass:architecture} hold, and let $\X_\D$ be the tensor-product
collocation grid \eqref{eq:tensor grid}.  The following statements hold for every $s=1,\ldots,d$.
\begin{enumerate}[label=\textup{(\roman*)}]

\item
For each $\beta=1,\ldots,R_\A$ and $a_s=1,\ldots,k_s$, define the
inactive contraction
\begin{equation}\label{eq:en-e}
e_{s,\beta,a_s}(i_{-s})
:=
\sum_{a_{-s}}
c_{a_1\ldots a_d}
\prod_{m\ne s}
\bigl(
\A_m^{(\beta)}w_m^{a_m}
\bigr)(z_{m,i_m}).
\end{equation}
Collecting these quantities over the inactive tensor grid defines the tensor
\begin{equation}\label{eq:inactive_contraction_tensor}
\mathcal E_s^{(\beta)}
\in
\R^{
n_1\times\cdots\times n_{s-1}
\times n_{s+1}\times\cdots\times n_d
\times k_s},
\qquad
\mathcal E_s^{(\beta)}(i_{-s},a_s)
:=
e_{s,\beta,a_s}(i_{-s}).
\end{equation}
Vectorizing its slices according to \eqref{eq:multiindex} and stacking them
column-wise gives the matrix
\begin{equation}\label{eq:inactive_contraction_matrix}
\mathsf E_s^{(\beta)}
\in\R^{N_{-s}\times k_s},
\qquad
\mathsf E_s^{(\beta)}(:,a_s)
:=
\vecop\bigl(
\mathcal E_s^{(\beta)}(:,\ldots,:,a_s)
\bigr),
\end{equation}
whose entries are
\begin{equation}\label{eq:inactive_contraction_vector}
\mathsf E_s^{(\beta)}\bigl(\overline{i_{-s}},a_s\bigr)
=
e_{s,\beta,a_s}(i_{-s}).
\end{equation}

\item
The corresponding local residual sensitivity matrix is given by
\begin{equation}
\Cmat_s^{(\beta)}
\in\R^{n_sk_s\times P_s},
\qquad
\Cmat_s^{(\beta)}\bigl(\overline{i_sa_s},\,j\bigr)
:=
\bigl(
\A_s^{(\beta)}
\partial_{\theta_j^s}w_s^{a_s}
\bigr)(z_{s,i_s}).
\label{eq:C_beta}
\end{equation}

\item
Define $\Psi_s^{(\beta)}\in\R^{N_\D\times n_sk_s}$ entrywise by
\begin{equation}\label{eq:Psi_beta}
\Psi_s^{(\beta)}
\bigl(i,\overline{j_sa_s}\bigr)
:=
\delta_{i_sj_s}\,
\mathsf E_s^{(\beta)}\bigl(\overline{i_{-s}},a_s\bigr),
\qquad
i=\overline{i_1\ldots i_d},
\end{equation}
and let
\begin{equation}
\Psi_s
:=
\begin{bmatrix}
\Psi_s^{(1)} & \cdots & \Psi_s^{(R_\A)}
\end{bmatrix}
\in\R^{N_{\mathcal D}\times N_s},
\ \ 
\Cmat_s
:=
\begin{bmatrix}
\Cmat_s^{(1)}\\
\vdots\\
\Cmat_s^{(R_\A)}
\end{bmatrix}
\in\R^{N_s\times P_s},
\ \ 
N_s:=R_\A n_sk_s.
\label{eq:Psi_block}
\end{equation}
Then the collocated Jacobian block admits the factorization
\begin{equation}\label{eq:Js_factorization}
\J_s=\Psi_s\Cmat_s.
\end{equation}
\end{enumerate}
\end{lemma}

\begin{proof}
We prove the three assertions in three steps. Throughout the proof, matrices indexed by the
full tensor grid use the natural ordering $\overline{i_1\ldots i_d}$.

\emph{Step 1:}
Collocation of the quasi-matrix \eqref{eq:block_residual_pushforward} produces
the tensor-valued Jacobian block
\begin{equation*}
\mathcal J_s(i_1,\ldots,i_d,j)
:=
\bigl(
\A\partial_{\theta_j^s}u_\theta
\bigr)(z_i), \qquad \mathcal J_s
\in
\R^{n_1\times\cdots\times n_d\times P_s}.
\end{equation*}
Separating the active local-component index $a_s$ from the inactive
indices $a_{-s}$ in \eqref{eq:separable_pushforward} gives
\begin{align}
\mathcal J_s(i_1,\ldots,i_d,j)
&=
\sum_{\beta=1}^{R_\A}
\sum_{a_s=1}^{k_s}
\left[
\sum_{a_{-s}}
c_{a_1\ldots a_d}
\prod_{m\ne s}
\bigl(
\A_m^{(\beta)}w_m^{a_m}
\bigr)(z_{m,i_m})
\right]
\bigl(
\A_s^{(\beta)}
\partial_{\theta_j^s}w_s^{a_s}
\bigr)(z_{s,i_s}).
\label{eq:en-eval-column}
\end{align}
The quantity in square brackets is the inactive contraction \eqref{eq:en-e}. Collecting its values over the inactive grid and vectorizing in the ordering $\overline{i_{-s}}$ gives
\eqref{eq:inactive_contraction_matrix}--\eqref{eq:inactive_contraction_vector},
which proves \textup{(i)}.

\medskip

\emph{Step 2:}
Define the local residual sensitivity tensor
\begin{equation*}
\mathcal C_s^{(\beta)}
\in
\R^{n_s\times k_s\times P_s},
\qquad
\mathcal C_s^{(\beta)}(i_s,a_s,j)
:=
\bigl(
\A_s^{(\beta)}
\partial_{\theta_j^s}w_s^{a_s}
\bigr)(z_{s,i_s}).
\end{equation*}
Stacking the tensor slices
$\mathcal C_s^{(\beta)}(i_s,:,:)\in\R^{k_s\times P_s}$ column-wise gives
\begin{equation}\label{eq:C_beta_stacking}
\Cmat_s^{(\beta)}
=
\operatorname{col}\bigl(
\mathcal C_s^{(\beta)}(1,:,:),
\ldots,
\mathcal C_s^{(\beta)}(n_s,:,:)
\bigr)
\in\R^{n_sk_s\times P_s}.
\end{equation}
which agrees with the row ordering prescribed in
\eqref{eq:compressed_row_index}, and therefore with \eqref{eq:C_beta}, proving \textup{(ii)}.

\medskip

\emph{Step 3:}
For each operator term $\beta$, define its contribution to the tensor-valued Jacobian by
\begin{equation}\label{eq:tensor_jacobian_contraction}
\mathcal J_s^{(\beta)}
\in
\R^{n_1\times\cdots\times n_d\times P_s},
\qquad
\mathcal J_s^{(\beta)}(i_1,\ldots,i_d,j)
:=
\sum_{a_s=1}^{k_s}
\mathcal E_s^{(\beta)}(i_{-s},a_s)\,
\mathcal C_s^{(\beta)}(i_s,a_s,j).
\end{equation}
For fixed $i_s$, let $\mathcal J_s^{(\beta)}(i_s,:,:)$ denote the
$N_{-s}\times P_s$ matrix whose rows are indexed by the inactive
multi-index $i_{-s}$, ordered lexicographically as $\overline{i_{-s}}$.
Then it holds:
\begin{equation}\label{eq:active_slice_factorization}
\mathcal J_s^{(\beta)}(i_s,:,:)
=
\mathsf E_s^{(\beta)}\mathcal C_s^{(\beta)}(i_s,:,:),
\qquad i_s=1,\ldots,n_s.
\end{equation}
Stacking these matrices lists the rows in the active-first ordering
$\overline{i_si_{-s}}$. Using \eqref{eq:C_beta_stacking} and
\eqref{eq:active_slice_factorization}, ordinary block matrix multiplication
gives
\begin{align*}
\operatorname{col}\bigl(
\mathcal J_s^{(\beta)}(1,:,:),\ldots,
\mathcal J_s^{(\beta)}(n_s,:,:)
\bigr)
=
\bigl(I_{n_s}\otimes\mathsf E_s^{(\beta)}\bigr)
\Cmat_s^{(\beta)}.
\end{align*}
To recover the natural tensor-grid ordering
$\overline{i_1\ldots i_d}$, let $\Pi_s\in\R^{N_\D\times N_\D}$ denote
the permutation matrix \eqref{permutation}. The natural-order matricization of
\eqref{eq:tensor_jacobian_contraction} then satisfies
\begin{equation}\label{eq:active_slice_stacking}
\begin{aligned}
\J_s^{(\beta)}
=
\Pi_s\,\operatorname{col}\bigl(
\mathcal J_s^{(\beta)}(1,:,:),\ldots,
\mathcal J_s^{(\beta)}(n_s,:,:)
\bigr)
=
\Pi_s\bigl(I_{n_s}\otimes\mathsf E_s^{(\beta)}\bigr)
\Cmat_s^{(\beta)}
=
\Psi_s^{(\beta)}\Cmat_s^{(\beta)}.
\end{aligned}
\end{equation}
The final equality agrees with \eqref{eq:Psi_beta}, since
\begin{align*}
\bigl[\Pi_s\bigl(I_{n_s}\otimes\mathsf E_s^{(\beta)}\bigr)\bigr]
\bigl(i,\overline{j_sa_s}\bigr)
=
\delta_{i_sj_s}\,
\mathsf E_s^{(\beta)}
\bigl(\overline{i_{-s}},a_s\bigr)
=
\Psi_s^{(\beta)}
\bigl(i,\overline{j_sa_s}\bigr),
\qquad
i=\overline{i_1\ldots i_d}.
\end{align*}
Finally,  \eqref{eq:Psi_block} and \eqref{eq:en-eval-column} give
\begin{equation}
\begin{aligned}
\J_s
=\sum_{\beta=1}^{R_\A}\J_s^{(\beta)}
=\sum_{\beta=1}^{R_\A}\Psi_s^{(\beta)}\Cmat_s^{(\beta)}
=\begin{bmatrix}\Psi_s^{(1)}&\cdots&\Psi_s^{(R_\A)}\end{bmatrix}
\operatorname{col}\bigl(\Cmat_s^{(1)},\ldots,\Cmat_s^{(R_\A)}\bigr)
=\Psi_s\Cmat_s.
\end{aligned}
\end{equation}
All matrices in this identity use the natural tensor-grid ordering,
which proves \textup{(iii)}.
\end{proof}
The following remark will be useful in computations appearing in later sections.
\begin{remark}\label{rem:active_ordering}
The permutation $\Pi_s$ acts only on the full tensor-grid row index of
$\Psi_s^{(\beta)}$; its columns, indexed by $\overline{j_sa_s}$, remain
unchanged. In particular, we have
\begin{equation}\label{eq:Psi_kron}
\Psi_s^{(\beta)}
=
\Pi_s\bigl(I_{n_s}\otimes\mathsf E_s^{(\beta)}\bigr),
\qquad
\Pi_s^\top\Psi_s^{(\beta)}
=
I_{n_s}\otimes\mathsf E_s^{(\beta)}.
\end{equation}
The permutation $\Pi_s$ is independent
of the operator index. Since $\Pi_s^\top\Pi_s=I_{N_\D}$, we have 
\begin{align*}
\bigl(\Psi_s^{(\beta)}\bigr)^\top\Psi_s^{(\gamma)}
=
\bigl(I_{n_s}\otimes\mathsf E_s^{(\beta)}\bigr)^\top
\Pi_s^\top\Pi_s
\bigl(I_{n_s}\otimes\mathsf E_s^{(\gamma)}\bigr)
=
I_{n_s}\otimes
\bigl[
\bigl(\mathsf E_s^{(\beta)}\bigr)^\top
\mathsf E_s^{(\gamma)}
\bigr],
\end{align*}
For $s\neq t$, the active-first row orderings are generally different, and
the corresponding Kronecker representations involve
$\Pi_s^\top\Pi_t$. Mixed-coordinate products are therefore further evaluated directly from the representation \eqref{eq:Psi_beta}.
\end{remark}

Lemma~\ref{lem:collocated_jacobian_factorization} lies at the heart of the structural
decomposition underlying the subsequent factorizations of the residual Jacobian
\eqref{eq:discrete_block_jacobian}, the discrete Gauss--Newton Gramian
\eqref{eq:discrete_gramian}, and the discrete gradient. The following theorem provides the key structural result of this work.
\begin{theorem}[Algebraic factorization of the Gauss--Newton system]\label{prop:discrete_jacobian_factorization}
Let Assumptions~\ref{ass:domain}, \ref{ass:operator}, and
\ref{ass:architecture} hold, and let $\X_\D$ be the tensor-product
collocation grid \eqref{eq:tensor grid}. The following statements hold.
\begin{enumerate}
\item
The full residual Jacobian
$\mathrm{J}_{\rvec}=[\mathrm{J}_1\ \cdots\ \mathrm{J}_d]\in\R^{N_\D\times P}$ admits the
factorization
\begin{equation}\label{eq:full_jacobian_factorization}
\mathrm{J}_{\rvec}
=
\Psi\Cmat,
\qquad
\Psi:=
[\Psi_1\ \cdots\ \Psi_d]
\in\R^{N_\D\times\Neff},
\qquad
\Cmat:=
\operatorname{diag}(\Cmat_1,\ldots,\Cmat_d)
\in\R^{\Neff\times P},
\end{equation}
where, in view of \eqref{eq:Psi_block}, the effective compressed dimension is defined by
\begin{equation}\label{eq:effective_compressed_dimension}
\Neff
:=
\sum_{s=1}^{d}N_s
=
R_\A\sum_{s=1}^{d}n_s k_s.
\end{equation}

\item
Define the compressed Gramian $\Upsilon\in\R^{\Neff\times\Neff}$ by
\begin{equation}\label{eq:compressed_gramian_blocks}
\Upsilon
:=
[\Upsilon_{st}]_{s,t=1}^{d}
=
\frac{1}{N_\D}\Psi^\top\Psi,
\qquad
\Upsilon_{st}
:=
\frac{1}{N_\D}\Psi_s^\top\Psi_t
\in\R^{N_s\times N_t}.
\end{equation}
Each block $\Upsilon_{st}$ further consists of the operator-indexed blocks
\begin{equation}\label{eq:compressed_gramian_operator_blocks}
\Upsilon_{st}
=
\left[
\Upsilon_{st}^{(\beta,\gamma)}
\right]_{\beta,\gamma=1}^{R_\A},
\qquad
\Upsilon_{st}^{(\beta,\gamma)}
:=
\frac{1}{N_\D}
\bigl(\Psi_s^{(\beta)}\bigr)^\top
\Psi_t^{(\gamma)}
\in
\R^{n_sk_s\times n_tk_t},
\end{equation}
whose rows and columns are addressed by the compressed row indices
$\overline{i_sa_s}$ and $\overline{i_ta_t}$ of \eqref{eq:compressed_row_index}.
Then the discrete Gauss--Newton Gramian satisfies
\begin{equation}\label{eq:discrete_gramian_factorization}
\DGram
=
\frac{1}{N_\D}\mathrm J_{\rvec}^\top\mathrm J_{\rvec}
=
\Cmat^\top\Upsilon\Cmat,
\qquad
\DGram_{st}
=
\Cmat_s^\top\Upsilon_{st}\Cmat_t.
\end{equation}
In particular, we have
\begin{equation}\label{eq:discrete_gramian_rank}
\rank(\DGram)
=
\rank(\mathrm J_{\rvec})
\le
\min\{N_\D,\Neff,P\}.
\end{equation}
\item
Define the compressed residual
$\rho_{\mathrm c}\in\R^{\Neff}$ by
\begin{equation}\label{eq:compressed_residual}
\rho_{\mathrm c}
:=
\operatorname{col}\bigl(
\rho_{\mathrm c,1},
\ldots,
\rho_{\mathrm c,d}
\bigr)
=
\frac{1}{N_\D}\Psi^\top\rvec(\theta),
\qquad
\rho_{\mathrm c,s}
:=
\frac{1}{N_\D}\Psi_s^\top\rvec(\theta)
\in\R^{N_s}.
\end{equation}
Each block $\rho_{\mathrm c,s}$ further consists of the operator-indexed
blocks
\begin{equation}\label{eq:compressed_residual_operator_blocks}
\rho_{\mathrm c,s}
=
\operatorname{col}\bigl(
\rho_{\mathrm c,s}^{(1)},
\ldots,
\rho_{\mathrm c,s}^{(R_\A)}
\bigr),
\qquad
\rho_{\mathrm c,s}^{(\beta)}
:=
\frac{1}{N_\D}
\bigl(\Psi_s^{(\beta)}\bigr)^\top
\rvec(\theta)
\in\R^{n_sk_s}.
\end{equation}
Then the discrete loss gradient admits the factorization
\begin{equation}\label{eq:compressed_gradient}
\nabla\widehat{\mathcal L}(\theta)
=
\Cmat^\top\rho_{\mathrm c},
\qquad
\nabla_{\theta^s}\widehat{\mathcal L}(\theta)
=
\Cmat_s^\top\rho_{\mathrm c,s},
\qquad
s=1,\ldots,d.
\end{equation}

\end{enumerate}
\end{theorem}

\begin{proof}
From the block factorization \eqref{eq:Js_factorization} we readily obtain
\begin{align*}
\mathrm J_{\rvec}
=
[\mathrm J_1\ \cdots\ \mathrm J_d]
=
[\Psi_1\Cmat_1\ \cdots\ \Psi_d\Cmat_d]
=
[\Psi_1\ \cdots\ \Psi_d]
\operatorname{diag}(\Cmat_1,\ldots,\Cmat_d)
=
\Psi\Cmat,
\end{align*}
which proves \eqref{eq:full_jacobian_factorization}. Furthermore, \eqref{eq:Psi_block} yields
\begin{equation*}
\Psi_s^\top\Psi_t
=
\left[
\bigl(\Psi_s^{(\beta)}\bigr)^\top
\Psi_t^{(\gamma)}
\right]_{\beta,\gamma=1}^{R_\A},
\end{equation*}
which gives the nested block structure
\eqref{eq:compressed_gramian_operator_blocks}. Substituting \eqref{eq:full_jacobian_factorization} into
\eqref{eq:discrete_gramian} gives
\begin{align*}
\DGram
&=
\frac{1}{N_\D}
(\Psi\Cmat)^\top(\Psi\Cmat)
=
\Cmat^\top
\left(
\frac{1}{N_\D}\Psi^\top\Psi
\right)
\Cmat
=
\Cmat^\top\Upsilon\Cmat,
\end{align*}
and the identity
$\DGram_{st}=\Cmat_s^\top\Upsilon_{st}\Cmat_t$ follows directly from
\eqref{eq:compressed_gramian_blocks}. Finally, we have
\begin{align*}
\rank(\DGram)
=
\rank(\mathrm J_{\rvec}^{\top}\mathrm J_{\rvec})
=
\rank(\mathrm J_{\rvec}),
\end{align*}
while
\begin{align*}
\rank(\mathrm J_{\rvec})
=
\rank(\Psi\Cmat)
\le
\min\{\rank(\Psi),\rank(\Cmat)\}
\le
\min\{N_\D,\Neff,P\},
\end{align*}
which proves the rank bound \eqref{eq:discrete_gramian_rank}.

Similarly, the decomposition of $\Psi$ gives
\begin{equation*}
\Psi^\top\rvec(\theta)
=
\operatorname{col}\bigl(
\Psi_1^\top\rvec(\theta),
\ldots,
\Psi_d^\top\rvec(\theta)
\bigr),
\quad
\Psi_s^\top\rvec(\theta)
=
\operatorname{col}\bigl(
(\Psi_s^{(1)})^\top\rvec(\theta),
\ldots,
(\Psi_s^{(R_\A)})^\top\rvec(\theta)
\bigr).
\end{equation*}
which proves the block representations
\eqref{eq:compressed_residual}--%
\eqref{eq:compressed_residual_operator_blocks}. Substituting
\eqref{eq:full_jacobian_factorization} into \eqref{eq:discrete_gradient}
yields
\begin{align*}
\nabla\widehat{\mathcal L}(\theta)
&=
\frac{1}{N_\D}
(\Psi\Cmat)^\top\rvec(\theta)=
\Cmat^\top
\left(
\frac{1}{N_\D}\Psi^\top\rvec(\theta)
\right)
=
\Cmat^\top\rho_{\mathrm c},
\end{align*}
which readily proves \eqref{eq:compressed_gradient}.
\end{proof}
\noindent
The factorization \eqref{eq:full_jacobian_factorization} separates global
residual directions from local sensitivity coefficients. Accordingly, we sometimes call
$\Psi_s$ a \emph{direction matrix} and $\Cmat_s$ a \emph{sensitivity matrix}.
Indeed, $\ran(\J_s)\subseteq\ran(\Psi_s)$, thus the columns of $\Psi_s$ span a
structured residual subspace containing all directions generated by variations
of the parameter block $\theta^s$, while $\Cmat_s$ maps these variations to
coefficients with respect to these spanning columns. The compressed Gramian
$\Upsilon$, whose dimension is governed by $\Neff$, then encodes the interactions
between these residual directions.

\begin{remark}\label{rem:nested structure}
Observe that the compressed Gramian \eqref{eq:compressed_gramian_blocks}
has a nested block structure: the outer blocks $\Upsilon_{st}$ correspond to
the parameter blocks $(s,t)$, while
\eqref{eq:compressed_gramian_operator_blocks} gives the inner decomposition
with respect to the operator terms $(\beta,\gamma)$. This nested decomposition
is the finest structure available at the level of the abstract contraction
pattern. However, for a specific architecture, the Gramian blocks may admit
further algebraic structure inherited from the inactive-coordinate contractions
\eqref{eq:inactive_contraction_matrix}, which take a more explicit form after
specifying \eqref{eq:multilinear_contraction}. We study this in the later
sections.
\end{remark}

Under the additional local differential-channel structure of
Assumption~\ref{ass:local_differential_channels}, the preceding factorization
admits a further exact compression.
\begin{theorem}[Exact local differential-channel compression]
\label{prop:local_channel_compression}
Let Assumptions~\ref{ass:domain}--\ref{ass:architecture} hold, and let
$\X_\D$ be the tensor-product collocation grid \eqref{eq:tensor grid}.
For every $s=1,\ldots,d$ and $\ell\in\Lambda_s$, define the local
sensitivity matrix
\begin{equation}\label{eq:local_channel_sensitivities}
\mathsf H_{s,\ell}
\in\R^{n_sk_s\times P_s},
\qquad
\mathsf H_{s,\ell}\bigl(\overline{i_sa_s},j\bigr)
:=
\bigl(
\mathcal D_{s,\ell}\partial_{\theta_j^s}w_s^{a_s}
\bigr)(z_{s,i_s}).
\end{equation}
Using the inactive contractions
\eqref{eq:inactive_contraction_matrix}, define
$\Psi_{s,\ell}^{\mathrm{ch}}\in\R^{N_\D\times n_sk_s}$ entrywise by
\begin{equation}\label{eq:local_channel_direction_columns}
\Psi_{s,\ell}^{\mathrm{ch}}
\bigl(i,\overline{j_sa_s}\bigr)
:=
\delta_{i_sj_s}
\sum_{\beta=1}^{R_\A}
b_{s,\ell}^{(\beta)}(z_{s,j_s})\,
\mathsf E_s^{(\beta)}\bigl(\overline{i_{-s}},a_s\bigr),
\qquad
i=\overline{i_1\ldots i_d},
\end{equation}
Writing $\Lambda_s=\{\ell_1,\ldots,\ell_{L_s}\}$ in a fixed order, let
\begin{equation}\label{eq:local_channel_stacking}
\Psi_s^{\mathrm{ch}}
:=
\begin{bmatrix}
\Psi_{s,\ell_1}^{\mathrm{ch}} & \cdots &
\Psi_{s,\ell_{L_s}}^{\mathrm{ch}}
\end{bmatrix}
\in\R^{N_\D\times N_s^{\mathrm{ch}}},
\quad
\Cmat_s^{\mathrm{ch}}
:=
\operatorname{col}\bigl(
\mathsf H_{s,\ell_1},\ldots,\mathsf H_{s,\ell_{L_s}}
\bigr)
\in\R^{N_s^{\mathrm{ch}}\times P_s},
\end{equation}
where $N_s^{\mathrm{ch}}:=L_s n_sk_s$. The following statements hold.

\begin{enumerate}[label=\textup{(\roman*)}]

\item
The collocated Jacobian blocks admit the factorizations
\begin{equation}\label{eq:local_channel_local_jacobian}
\J_s=\Psi_s^{\mathrm{ch}}\Cmat_s^{\mathrm{ch}},
\qquad s=1,\ldots,d.
\end{equation}
The full residual Jacobian
$\mathrm J_{\rvec}\in\R^{N_\D\times P}$
 satisfies $\mathrm J_{\rvec}=\Psi^{\mathrm{ch}}\Cmat^{\mathrm{ch}}$ with
\begin{equation*}
\Psi^{\mathrm{ch}}
:=
[\Psi_1^{\mathrm{ch}}\ \cdots\ \Psi_d^{\mathrm{ch}}]
\in\R^{N_\D\times\Neff^{\mathrm{ch}}},
\quad
\Cmat^{\mathrm{ch}}
:=
\operatorname{diag}
(\Cmat_1^{\mathrm{ch}},\ldots,\Cmat_d^{\mathrm{ch}})
\in\R^{\Neff^{\mathrm{ch}}\times P}
\end{equation*}
The effective compressed dimension is given by
\begin{equation}\label{eq:local_channel_dimension}
\Neff^{\mathrm{ch}}
:=
\sum_{s=1}^d N_s^{\mathrm{ch}}
=
\sum_{s=1}^d L_s n_sk_s.
\end{equation}

\item
Define the compressed Gramian
$\Upsilon^{\mathrm{ch}}
\in\R^{\Neff^{\mathrm{ch}}\times\Neff^{\mathrm{ch}}}$ by
\begin{equation*}
\Upsilon^{\mathrm{ch}}
:=
[\Upsilon_{st}^{\mathrm{ch}}]_{s,t=1}^d
=
\frac{1}{N_\D}
(\Psi^{\mathrm{ch}})^\top\Psi^{\mathrm{ch}},
\qquad
\Upsilon_{st}^{\mathrm{ch}}
:=
\frac{1}{N_\D}
(\Psi_s^{\mathrm{ch}})^\top\Psi_t^{\mathrm{ch}}
\in\R^{N_s^{\mathrm{ch}}\times N_t^{\mathrm{ch}}}.
\end{equation*}
Each block $\Upsilon_{st}^{\mathrm{ch}}$ further consists of the
channel-indexed blocks
\begin{equation*}
\Upsilon_{st}^{\mathrm{ch}}
=
\bigl[
\Upsilon_{st}^{\mathrm{ch},(\ell,\ell')}
\bigr]_{\ell\in\Lambda_s,\ell'\in\Lambda_t},
\qquad
\Upsilon_{st}^{\mathrm{ch},(\ell,\ell')}
:=
\frac{1}{N_\D}
(\Psi_{s,\ell}^{\mathrm{ch}})^\top
\Psi_{t,\ell'}^{\mathrm{ch}}
\in\R^{n_sk_s\times n_tk_t},
\end{equation*}
whose rows and columns are addressed by the compressed indices
$\overline{i_sa_s}$ and $\overline{i_ta_t}$ of \eqref{eq:compressed_row_index}.
Then the discrete Gauss--Newton Gramian satisfies
\begin{equation}\label{eq:local_channel_gramian_factorization}
\DGram
=
\frac{1}{N_\D}\mathrm J_{\rvec}^\top\mathrm J_{\rvec}
=
(\Cmat^{\mathrm{ch}})^\top
\Upsilon^{\mathrm{ch}}\Cmat^{\mathrm{ch}},
\qquad
\DGram_{st}
=
(\Cmat_s^{\mathrm{ch}})^\top
\Upsilon_{st}^{\mathrm{ch}}\Cmat_t^{\mathrm{ch}}.
\end{equation}
In particular,
\begin{equation*}
\rank(\DGram)
=
\rank(\mathrm J_{\rvec})
\le
\min\{N_\D,\Neff^{\mathrm{ch}},P\}.
\end{equation*}

\item
Define the compressed residual
$\rho_{\mathrm c}^{\mathrm{ch}}\in\R^{\Neff^{\mathrm{ch}}}$ by
\begin{equation*}
\rho_{\mathrm c}^{\mathrm{ch}}
:=
\operatorname{col}
(\rho_{\mathrm c,1}^{\mathrm{ch}},\ldots,
\rho_{\mathrm c,d}^{\mathrm{ch}})
=
\frac{1}{N_\D}
(\Psi^{\mathrm{ch}})^\top\rvec(\theta),
\qquad
\rho_{\mathrm c,s}^{\mathrm{ch}}
:=
\frac{1}{N_\D}
(\Psi_s^{\mathrm{ch}})^\top\rvec(\theta)
\in\R^{N_s^{\mathrm{ch}}}.
\end{equation*}
Each block $\rho_{\mathrm c,s}^{\mathrm{ch}}$ further consists of
the channel-indexed blocks
\begin{equation*}
\rho_{\mathrm c,s}^{\mathrm{ch}}
=
\operatorname{col}_{\ell\in\Lambda_s}
\bigl(\rho_{\mathrm c,s}^{\mathrm{ch},(\ell)}\bigr),
\qquad
\rho_{\mathrm c,s}^{\mathrm{ch},(\ell)}
:=
\frac{1}{N_\D}
(\Psi_{s,\ell}^{\mathrm{ch}})^\top\rvec(\theta)
\in\R^{n_sk_s}.
\end{equation*}
Then the discrete loss gradient admits the factorization
\begin{equation*}
\nabla\widehat{\mathcal L}(\theta)
=
(\Cmat^{\mathrm{ch}})^\top\rho_{\mathrm c}^{\mathrm{ch}},
\qquad
\nabla_{\theta^s}\widehat{\mathcal L}(\theta)
=
(\Cmat_s^{\mathrm{ch}})^\top\rho_{\mathrm c,s}^{\mathrm{ch}},
\qquad s=1,\ldots,d.
\end{equation*}

\end{enumerate}
\end{theorem}

\begin{proof}
Applying \eqref{eq:local_differential_channel_representation} to
the local sensitivities in \eqref{eq:C_beta} gives
\begin{equation*}
\Cmat_s^{(\beta)}\bigl(\overline{i_sa_s},j\bigr)
=
\sum_{\ell\in\Lambda_s}
b_{s,\ell}^{(\beta)}(z_{s,i_s})\,
\mathsf H_{s,\ell}\bigl(\overline{i_sa_s},j\bigr).
\end{equation*}
Substituting this expression into the block factorization
\eqref{eq:Js_factorization} of
Lemma~\ref{lem:collocated_jacobian_factorization}, using
\eqref{eq:Psi_beta}, and regrouping the finite sums, we obtain
\begin{align*}
\J_s(i,j)
=
\sum_{\ell\in\Lambda_s}\sum_{a_s=1}^{k_s}
\left[
\sum_{\beta=1}^{R_\A}
b_{s,\ell}^{(\beta)}(z_{s,i_s})\,
\mathsf E_s^{(\beta)}\bigl(\overline{i_{-s}},a_s\bigr)
\right]
\mathsf H_{s,\ell}\bigl(\overline{i_sa_s},j\bigr).
\end{align*}
Observe that the expression in square brackets
equals $\Psi_{s,\ell}^{\mathrm{ch}}\bigl(i,\overline{i_sa_s}\bigr)$.
Since the remaining entries with $j_s\ne i_s$ vanish, ordinary matrix
multiplication and \eqref{eq:local_channel_stacking} yield
\begin{align*}
\J_s
=
\sum_{\ell\in\Lambda_s}
\Psi_{s,\ell}^{\mathrm{ch}}\mathsf H_{s,\ell}
=
\Psi_s^{\mathrm{ch}}\Cmat_s^{\mathrm{ch}},
\end{align*}
which proves \eqref{eq:local_channel_local_jacobian}.
Each $\Cmat_s^{\mathrm{ch}}$ stacks $L_s$ matrices with $n_sk_s$ rows,
thus $N_s^{\mathrm{ch}}=L_s n_sk_s$, and \eqref{eq:local_channel_dimension} follows after summation over $s$. The remainder of the proof follows by the same algebraic arguments as in
Theorem~\ref{prop:discrete_jacobian_factorization}, with $\Psi_s$ and
$\Cmat_s$ replaced by $\Psi_s^{\mathrm{ch}}$ and
$\Cmat_s^{\mathrm{ch}}$, respectively, and with the local blocks indexed
by $\ell\in\Lambda_s$ instead of $\beta=1,\ldots,R_\A$.
\end{proof}

Before we proceed to applications of our abstract framework, we note that for the concrete tensor formats considered in Section \ref{sec:architectures}, we use the generic
factorization of Theorem~\ref{prop:discrete_jacobian_factorization}, whose
operator-term indexing makes it more general. The channel
formulation of Theorem~\ref{prop:local_channel_compression}, however, uses the same
tensor contractions, replacing $\Cmat_s^{(\beta)}$ by $\mathsf H_{s,\ell}$
from \eqref{eq:local_channel_sensitivities} and regrouping
$\Psi_s^{(\beta)}$ into $\Psi_{s,\ell}^{\mathrm{ch}}$ through the coefficients
$b_{s,\ell}^{(\beta)}$ as in \eqref{eq:local_channel_direction_columns}.
The resulting blocks are then indexed by channels rather than operator terms.
Both the subsequent factorizations and the compressed Gauss--Newton solver
in Section~\ref{sec:woodbury} preserve their form under the replacements
$(\Cmat,\Upsilon,\rho_{\mathrm c},\Neff)
\mapsto (\Cmat^{\mathrm{ch}},\Upsilon^{\mathrm{ch}},
\rho_{\mathrm c}^{\mathrm{ch}},\Neff^{\mathrm{ch}})$, with the corresponding change in effective compressed dimension. In Section \ref{sec:numerics}, we use the channel factorization of
Theorem~\ref{prop:local_channel_compression} due to its greater efficiency.

\section{The Gauss--Newton solver in compressed space}
\label{sec:woodbury}

In this section, we use the factorization in
Theorem~\ref{prop:discrete_jacobian_factorization} to formulate the
Gauss--Newton step directly in the effective residual space. In particular,
\eqref{eq:discrete_gramian_factorization} and
\eqref{eq:compressed_gradient} give
\begin{align*}
\DGram=\Cmat^\top\Upsilon\Cmat,
\qquad
\nabla\widehat{\mathcal L}(\theta)
=
\Cmat^\top\rho_{\mathrm c}.
\end{align*}
Thus, both the discrete Gauss--Newton matrix and the gradient factor through the same $\Neff$-dimensional space. We are
particularly interested in the regime
\begin{align*}
\Neff \ll P,
\end{align*}
where the effective compressed space is substantially smaller than the parameter space.

To exploit this structure in the Gauss--Newton linear solve, define
\begin{equation}\label{eq:wb_pushthrough}
\Kmat
:=
\Cmat\Cmat^\top
=
\diag(\Kmat_1,\ldots,\Kmat_d),
\qquad
\Kmat_s:=\Cmat_s\Cmat_s^\top.
\end{equation}
Since $\Cmat$ is block diagonal,
\begin{equation}\label{eq:wb_block_product}
(\Upsilon\Kmat)_{st}
=
\Upsilon_{st}\Kmat_t,
\qquad
s,t=1,\ldots,d.
\end{equation}
The next result gives the exact compressed formulation of
the damped Gauss--Newton step.
\begin{theorem}[Compressed Gauss--Newton step] \label{prop:compressed_gn_step} Let $\mu>0$. Then the matrix $\Upsilon\Kmat+\mu I_{\Neff}$ is nonsingular, and the unique solution of the damped Gauss--Newton system \begin{equation}\label{eq:algebraic_gn_compressed_section} \bigl(\DGram+\mu I_P\bigr)\delta\theta = \nabla\widehat{\mathcal L}(\theta) \end{equation} is obtained from the compressed system \begin{equation}\label{eq:wb_step} 
\bigl(\Upsilon\Kmat+\mu I_{\Neff}\bigr)y = \rho_{\mathrm c}, \qquad \delta\theta = \Cmat^{\top}y . 
\end{equation} 
\end{theorem} 
\begin{proof} Using \eqref{eq:wb_pushthrough}, we apply the push-through identity
\begin{equation}\label{eq:wb_pushthrough_identity} \bigl(\Cmat^{\top}\Upsilon\Cmat+\mu I_P\bigr)\Cmat^{\top} = \Cmat^{\top}\bigl(\Upsilon\Kmat+\mu I_{\Neff}\bigr). 
\end{equation}
The respective inverse push-through identity then yields
\begin{align*}
\delta\theta
&=
\bigl(\DGram+\mu I_P\bigr)^{-1}
\nabla\widehat{\mathcal L}(\theta)
=
\bigl(\Cmat^{\top}\Upsilon\Cmat+\mu I_P\bigr)^{-1}
\Cmat^{\top}\rho_{\mathrm c}
=
\Cmat^{\top}
\bigl(\Upsilon\Kmat+\mu I_{\Neff}\bigr)^{-1}
\rho_{\mathrm c}.
\end{align*}
Hence, defining $y:= \bigl(\Upsilon\Kmat+\mu I_{\Neff}\bigr)^{-1} \rho_{\mathrm c}$, we obtain the compressed system \eqref{eq:wb_step}.

It remains to verify that \eqref{eq:wb_step} is actually well-posed. Using \eqref{eq:wb_pushthrough} and 
the Sylvester's determinant identity\footnote{$\det(I+AB)=\det(I+BA)$ for conformal matrices $A$ and $B$.}, we obtain
\begin{align*}
\det\bigl(\Upsilon\Kmat+\mu I_{\Neff}\bigr)
&=
\mu^{\Neff}
\det\bigl(
I_{\Neff}+\mu^{-1}\Upsilon\Cmat\Cmat^\top
\bigr)
\\
&=
\mu^{\Neff}
\det\bigl(
I_P+\mu^{-1}\Cmat^\top\Upsilon\Cmat
\bigr)
\\
&=
\mu^{\Neff}
\det\bigl(
I_P+\mu^{-1}\DGram
\bigr)
>0.
\end{align*}
Hence $\Upsilon\Kmat+\mu I_{\Neff}$ is non-singular. It then follows from \eqref{eq:wb_pushthrough_identity} that
\begin{align*}
\bigl(\DGram+\mu I_P\bigr)\Cmat^\top y=
\Cmat^\top
\bigl(\Upsilon\Kmat+\mu I_{\Neff}\bigr)y
=
\Cmat^\top\rho_{\mathrm c}
=
\nabla\widehat{\mathcal L}(\theta).
\end{align*}
Since $\DGram+\mu I_P$ is positive definite, \eqref{eq:algebraic_gn_compressed_section} admits a unique solution
$\delta\theta=\Cmat^\top y$.
\end{proof} 
Although $\Upsilon$ and $\Kmat$ are symmetric positive semidefinite, their product is generally nonsymmetric, which renders LU factorization a natural first choice for solving  \eqref{eq:wb_step}. We summarize the computational procedure in Algorithm~\ref{alg:woodbury}.
\begin{algorithm}[H]
\caption{Compressed Gauss--Newton step}
\label{alg:woodbury}
\begin{algorithmic}[1]
\setlength{\itemsep}{0.35em}

\REQUIRE Current parameters $\theta$; architecture-dependent collocated
quantities and residual representation; damping parameter $\mu>0$

\STATE Assemble $\Cmat$, $\Upsilon$, and $\rho_{\mathrm c}$

\STATE $\Kmat_s\leftarrow \Cmat_s\Cmat_s^\top$,
\quad $s=1,\ldots,d$

\STATE $\Kmat\leftarrow\diag(\Kmat_1,\ldots,\Kmat_d)$

\STATE Form $A\leftarrow\Upsilon\Kmat+\mu I_{\Neff}$ blockwise with
$A_{st}=\Upsilon_{st}\Kmat_t+\mu\delta_{st}I_{N_s}$

\STATE $y\leftarrow\operatorname{solve}_{\mathrm{LU}}
\bigl(A,\rho_{\mathrm c}\bigr)$

\STATE Partition $y=(y_1,\ldots,y_d)$ according to
$\Neff=\sum_{s=1}^d N_s$

\STATE $\delta\theta^s\leftarrow\Cmat_s^\top y_s$,
\quad $s=1,\ldots,d$

\STATE Choose the step size $\eta_k>0$

\STATE $\theta\leftarrow\theta-\eta_k\,\delta\theta$

\ENSURE Updated parameters $\theta$

\end{algorithmic}
\end{algorithm}
We note that the construction of $\Cmat$, $\Upsilon$, and $\rho_{\mathrm c}$ is the only architecture-dependent part of the procedure, since the quantities used to build these objects are determined by the contraction pattern. Once these quantities are available, forming the local blocks $\Kmat_s$ costs $\mathcal O\bigl(\sum_s N_s^2P_s\bigr)$, while the block product in \eqref{eq:wb_block_product} costs $\mathcal O\bigl(\sum_s\Neff N_s^2\bigr)$. The dense linear algebra cost is the LU factorization, namely $\mathcal O(\Neff^3)$ operations and $\mathcal O(\Neff^2)$ storage, compared with $\mathcal O(P^3)$ operations and $\mathcal O(P^2)$ storage for a direct factorization of the discrete Gauss--Newton matrix in parameter space. The compressed route is therefore advantageous when $\Neff<P$. Note that the reduction to the compressed space yields an exact step, and no object of size $N_\D$ needs to be formed.

\section{Separable neural network architectures}
\label{sec:architectures}

The compressed Gauss--Newton solver is independent of the particular neural architecture once $\Cmat$, $\Upsilon$, and $\rho_{\mathrm c}$ have been assembled. The architecture enters through the contraction pattern in Assumption~\ref{ass:architecture}, which determines the structure and number of compressed residual directions, and hence the effective dimension $\Neff$, as well as the contractions required for their assembly. We now specialize the framework to two contraction patterns of interest: the canonical polyadic (CP) and tensor-train (TT) representations, which lead to different effective compressed dimensions and assembly procedures.

\subsection{Separable PINNs: canonical polyadic representation} 
\label{sec:separable_pinns}
We first consider the canonical polyadic (CP) representation, which corresponds to the canonical decomposition used as a reference point for the tensor-train format in \cite{oseledets2011tensor}.

For $R\in\mathbb N$, we first define the CP class of functions
\begin{equation}\label{eq:cp_tensor_class}
\mathfrak T_R^{\mathrm{CP}}
:=
\left\{
\sum_{\alpha=1}^R
w_1^\alpha\otimes\cdots\otimes w_d^\alpha:
(w_m^1,\ldots,w_m^R)\in\W_m^R,\ m=1,\ldots,d
\right\}
\subset H.
\end{equation}
This class corresponds to Assumption~\ref{ass:architecture} with
$k_m=R$ and the diagonal contraction pattern
\begin{equation}\label{eq:cp_pattern}
c_{a_1\ldots a_d}
=
\prod_{m=1}^{d-1}\delta_{a_m a_{m+1}}.
\end{equation}
Indeed, fixing $a_s=\alpha$ in \eqref{eq:cp_pattern} forces every other term in
the sum over $a_{-s}$ to vanish except for the single choice
$a_1=\cdots=a_{s-1}=a_{s+1}=\cdots=a_d=\alpha$. 

For the univariate parametrizations in Assumption~\ref{ass:architecture}, we set
\begin{equation*}
\Phi_m^R(\theta^m)
=
\bigl(
u_m^1(\,\cdot\,;\theta^m),\ldots,
u_m^R(\,\cdot\,;\theta^m)
\bigr)
\in\W_m^R.
\end{equation*}
Hence, the corresponding neural model class is given by
\begin{equation*}
\M
=
\Cmap_{\mathrm c}\bigl(
\ran\Phi_1^R\times\cdots\times\ran\Phi_d^R
\bigr)
\subset
\mathfrak T_R^{\mathrm{CP}},
\end{equation*}
and every $u_\theta\in\M$ takes the form
\begin{equation}\label{eq:rank_R_ansatz}
u_\theta(z)
=
\sum_{\alpha=1}^R
\prod_{m=1}^d u_m^\alpha(z_m;\theta^m).
\end{equation}
Collocating $u_\theta$ on the tensor grid \eqref{eq:tensor grid} yields the output tensor
\begin{equation}\label{eq:cp_output_tensor}
\mathcal U_\theta
\in
\R^{n_1\times\cdots\times n_d},
\qquad
\mathcal U_\theta(i_1,\ldots,i_d)
:=
u_\theta(z_{1,i_1},\ldots,z_{d,i_d}).
\end{equation}
For every coordinate, collect the sampled univariate factors in the factor matrix
\begin{equation}\label{eq:cp_output_matrix}
\mathsf U_m
\in\R^{n_m\times R},
\qquad
\mathsf U_m(i_m,\alpha)
:=
u_m^\alpha(z_{m,i_m}),
\qquad
m=1,\ldots,d,
\end{equation}
whose $\alpha$-th column $\mathsf U_m(:,\alpha)$ is the $\alpha$-th univariate
factor sampled on the $m$-th factor grid \eqref{eq:factor_grids}. Then
\eqref{eq:rank_R_ansatz} gives the exact canonical representation
\begin{equation}\label{eq:cp_collocated_representation}
\mathcal U_\theta(i_1,\ldots,i_d)
=
\sum_{\alpha=1}^{R}
\prod_{m=1}^{d}
\mathsf U_m(i_m,\alpha),
\qquad
\mathcal U_\theta
=
\sum_{\alpha=1}^{R}
\mathsf U_1(:,\alpha)\otimes\cdots\otimes\mathsf U_d(:,\alpha).
\end{equation}
Therefore, the SPINN output is stored through the factor matrices $\mathsf U_1,\ldots,\mathsf U_d$, requiring only $\mathcal{O}(R\sum_{m=1}^{d}n_m)$ storage.

For each coordinate $m$ and operator term $\beta$, define the collocated
operated\footnote{We use this term throughout the entire work since the factor is obtained by applying $\A_m^{(\beta)}$.} factor matrix
\begin{equation}\label{eq:cp_F_note}
\mathsf F_m^{(\beta)}
\in\R^{n_m\times R},
\qquad
\mathsf F_m^{(\beta)}(i_m,\alpha)
:=
\bigl(\A_m^{(\beta)}u_m^\alpha\bigr)(z_{m,i_m}),
\end{equation}
For the active coordinate $s$, \eqref{eq:C_beta} becomes
\begin{equation}\label{eq:cp_Cmat}
\Cmat_s^{(\beta)}
\bigl(\overline{i_s\alpha},j\bigr)
=
\bigl(
\A_s^{(\beta)}
\partial_{\theta_j^s}u_s^\alpha
\bigr)(z_{s,i_s}),
\qquad
\Cmat_s^{(\beta)}
\in\R^{n_sR\times P_s}.
\end{equation}
Since $k_s=R$, the effective compressed dimension in the CP case is then given by
\begin{equation}\label{eq:cp_effective_dimension}
N_s=R_\A n_sR,
\qquad
\Neff=R_\A R\sum_{s=1}^{d}n_s.
\end{equation}
Observe that \eqref{eq:cp_Cmat} involve only univariate
network factors and can therefore be computed efficiently by automatic
differentiation in parallel across coordinate directions.

Collocating \eqref{eq:en-e} over the inactive tensor grid gives the $\alpha$-th slice of
\eqref{eq:inactive_contraction_tensor}:
\begin{align*}
\mathcal E_s^{(\beta)}(:,\ldots,:,\alpha)
=
\bigotimes_{m\neq s}\mathsf F_m^{(\beta)}(:,\alpha)
\in
\R^{n_1\times\cdots\times n_{s-1}\times n_{s+1}\times\cdots\times n_d},
\end{align*}
which is a rank-one tensor over the inactive grid modes with
$N_{-s}$ entries. Vectorizing these tensor slices according
to \eqref{eq:inactive_contraction_vector} yields the Kronecker products of the operated factor vectors. Stacking these vectorizations column-wise over
$\alpha=1,\ldots,R$ gives the Khatri--Rao product:
\begin{equation}\label{eq:cp_E_note}
\mathsf E_s^{(\beta)}
=
\mathop{\bigodot}_{m\neq s}\mathsf F_m^{(\beta)}
\in\R^{N_{-s}\times R},
\qquad
\mathsf E_s^{(\beta)}\bigl(\overline{i_{-s}},\alpha\bigr)
=
\prod_{m\neq s}\mathsf F_m^{(\beta)}(i_m,\alpha),
\end{equation}
where the factors are ordered by increasing $m$, following the
lexicographic convention \eqref{eq:multiindex}.

The Khatri--Rao product assembles the collocated inactive contractions
\eqref{eq:en-e} from the univariate operated factors, with
each column corresponding to a fixed canonical index. For two
such columns, associated with canonical indices $\alpha$ and $\alpha'$ and
operator terms $\beta$ and $\gamma$, their normalized dot product provides the
discrete counterpart of the $L^2$ inner product between the underlying
inactive factors:
\[
\frac{1}{N_{-s}}
\mathsf E_s^{(\beta)}(:,\alpha)^\top
\mathsf E_s^{(\gamma)}(:,\alpha')
\approx
\left\langle
e_{s,\beta,\alpha},
e_{s,\gamma,\alpha'}
\right\rangle_
{L^2\left(\prod_{m\neq s}\D_m\right)}.
\]
Since the inactive factors are separable, the continuous inner product factorizes as
\[
\left\langle
e_{s,\beta,\alpha},
e_{s,\gamma,\alpha'}
\right\rangle_
{L^2\left(\prod_{m\neq s}\D_m\right)}
=
\prod_{m\neq s}
\left\langle
\A_m^{(\beta)}u_m^\alpha,
\A_m^{(\gamma)}u_m^{\alpha'}
\right\rangle_{L^2(\D_m)}.
\]
The tensor grid discretization inherits this factorization: the corresponding univariate (normalized) dot products are collected in small cross-Gram matrices, whose Hadamard products assemble the multidimensional inner products. The next proposition formalizes this discussion. However, we refer first to Remark~\ref{rem:active_ordering}, which now appears useful.
\begin{theorem}[Algebraic characterization of the compressed CP Gramian]\label{prop:cp_compressed_gramian_blocks}
For the rank-$R$ separable PINN \eqref{eq:rank_R_ansatz}, define for
each operator pair $(\beta,\gamma)$ the univariate cross-Gram matrices
\begin{equation}\label{eq:cp_H}
\mathsf H_m^{(\beta,\gamma)}
:=
\frac{1}{n_m}
\bigl(\mathsf F_m^{(\beta)}\bigr)^\top
\mathsf F_m^{(\gamma)}
\in\R^{R\times R},
\qquad
\mathsf H_m^{(\beta,\gamma)}(\alpha,\alpha')
=
\frac{1}{n_m}\sum_{i_m=1}^{n_m}
\mathsf F_m^{(\beta)}(i_m,\alpha)
\mathsf F_m^{(\gamma)}(i_m,\alpha'),
\end{equation}
together with
\begin{equation}\label{eq:cp_H_contractions}
\mathsf H_{-s}^{(\beta,\gamma)}
:=
\mathop{\bigcirc}_{m\neq s}\mathsf H_m^{(\beta,\gamma)},
\qquad
\mathsf H_{-st}^{(\beta,\gamma)}
:=
\mathop{\bigcirc}_{m\neq s,t}\mathsf H_m^{(\beta,\gamma)},
\end{equation}
where $\circ$ denotes the Hadamard product.
Then the operator-indexed compressed Gramian blocks
$\Upsilon_{st}^{(\beta,\gamma)}$ in
\eqref{eq:compressed_gramian_operator_blocks} satisfy the following.

\begin{enumerate}

\item
For $s=t$, we have
\begin{equation}\label{eq:cp_upsilon_diag}
\Upsilon_{ss}^{(\beta,\gamma)}
=
\frac{1}{n_s}
I_{n_s}\otimes
\mathsf H_{-s}^{(\beta,\gamma)}
\in\R^{n_sR\times n_sR}.
\end{equation}

\item
For $s<t$, the $(i_s,i_t)$ block of $\Upsilon_{st}^{(\beta,\gamma)}$ of
size $R\times R$ is given by
\begin{equation}\label{eq:cp_upsilon_offdiag}
\Upsilon_{st}^{(\beta,\gamma)}
\bigl(\overline{i_s:}\,,\overline{i_t:}\bigr)
=
\frac{1}{n_sn_t}
\Diag\!\bigl(\mathsf F_t^{(\beta)}(i_t,:)\bigr)
\mathsf H_{-st}^{(\beta,\gamma)}
\Diag\!\bigl(\mathsf F_s^{(\gamma)}(i_s,:)\bigr).
\end{equation}
The opposite block is obtained by transpose symmetry:
\begin{equation}\label{eq:cp_gramian_block_symmetry}
\Upsilon_{ts}^{(\gamma,\beta)}
=
\bigl(\Upsilon_{st}^{(\beta,\gamma)}\bigr)^\top.
\end{equation}

\end{enumerate}
\end{theorem}

\begin{proof}
By Remark~\ref{rem:active_ordering},
$\Psi_s^{(\beta)}=\Pi_s(I_{n_s}\otimes\mathsf E_s^{(\beta)})$. Since $\Pi_s^\top\Pi_s=I_{N_\D}$ and
$N_\D=n_sN_{-s}$, we obtain 
\begin{align}\label{eq:diagonal_gramm_CP}
\Upsilon_{ss}^{(\beta,\gamma)}
=
\frac{1}{N_\D}
\bigl(\Psi_s^{(\beta)}\bigr)^\top
\Psi_s^{(\gamma)}=
\frac{1}{n_s}
I_{n_s}\otimes
\left[
\frac{1}{N_{-s}}
\bigl(\mathsf E_s^{(\beta)}\bigr)^\top
\mathsf E_s^{(\gamma)}
\right].
\end{align}
For conforming matrices, the Khatri--Rao product satisfies the identity
\begin{equation}\label{eq:cp_khatri_rao_identity}
(\mathsf A\odot \mathsf B)^\top(\mathsf C\odot \mathsf D)
=
(\mathsf A^\top \mathsf C)\circ(\mathsf B^\top \mathsf D).
\end{equation}
Using this identity together with \eqref{eq:cp_E_note}, we obtain
\begin{align*}
\frac{1}{N_{-s}}
\bigl(\mathsf E_s^{(\beta)}\bigr)^\top
\mathsf E_s^{(\gamma)}
&=
\mathop{\bigcirc}_{m\neq s}
\left[
\frac{1}{n_m}
\bigl(\mathsf F_m^{(\beta)}\bigr)^\top
\mathsf F_m^{(\gamma)}
\right]
=
\mathsf H_{-s}^{(\beta,\gamma)}.
\end{align*}
Substituting this expression into
\eqref{eq:diagonal_gramm_CP} yields \eqref{eq:cp_upsilon_diag}.

Now let $s < t$. Using the entrywise representation \eqref{eq:Psi_beta} of
$\Psi_s^{(\beta)}$, we have
\begin{align*}
\Upsilon_{st}^{(\beta,\gamma)}
\bigl(\overline{i_s\alpha},\overline{i_t\alpha'}\bigr)
&=
\frac{1}{N_\D}
\sum_{j_1=1}^{n_1}\cdots\sum_{j_d=1}^{n_d}
\delta_{j_si_s}\delta_{j_ti_t}
\prod_{m\neq s}
\mathsf F_m^{(\beta)}(j_m,\alpha)
\prod_{m\neq t}
\mathsf F_m^{(\gamma)}(j_m,\alpha').
\end{align*}
Since $s\neq t$, the two active coordinate factors can be separated from the remaining inactive ones:
\begin{align*}
\Upsilon_{st}^{(\beta,\gamma)}
\bigl(\overline{i_s\alpha},\overline{i_t\alpha'}\bigr)
&=
\frac{1}{N_\D}
\left[
\sum_{j_s=1}^{n_s}
\delta_{j_si_s}
\mathsf F_s^{(\gamma)}(j_s,\alpha')
\right]
\left[
\sum_{j_t=1}^{n_t}
\delta_{j_ti_t}
\mathsf F_t^{(\beta)}(j_t,\alpha)
\right]
\\
&\qquad\qquad\times
\prod_{m\neq s,t}
\left[
\sum_{j_m=1}^{n_m}
\mathsf F_m^{(\beta)}(j_m,\alpha)
\mathsf F_m^{(\gamma)}(j_m,\alpha')
\right].
\end{align*}
Using first the elementary identity $\sum_{j=1}^{n}\delta_{ji}x_j=x_i$, the active sums collapse to
\begin{align*}
\Upsilon_{st}^{(\beta,\gamma)}
\bigl(\overline{i_s\alpha},\overline{i_t\alpha'}\bigr)
&=
\frac{1}{N_\D}
\mathsf F_t^{(\beta)}(i_t,\alpha)\,
\mathsf F_s^{(\gamma)}(i_s,\alpha')
\prod_{m\neq s,t}
\left[
\sum_{j_m=1}^{n_m}
\mathsf F_m^{(\beta)}(j_m,\alpha)
\mathsf F_m^{(\gamma)}(j_m,\alpha')
\right].
\end{align*}
Since $N_\D=n_sn_t\prod_{m\neq s,t}n_m$ and by \eqref{eq:cp_H} it holds
\[
\frac{1}{n_m}
\sum_{j_m=1}^{n_m}
\mathsf F_m^{(\beta)}(j_m,\alpha)
\mathsf F_m^{(\gamma)}(j_m,\alpha')
=
\mathsf H_m^{(\beta,\gamma)}(\alpha,\alpha'),
\]
we may use
$\bigl[
\Diag(a)\,\mathsf H\,\Diag(b)
\bigr](\alpha,\alpha')
=
a_\alpha \mathsf H(\alpha,\alpha')\,b_{\alpha'}$
to write
\begin{align*}
\Upsilon_{st}^{(\beta,\gamma)}
\bigl(\overline{i_s\alpha},\overline{i_t\alpha'}\bigr)
&=
\frac{1}{n_sn_t}
\mathsf F_t^{(\beta)}(i_t,\alpha)\,
\mathsf H_{-st}^{(\beta,\gamma)}(\alpha,\alpha')\,
\mathsf F_s^{(\gamma)}(i_s,\alpha')
\\
&=
\left[
\frac{1}{n_sn_t}
\Diag\!\bigl(\mathsf F_t^{(\beta)}(i_t,:)\bigr)
\mathsf H_{-st}^{(\beta,\gamma)}
\Diag\!\bigl(\mathsf F_s^{(\gamma)}(i_s,:)\bigr)
\right](\alpha,\alpha').
\end{align*}
Collecting the canonical indices $(\alpha,\alpha')$ for fixed active samples
$(i_s,i_t)$ gives \eqref{eq:cp_upsilon_offdiag}.

Finally, for $s>t$, the defining Gramian product gives
\begin{align*}
\Upsilon_{st}^{(\beta,\gamma)}
=
\frac{1}{N_\D}
\bigl(\Psi_s^{(\beta)}\bigr)^\top\Psi_t^{(\gamma)}
=
\left[
\frac{1}{N_\D}
\bigl(\Psi_t^{(\gamma)}\bigr)^\top\Psi_s^{(\beta)}
\right]^\top
=
\bigl(\Upsilon_{ts}^{(\gamma,\beta)}\bigr)^\top.
\end{align*}
Moreover, \eqref{eq:cp_H}--\eqref{eq:cp_H_contractions} imply
$\bigl(\mathsf H_{-st}^{(\beta,\gamma)}\bigr)^\top
=\mathsf H_{-st}^{(\gamma,\beta)}$. Transposing the formula for $(t,s)$ and $(\gamma,\beta)$ then
recovers \eqref{eq:cp_upsilon_offdiag} for $s>t$, completing the proof.
\end{proof}

Here we pause for a moment and compare Theorem~\ref{prop:cp_compressed_gramian_blocks} with the abstract compressed Gramian structure in Theorem~\ref{prop:discrete_jacobian_factorization}. Observe that the latter identifies the
coordinate-pair and operator-pair block levels in $\Upsilon$; recall also Remark \ref{rem:nested structure}. The CP specialization reveals one additional explicit level: each operator-pair block decomposes into active-sample blocks $(i_s,i_t)$ of size $R\times R$, indexed by the canonical indices $(\alpha,\alpha')$. Thus, the resulting block hierarchy is built from $R\times R$ atomic blocks, while the larger indices encode coordinate, operator, and active-sample couplings (Fig.~\ref{fig:cp_gramian_structure}). The atomic blocks are assembled from the quantities in
\eqref{eq:cp_F_note}, which are obtained by applying automatic
differentiation only to the univariate network factors and can be computed
efficiently in parallel across the coordinate directions.

\definecolor{gramA}{RGB}{0,114,178}     
\definecolor{gramB}{RGB}{213,94,0}      
\definecolor{gramC}{RGB}{0,158,115}     
\tikzset{
  gr blk/.style={draw=black!55, line width=0.5pt, minimum size=7mm, inner sep=0pt},
  gr diag/.style={gr blk, fill=gramA!16},
  gr off/.style={gr blk, fill=gramB!28},
  gr tr/.style={gr blk, fill=gramB!10},
  gr void/.style={gr blk, fill=black!3},
  gr cell/.style={draw=black!45, line width=0.4pt, minimum size=4.4mm, inner sep=0pt},
  gr fac/.style={draw=gramA!85!black, fill=gramA!14, line width=0.5pt,
                 rounded corners=1.5pt, minimum width=8mm, minimum height=6.5mm, inner sep=1pt},
  gr skip/.style={draw=black!35, fill=black!5, line width=0.5pt, text=black!55,
                  rounded corners=1.5pt, minimum width=8mm, minimum height=6.5mm, inner sep=1pt},
  gr lbl/.style={font=\small},
  gr tag/.style={font=\footnotesize},
  gr note/.style={font=\footnotesize, text=black!65},
  gr arrow/.style={-{Latex[length=2mm]}, draw=black!60, line width=0.5pt},
}

\begin{figure}[H]
\centering
\begin{tikzpicture}[x=1cm, y=1cm]

\begin{scope}[shift={(0,0)}]
  \foreach \i in {1,...,4}{\foreach \j in {1,...,4}{
      \node[gr void] at (0.7*\j-0.7,-0.7*\i+0.7) {};}}
  \foreach \i in {1,...,4}{\node[gr diag] at (0.7*\i-0.7,-0.7*\i+0.7) {};}
  \node[gr off] at (2.1,-0.7) {};
  \node[gr tr] at (0.7,-2.1) {};
  \node[gr tag] at (0,0) {$\Bgram_{ss}$};
  \node[gr tag] at (2.1,-0.7) {$\Bgram_{st}$};
  \node[gr tag] at (0.7,-2.1) {$\Bgram_{st}^{\!\top}$};
  \node[gr lbl, anchor=south] at (1.05,0.55) {(a) $\Bgram^{(\beta,\gamma)}$};
  \node[gr note, anchor=north] at (1.05,-2.55) {$n_sR\times n_tR$ blocks};
\end{scope}

\begin{scope}[shift={(4.3,-0.05)}]
  \node[gr lbl, anchor=south] at (0.66,0.60)
       {(b) $\Bgram_{ss}^{(\beta,\gamma)}
             =\tfrac{1}{n_s}I_{n_s}\otimes\mathsf H_{-s}^{(\beta,\gamma)}$};
  \foreach \i in {1,...,4}{\foreach \j in {1,...,4}{
      \node[gr cell, fill=white] at (0.44*\j-0.44,-0.44*\i+0.44) {};}}
  \foreach \i in {1,...,4}{
      \node[gr cell, fill=gramA!22] (Dg\i) at (0.44*\i-0.44,-0.44*\i+0.44) {};}
  \node[gr note, anchor=north] at (0.66,-1.75) {block-diagonal in $i_s$};
  \node[gr tag, text=gramA!75!black, anchor=west, inner sep=1pt] (Hlab) at (1.85,-0.25)
       {$\mathsf H_{-s}^{(\beta,\gamma)}\!\in\!\R^{R\times R}$};
  \draw[gr arrow, draw=gramA!80!black] (Hlab.west) -- (Dg3.east);
\end{scope}

\begin{scope}[shift={(9.3,-0.05)}]
  \node[gr lbl, anchor=south] at (0.66,0.60) {(c) $\Bgram_{st}^{(\beta,\gamma)}$, \ $s<t$};
  \foreach \i in {1,...,4}{\foreach \j in {1,...,4}{
      \node[gr cell, fill=gramB!12] at (0.44*\j-0.44,-0.44*\i+0.44) {};}}
  \node[gr cell, fill=gramB!55] (pick) at (0.88,-0.44) {};
  \node[gr note, anchor=north] at (0.66,-1.75) {dense in $(i_s,i_t)$};
\end{scope}

\node[draw=gramB!80!black, fill=gramB!6, line width=0.5pt, rounded corners=2pt,
      gr tag, inner sep=3pt] (CALL) at (10.15,-3.15)
     {$\tfrac{1}{n_sn_t}
       \Diag\bigl(\mathsf F_t^{(\beta)}(i_t,:)\bigr)
       \mathsf H_{-st}^{(\beta,\gamma)}
       \Diag\bigl(\mathsf F_s^{(\gamma)}(i_s,:)\bigr)$};
\draw[gr arrow, draw=gramB!80!black] (pick.east) -- ++(0.3,0)
      -- (11.75,-0.75) -- (11.75,-2.75);

\draw[gr arrow] (2.60,-1.05) -- (3.90,-1.05);
\draw[gr arrow] (7.05,-1.05) -- (8.95,-1.05);

\begin{scope}[shift={(0.35,-5.15)}]
  \node[gr lbl, anchor=west] at (-0.35,1.05)
       {(d) univariate cross-Grams
        $\mathsf H_m^{(\beta,\gamma)}
          =\tfrac1{n_m}\bigl(\mathsf F_m^{(\beta)}\bigr)^{\!\top}\mathsf F_m^{(\gamma)}$};
  \foreach \m/\x in {1/0,2/1.30,4/3.90,5/5.20}{\node[gr fac] (A\m) at (\x,0) {$\mathsf H_{\m}$};}
  \node[gr skip] (A3) at (2.60,0) {$\mathsf H_{s}$};
  \foreach \x in {0.65,1.95,3.25,4.55}{\node[gr tag] at (\x,0) {$\had$};}
  \draw[black!30, line width=0.6pt] (A3.north west) -- (A3.south east);
  \node[gr lbl, anchor=west] at (5.95,0)
       {$=\ \mathsf H_{-s}^{(\beta,\gamma)}$ \ \big(leave-one-out in (b)\big)};

  \foreach \m/\x in {1/0,2/1.30,4/3.90}{\node[gr fac] (B\m) at (\x,-1.15) {$\mathsf H_{\m}$};}
  \node[gr skip] (B3) at (2.60,-1.15) {$\mathsf H_{s}$};
  \node[gr skip] (B5) at (5.20,-1.15) {$\mathsf H_{t}$};
  \foreach \x in {0.65,1.95,3.25,4.55}{\node[gr tag] at (\x,-1.15) {$\had$};}
  \foreach \n in {B3,B5}{\draw[black!30, line width=0.6pt] (\n.north west) -- (\n.south east);}
  \node[gr lbl, anchor=west] at (5.95,-1.15)
       {$=\ \mathsf H_{-st}^{(\beta,\gamma)}$ \ \big(leave-two-out in (c)\big)};
\end{scope}

\end{tikzpicture}
\caption{Compressed CP Gramian with hierarchical structure: every block is
built from the univariate cross-Grams $\mathsf H_m^{(\beta,\gamma)}\in\R^{R\times R}$,
$m=1,\dots,d$, through leave-one-out and leave-two-out Hadamard products.}
\label{fig:cp_gramian_structure}
\end{figure}
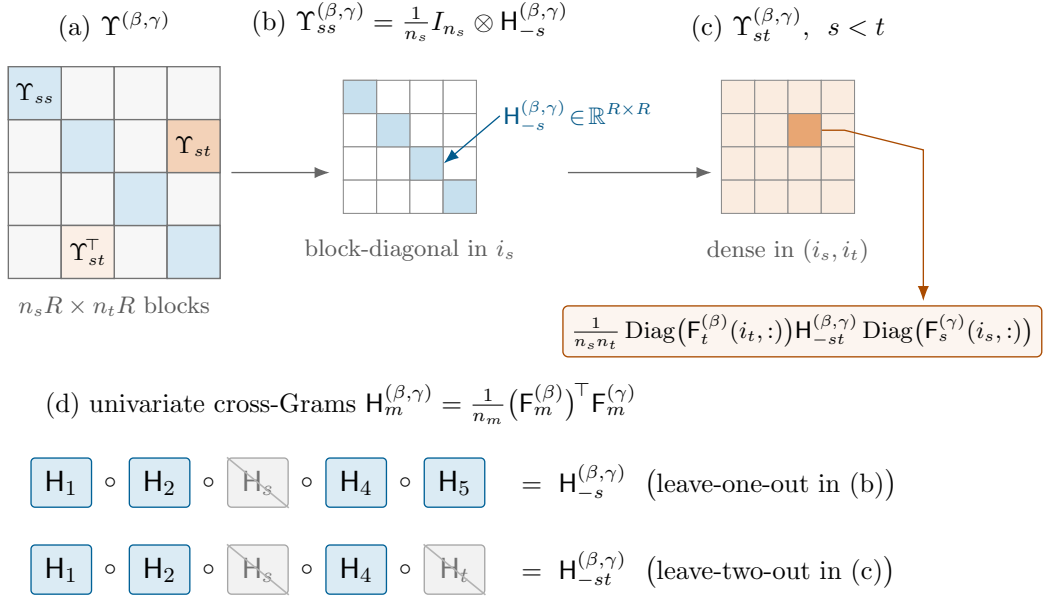

Finally, we comment on computing the gradient of the loss in the CP format. The following proposition gives the corresponding CP representation of the compressed residual \eqref{eq:compressed_residual}, which is required to compute the compressed gradient \eqref{eq:compressed_gradient}.
\begin{proposition}\label{prop:cp_compressed_gradient}
Suppose that the collocated residual admits the finite canonical representation
\begin{equation}\label{eq:cp_residual_representation}
\rvec(\theta)\bigl(\overline{i_1\ldots i_d}\bigr)
=
\sum_{\eta=1}^{R_{\rho}}
\prod_{m=1}^d\mathsf Y_m(i_m,\eta),
\qquad
\mathsf Y_m
\in\R^{n_m\times R_{\rho}}.
\end{equation}
For every coordinate $s=1,\ldots,d$ and operator term
$\beta=1,\ldots,R_\A$, define
\begin{equation}\label{eq:cp_rhs_correlations}
\mathsf Q_m^{(\beta)}
:=
\frac{1}{n_m}
\bigl(\mathsf F_m^{(\beta)}\bigr)^\top\mathsf Y_m
\in\R^{R\times R_{\rho}},
\qquad
\mathsf Q_{-s}^{(\beta)}
:=
\mathop{\bigcirc}_{m\neq s}\mathsf Q_m^{(\beta)}
\in\R^{R\times R_{\rho}}.
\end{equation}
Let
\begin{equation}\label{eq:cp_rhs_matrix}
\mathsf P_s^{(\beta)}
:=
\frac{1}{n_s}
\mathsf Y_s
\bigl(\mathsf Q_{-s}^{(\beta)}\bigr)^\top
\in\R^{n_s\times R}.
\end{equation}
Then $\rho_{\mathrm c,s}^{(\beta)}$ in
\eqref{eq:compressed_residual_operator_blocks} is the
vectorization of $\mathsf P_s^{(\beta)}$:
\begin{equation}\label{eq:cp_compressed_rhs}
\rho_{\mathrm c,s}^{(\beta)}
=
\vecop\bigl(\mathsf P_s^{(\beta)}\bigr)
\in\R^{n_sR},
\qquad
\rho_{\mathrm c,s}^{(\beta)}\bigl(\overline{i_s\alpha}\bigr)
=
\mathsf P_s^{(\beta)}(i_s,\alpha).
\end{equation}
\end{proposition}

\begin{proof}
From \eqref{eq:compressed_residual_operator_blocks} and the entrywise
representation \eqref{eq:Psi_beta}, we have
\begin{align*}
\rho_{\mathrm c,s}^{(\beta)}\bigl(\overline{i_s\alpha}\bigr)
&=
\frac{1}{N_\D}
\sum_{j_1=1}^{n_1}\cdots\sum_{j_d=1}^{n_d}
\delta_{j_si_s}
\prod_{m\neq s}
\mathsf F_m^{(\beta)}(j_m,\alpha)\,
\rvec(\theta)\bigl(\overline{j_1\ldots j_d}\bigr).
\end{align*}
Substituting \eqref{eq:cp_residual_representation} and separating the
coordinate sums gives
\begin{align*}
\rho_{\mathrm c,s}^{(\beta)}\bigl(\overline{i_s\alpha}\bigr)
&=
\frac{1}{N_\D}
\sum_{\eta=1}^{R_{\rho}}
\left[
\sum_{j_s=1}^{n_s}
\delta_{j_si_s}\mathsf Y_s(j_s,\eta)
\right]
\prod_{m\neq s}
\left[
\sum_{j_m=1}^{n_m}
\mathsf F_m^{(\beta)}(j_m,\alpha)\,
\mathsf Y_m(j_m,\eta)
\right].
\end{align*}
Using
$\sum_{j=1}^{n}\delta_{ji}x_j=x_i$ and
$N_\D=n_s\prod_{m\neq s}n_m$, we obtain
\begin{align*}
\rho_{\mathrm c,s}^{(\beta)}\bigl(\overline{i_s\alpha}\bigr)
&=
\frac{1}{n_s}
\sum_{\eta=1}^{R_{\rho}}
\mathsf Y_s(i_s,\eta)\,
\prod_{m\neq s}
\mathsf Q_m^{(\beta)}(\alpha,\eta)
=
\frac{1}{n_s}
\sum_{\eta=1}^{R_{\rho}}
\mathsf Y_s(i_s,\eta)\,
\mathsf Q_{-s}^{(\beta)}(\alpha,\eta)
=
\mathsf P_s^{(\beta)}(i_s,\alpha).
\end{align*}
Collecting these entries in the lexicographic order of $\overline{i_s\alpha}$
proves \eqref{eq:cp_compressed_rhs}.
\end{proof}

Algorithm~\ref{alg:cp_assembly} summarizes the CP assembly of the
compressed Gauss--Newton system.
\begin{algorithm}[H]
\caption{CP assembly for the compressed Gauss--Newton step}
\label{alg:cp_assembly}
\small
\begin{algorithmic}[1]
\setlength{\itemsep}{0.25em}
\REQUIRE Current parameters $\theta$; factor grids $\X_s$; rank $R$;
separated representation of the differential operator $\mathcal{A}$; source data admitting
the residual representation \eqref{eq:cp_residual_representation}
\FOR{$s=1,\ldots,d$}
  \STATE Evaluate $\mathsf F_s^{(\beta)}$ and obtain
  $\Cmat_s^{(\beta)}$ for
  $\beta=1,\ldots,R_\A$ by automatic differentiation.
  \STATE $\Cmat_s\leftarrow\operatorname{col}_{\beta=1}^{R_\A}
  \Cmat_s^{(\beta)}$, \quad $N_s\leftarrow R_\A n_sR$.
\ENDFOR
\STATE Form the residual factors $\mathsf Y_m$
from the operated factors and prescribed source data.
\FOR{$\beta,\gamma=1,\ldots,R_\A$}
  \STATE Compute $\mathsf H_m^{(\beta,\gamma)}
  \leftarrow(\mathsf F_m^{(\beta)})^\top\mathsf F_m^{(\gamma)}/n_m$
  for $m=1,\ldots,d$.
  \STATE Form $\mathsf H_{-s}^{(\beta,\gamma)}$ and
  $\mathsf H_{-st}^{(\beta,\gamma)}$ by the excluded-factor
  Hadamard products \eqref{eq:cp_H_contractions}.
  \STATE Assemble $\Upsilon_{ss}^{(\beta,\gamma)}$ for $s=1,\ldots,d$ using
  \eqref{eq:cp_upsilon_diag}. 
  \STATE Assemble $\Upsilon_{st}^{(\beta,\gamma)}$ for $s<t$ using
  \eqref{eq:cp_upsilon_offdiag}, and set
  $\Upsilon_{ts}^{(\gamma,\beta)}
  \leftarrow(\Upsilon_{st}^{(\beta,\gamma)})^\top$.
\ENDFOR
\FOR{$\beta=1,\ldots,R_\A$}
  \STATE Compute $\mathsf Q_m^{(\beta)}$ and
  $\mathsf Q_{-s}^{(\beta)}$ using \eqref{eq:cp_rhs_correlations}.
  \STATE Form $\mathsf P_s^{(\beta)}$ using \eqref{eq:cp_rhs_matrix}
  and set $\rho_{\mathrm c,s}^{(\beta)}
  \leftarrow\vecop(\mathsf P_s^{(\beta)})$, for all $s$.
\ENDFOR
\STATE Stack $\rho_{\mathrm c,s}^{(\beta)}$ and the Gramian blocks
in the order $(s,\beta,i_s,\alpha)$ to obtain
$\rho_{\mathrm c}$ and $\Upsilon$.
\STATE Compute
$\nabla_{\theta^s}\widehat{\mathcal L}(\theta)
\leftarrow\sum_{\beta=1}^{R_\A}
\bigl(\Cmat_s^{(\beta)}\bigr)^\top
\rho_{\mathrm c,s}^{(\beta)}$,
for $s=1,\ldots,d$.
\STATE Stack the coordinate gradients to obtain
$\nabla_\theta\widehat{\mathcal L}(\theta)
\leftarrow\operatorname{col}_{s=1}^d
\nabla_{\theta^s}\widehat{\mathcal L}(\theta)$.
\STATE Evaluate $\widehat{\mathcal L}(\theta)$ from the residual CP factors as $\widehat{\mathcal L}(\theta)
\leftarrow
\frac12
\sum_{\eta,\eta'=1}^{R_\rho}
\prod_{m=1}^d
\left[
\frac{1}{n_m}\mathsf Y_m^\top\mathsf Y_m
\right]_{\eta\eta'}$
\ENSURE $\{\Cmat_s\}_{s=1}^d$, $\Upsilon$, $\rho_{\mathrm c}$,
$\nabla_\theta\widehat{\mathcal L}(\theta)$,
and $\widehat{\mathcal L}(\theta)$, with
$\Neff=R_\A R\sum_s n_s$.
\end{algorithmic}
\end{algorithm}

\subsection{TT-PINNs: tensor-train representation}
\label{sec:tt_spinn}

We next specialize the abstract construction to the tensor-train (TT)
contraction pattern of \cite{oseledets2011tensor}.  For $\mathbf r=(r_0,\ldots,r_d)\in\mathbb N^{d+1}$ with $r_0=r_d=1$,
we first define the TT class of functions
\begin{equation}\label{eq:tt_tensor_class}
\mathfrak T_{\mathbf r}^{\mathrm{TT}}
:=
\left\{
z\mapsto G_1(z_1)\cdots G_d(z_d):
G_m(\,\cdot\,)(\alpha_{m-1},\alpha_m)\in H_m
\right\}
\subset H,
\end{equation}
where the matrix-valued functions $G_m:\D_m\to\R^{r_{m-1}\times r_m}$ are the TT cores and
$r_0,\ldots,r_d$ are the TT-ranks. Following \cite{oseledets2011tensor}, we also regard the
$m$-th core as a three-index object and write
\begin{equation}\label{eq:tt_core_entries}
G_m(\alpha_{m-1},z_m,\alpha_m)
:=
G_m(z_m)(\alpha_{m-1},\alpha_m),
\qquad
\alpha_{m-1}=1,\ldots,r_{m-1},
\quad
\alpha_m=1,\ldots,r_m,
\end{equation}
where the middle argument is the coordinate variable $z_m\in\D_m$,
while the two outer arguments are the TT bond indices. 

To identify \eqref{eq:tt_tensor_class} with \eqref{eq:ambient_contraction_class},
the left and right bond indices are merged into the single local index
\begin{align*}
a_m=\overline{a_m^-a_m^+},
\qquad
a_m^-=1,\ldots,r_{m-1},
\quad
a_m^+=1,\ldots,r_m,
\qquad
k_m=r_{m-1}r_m,
\end{align*}
and set $w_m^{a_m}:=G_m(a_m^-,\,\cdot\,,a_m^+)$.
The TT contraction pattern in Assumption~\ref{ass:architecture} is
then given by
\begin{equation}\label{eq:tt_pattern}
c_{a_1\ldots a_d}
:=
\prod_{m=1}^{d-1}
\delta_{a_m^+,a_{m+1}^-},
\end{equation}
which identifies neighboring bond indices consistently with the rules of
matrix multiplication and yields $\mathfrak T_{\mathrm c}=\mathfrak T_{\mathbf r}^{\mathrm{TT}}$.

For the univariate parametrizations in Assumption~\ref{ass:architecture},
we set
\begin{equation*}
\Phi_m^{k_m}(\theta^m)
=
\bigl(
G_m(\,\cdot\,;\theta^m)(a_m^-,a_m^+)
\bigr)_{a_m=1}^{k_m}
\in\W_m^{k_m}.
\end{equation*}
Hence, the corresponding neural model class is given by
\begin{equation*}
\M
=
\Cmap_{\mathrm c}\bigl(
\ran\Phi_1^{k_1}\times\cdots\times\ran\Phi_d^{k_d}
\bigr)
\subset
\mathfrak T_{\mathbf r}^{\mathrm{TT}},
\end{equation*}
and every $u_\theta\in\M$ takes the (functional) tensor-train neural form
\begin{equation}\label{eq:tt_ansatz}
u_\theta(z)
=
G_1(z_1;\theta^1)\cdots G_d(z_d;\theta^d),
\qquad
G_m(\,\cdot\,;\theta^m):\D_m\to\R^{r_{m-1}\times r_m}.
\end{equation}
We suppress the parameter dependence of the cores whenever no ambiguity
can arise. Denoting the shared bond index by $\alpha_m$ gives
$a_m=\overline{\alpha_{m-1}\alpha_m}$, and the matrix product
\eqref{eq:tt_ansatz} reads
\begin{equation}\label{eq:tt_index_form}
u_\theta(z)
=
\sum_{\alpha_1=1}^{r_1}\cdots
\sum_{\alpha_{d-1}=1}^{r_{d-1}}
\prod_{m=1}^{d}
G_m(\alpha_{m-1},z_m,\alpha_m),
\qquad
\alpha_0=\alpha_d=1.
\end{equation}
This expansion contains $\nu=\prod_{m=1}^{d-1}r_m$ separable terms.
However, \eqref{eq:tt_index_form} is never formed in practice, and the TT
is evaluated by successive matrix multiplications.

Collocating \eqref{eq:tt_ansatz} on the tensor grid
\eqref{eq:tensor grid} yields the output tensor $\mathcal U_\theta$
defined in \eqref{eq:cp_output_tensor}, whose entries now follow the
TT representation
\begin{equation}\label{eq:tt_collocated_representation}
\mathcal U_\theta(i_1,\ldots,i_d)
=
\mathsf G_1(i_1)\cdots\mathsf G_d(i_d),
\qquad
\mathsf G_m(i_m):=G_m(z_{m,i_m})\in\R^{r_{m-1}\times r_m}.
\end{equation}
Therefore, the TT representation stores only
$\mathcal O\!\left(\sum_{m=1}^{d}n_mr_{m-1}r_m\right)$ sampled core entries
instead of the full $N_\D$ tensor entries.
Fig.~\ref{fig:tt_diagram} shows the standard graphical representation:
the TT bond indices are contracted, and the collocation indices remain open.
By fixing these indices, the scalar entry
$\mathcal U_\theta(i_1,\ldots,i_d)$ is obtained.
\begin{figure}[H]
\centering
\setlength{\fboxsep}{8pt}
\fcolorbox{black!15}{black!4}{%
\begin{tikzpicture}[
  x=2.35cm,y=1.45cm,
  core/.style={draw,rounded corners=2pt,minimum width=1.10cm,
               minimum height=0.64cm,inner sep=1pt,font=\footnotesize},
  dots/.style={font=\footnotesize,inner sep=1pt},
  bond/.style={line width=0.5pt},
  leg/.style={line width=0.5pt},
  lab/.style={font=\footnotesize,inner sep=1pt}]

\node[core] (g1) at (1,0) {$\mathsf G_1$};
\node[core] (g2) at (2,0) {$\mathsf G_2$};
\node[dots] (g3) at (3,0) {$\cdots$};
\node[core] (g4) at (4,0) {$\mathsf G_{d-1}$};
\node[core] (g5) at (5,0) {$\mathsf G_{d}$};

\draw[bond] (g1.east) -- node[lab,above=7pt] {$\alpha_1$}     (g2.west);
\draw[bond] (g2.east) -- node[lab,above=7pt] {$\alpha_2$}     (g3.west);
\draw[bond] (g3.east) -- node[lab,above=7pt] {$\alpha_{d-2}$} (g4.west);
\draw[bond] (g4.east) -- node[lab,above=7pt] {$\alpha_{d-1}$} (g5.west);

\foreach \n/\lbl in {g1/{i_1},g2/{i_2},g4/{i_{d-1}},g5/{i_d}} {
  \draw[leg] (\n.south) -- ++(0,-0.62);
  \node[lab,anchor=north] at ($(\n.south)+(0,-0.66)$) {$\lbl$};
}

\node[lab,anchor=east] at ($(g1.west)+(-0.09,0)$) {$r_0=1$};
\node[lab,anchor=west] at ($(g5.east)+(0.09,0)$) {$r_d=1$};

\end{tikzpicture}%
}
\caption{Tensor train diagram for
$\mathsf G_1(i_1)\cdots\mathsf G_d(i_d)$. Each box is a core
$\mathsf G_m\in\R^{r_{m-1}\times n_m\times r_m}$; the horizontal lines are
the TT bonds $\alpha_1,\dots,\alpha_{d-1}$, and the downward legs are the
collocation indices $i_1,\dots,i_d$.}
\label{fig:tt_diagram}
\end{figure}
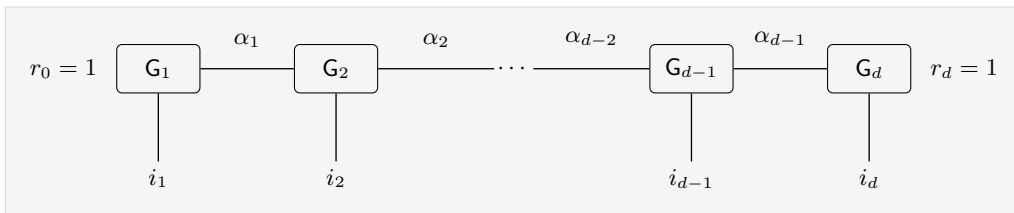

For every operator term $\beta$, we further set
$G_m^{(\beta)}:=\A_m^{(\beta)}G_m$, where
$\A_m^{(\beta)}$ acts entrywise on the matrix-valued core $G_m$ in the sense
of \eqref{eq:tt_core_entries}. Since the parameter block $\theta^s$ occurs only in the $s$-th core, the residual push-forward \eqref{eq:separable_pushforward} becomes
\begin{equation}\label{eq:tt_residual_pushforward}
\begin{aligned}
\bigl(\A\partial_{\theta_j^s}u_\theta\bigr)(z)
=
\sum_{\beta=1}^{R_\A}
G_1^{(\beta)}(z_1)\cdots
G_{s-1}^{(\beta)}(z_{s-1})
\bigl(\A_s^{(\beta)}\partial_{\theta_j^s}G_s\bigr)(z_s)
G_{s+1}^{(\beta)}(z_{s+1})\cdots
G_d^{(\beta)}(z_d).
\end{aligned}
\end{equation}
Observe that in this case, the active--inactive separation amounts to
replacing the active core by its local parameter derivative while contracting the remaining cores over the shared TT bond indices. We define the sampled operated cores
\begin{equation}\label{eq:tt_collocated_objects}
\mathsf G_m^{(\beta)}(i_m)
:=
G_m^{(\beta)}(z_{m,i_m})
=
\bigl(\A_m^{(\beta)}G_m\bigr)(z_{m,i_m})
\in\R^{r_{m-1}\times r_m}.
\end{equation}
Further substituting $a_s=\overline{\alpha_{s-1}\alpha_s}$ into the compressed row index
\eqref{eq:compressed_row_index}, the local residual sensitivity matrix
\eqref{eq:C_beta} takes the form
\begin{equation}\label{eq:tt_Cmat}
\Cmat_s^{(\beta)}
\bigl(\overline{i_s\alpha_{s-1}\alpha_s},\,j\bigr)
=
\bigl(\A_s^{(\beta)}\partial_{\theta_j^s}G_s\bigr)(z_{s,i_s})
(\alpha_{s-1},\alpha_s),
\qquad
\Cmat_s^{(\beta)}
\in\R^{n_sr_{s-1}r_s\times P_s},
\end{equation}
where the row of $\Cmat_s^{(\beta)}$ carrying the active sample $i_s$ and the active core entry $(\alpha_{s-1},\alpha_s)$ is addressed by the single index $\overline{i_s\alpha_{s-1}\alpha_s}$. Consequently, we have
\begin{equation}\label{eq:tt_effective_compressed_dimension}
N_s=R_\A n_sr_{s-1}r_s,
\qquad
\Neff=R_\A\sum_{s=1}^{d}n_sr_{s-1}r_s.
\end{equation}
Observe that unlike \eqref{eq:cp_effective_dimension}, which depends on one global canonical rank,
\eqref{eq:tt_effective_compressed_dimension} is governed by the two TT-ranks adjacent to each coordinate. Therefore, TT is advantageous when these local
ranks remain moderate while an accurate canonical representation requires a large global rank. 

For each $s=1,\ldots,d$, define the matrix units
\begin{equation}\label{eq:tt_matrix_units}
\mathsf M_{\alpha_{s-1}\alpha_s}
:=
\mathrm e_{\alpha_{s-1}}\mathrm e_{\alpha_s}^{\top}
\in\R^{r_{s-1}\times r_s},
\end{equation}
where $\mathrm e_j$ is the $j$-th canonical basis vector.
The family
$\{\mathsf M_{\alpha_{s-1}\alpha_s}\}$ is the canonical basis of
$\R^{r_{s-1}\times r_s}$ and, in particular,
\[
G_s(z_s;\theta^s)
=
\sum_{\alpha_{s-1}=1}^{r_{s-1}}
\sum_{\alpha_s=1}^{r_s}
G_s(z_s;\theta^s)(\alpha_{s-1},\alpha_s)
\mathsf M_{\alpha_{s-1}\alpha_s}.
\]
With $a_s=\overline{\alpha_{s-1}\alpha_s}$, we write
$\mathsf M_{a_s}$ for the corresponding matrix unit. These matrix units
will be used throughout to expand the TT cores and their derivatives and
to isolate active core entries in the subsequent contractions.

The inactive contraction matrix $\mathsf E_s^{(\beta)}\in\R^{N_{-s}\times r_{s-1}r_s}$ plays a central role throughout the compressed construction. In the TT case, its columns are indexed by the adjacent TT bond indices
$\alpha_{s-1}$ and $\alpha_s$ left open after removing the $s$-th active
core:
\begin{equation}\label{eq:tt_E_entry}
\begin{aligned}
\mathsf E_s^{(\beta)}
\bigl(
\overline{i_{-s}},
\overline{\alpha_{s-1}\alpha_s}
\bigr)=
\left[
\mathsf G_1^{(\beta)}(i_1)\cdots
\mathsf G_{s-1}^{(\beta)}(i_{s-1})
\right](1,\alpha_{s-1})
\left[
\mathsf G_{s+1}^{(\beta)}(i_{s+1})\cdots
\mathsf G_d^{(\beta)}(i_d)
\right](\alpha_s,1).
\end{aligned}
\end{equation}
The following proposition gives two equivalent matrix representations of
$\mathsf E_s^{(\beta)}$ that will be useful for characterizing the compressed
Gramian $\Upsilon$ in the TT format.
\begin{proposition} \label{prop:tt_inactive_contraction_structure}
For a fixed active coordinate $s$, set
\begin{align*}
N_{<s}:=\prod_{m<s}n_m,
\qquad
N_{>s}:=\prod_{m>s}n_m,
\qquad
N_{-s}=N_{<s}N_{>s},
\end{align*}
The following statements hold.
\begin{enumerate}
\item The inactive contraction matrix factorizes as
\begin{equation}\label{eq:tt_E_matrix}
\mathsf E_s^{(\beta)}
=
\mathsf U_{<s}^{(\beta)}
\otimes
\mathsf U_{>s}^{(\beta)},
\end{equation}
where $\mathsf U_{<s}^{(\beta)} \in\R^{N_{<s}\times r_{s-1}}$
and $\mathsf U_{>s}^{(\beta)}\in\R^{N_{>s}\times r_s}$ denote the left and right partial contractions, respectively, and are given by
\begin{align}\label{eq:partial trains}
\mathsf U_{<s}^{(\beta)}(\overline{i_{<s}},:)
:=
\mathsf G_1^{(\beta)}(i_1)\cdots
\mathsf G_{s-1}^{(\beta)}(i_{s-1}), \quad
\mathsf U_{>s}^{(\beta)}(\overline{i_{>s}},:)
:=
\left[
\mathsf G_{s+1}^{(\beta)}(i_{s+1})\cdots
\mathsf G_d^{(\beta)}(i_d)
\right]^\top 
\end{align}
with $\mathsf U_{<1}^{(\beta)}=1$ and $\mathsf U_{>d}^{(\beta)}=1$.
\item Each entry of $\mathsf E_s^{(\beta)}$ admits the representation
\begin{equation}\label{eq:tt_E_matrix_unit_representation}
\begin{aligned}
&\mathsf E_s^{(\beta)}
\bigl(
\overline{i_{-s}},
\overline{\alpha_{s-1}\alpha_s}
\bigr)=
\mathsf G_1^{(\beta)}(i_1)\cdots
\mathsf G_{s-1}^{(\beta)}(i_{s-1})
\mathsf M_{\alpha_{s-1}\alpha_s}
\mathsf G_{s+1}^{(\beta)}(i_{s+1})\cdots
\mathsf G_d^{(\beta)}(i_d),
\end{aligned}
\end{equation}
where
$\mathsf M_{\alpha_{s-1}\alpha_s}
\in\R^{r_{s-1}\times r_s}$ denotes the matrix unit \eqref{eq:tt_matrix_units}.
\end{enumerate}
\end{proposition}

\begin{proof}
Since $\overline{i_{-s}}=\overline{i_{<s}i_{>s}}$, then $\mathsf U_{<s}^{(\beta)}
\otimes
\mathsf U_{>s}^{(\beta)} \in \mathbb{R}^{N_{-s} \times r_{s-1}r_{s}}$. By the Kronecker convention \eqref{eq:kron_convention},
\begin{align*}
\bigl(
\mathsf U_{<s}^{(\beta)}
\otimes
\mathsf U_{>s}^{(\beta)}
\bigr)
\bigl(
\overline{i_{<s}i_{>s}},
\overline{\alpha_{s-1}\alpha_s}
\bigr)=\mathsf U_{<s}^{(\beta)}
\bigl(\overline{i_{<s}},\alpha_{s-1}\bigr)
\mathsf U_{>s}^{(\beta)}
\bigl(\overline{i_{>s}},\alpha_s\bigr).
\end{align*}
Observe that the right-hand side here is given by
\begin{align*}
\mathsf U_{<s}^{(\beta)}
\bigl(\overline{i_{<s}},\alpha_{s-1}\bigr)
\mathsf U_{>s}^{(\beta)}
\bigl(\overline{i_{>s}},\alpha_s\bigr) = \left[
\mathsf G_1^{(\beta)}(i_1)\cdots
\mathsf G_{s-1}^{(\beta)}(i_{s-1})
\right](1,\alpha_{s-1})\left[
\mathsf G_{s+1}^{(\beta)}(i_{s+1})\cdots
\mathsf G_d^{(\beta)}(i_d)
\right](\alpha_s,1),
\end{align*}
which coincides with
$\mathsf E_s^{(\beta)}
(
\overline{i_{-s}},
\overline{\alpha_{s-1}\alpha_s}
)$
by \eqref{eq:tt_E_entry}. Hence, \eqref{eq:tt_E_matrix} holds.

The two factors in \eqref{eq:tt_E_entry} can equivalently be selected by the corresponding coordinate vectors:
\begin{equation*}
\begin{aligned}
\mathsf E_s^{(\beta)}
\bigl(
\overline{i_{-s}},
\overline{\alpha_{s-1}\alpha_s}
\bigr)&=
\left[
\mathsf G_1^{(\beta)}(i_1)\cdots
\mathsf G_{s-1}^{(\beta)}(i_{s-1})
\,\mathrm e_{\alpha_{s-1}}
\right]
\left[
\mathrm e_{\alpha_s}^{\top}
\mathsf G_{s+1}^{(\beta)}(i_{s+1})\cdots
\mathsf G_d^{(\beta)}(i_d)
\right]
 \\ &=
\left[
\mathsf G_1^{(\beta)}(i_1)\cdots
\mathsf G_{s-1}^{(\beta)}(i_{s-1})
\right]
\mathsf M_{\alpha_{s-1}\alpha_s}
\left[
\mathsf G_{s+1}^{(\beta)}(i_{s+1})\cdots
\mathsf G_d^{(\beta)}(i_d)
\right].
\end{aligned}
\end{equation*}
Removing the brackets by associativity of matrix multiplication proves \eqref{eq:tt_E_matrix_unit_representation}.
\end{proof}

The representation \eqref{eq:tt_E_matrix_unit_representation}
writes each entry of $\mathsf E_s^{(\beta)}$ as a scalar TT contraction.
Since $\mathsf E_s^{(\beta)}$ determines the columns of
$\Psi_s^{(\beta)}$ through \eqref{eq:Psi_beta}, the compressed Gramian
requires tensor-grid averages of products of such contractions, with the
active coordinates pinned by the column indices. The following recursion
computes these averages efficiently.

\begin{lemma}\label{lem:averaged_tt_recursion}
Let $(p_m)_{m=0}^{d}$ and $(q_m)_{m=0}^{d}$ be two rank sequences with
$p_0=q_0=1$, and consider
$\mathsf A_m(j_m)\in\R^{p_{m-1}\times p_m}$ and
$\mathsf B_m(j_m)\in\R^{q_{m-1}\times q_m}$
for $j_m=1,\ldots,n_m$ and $m=1,\ldots,d$.
Define the recursion
\begin{equation}\label{eq:averaged_tt_recursion}
\mathsf Z_0:=1,
\qquad
\mathsf Z_m
:=
\frac{1}{n_m}
\sum_{j_m=1}^{n_m}
\mathsf A_m(j_m)^\top
\mathsf Z_{m-1}
\mathsf B_m(j_m)
\in\R^{p_m\times q_m},
\qquad
m=1,\ldots,d.
\end{equation}
Then
\begin{equation}\label{eq:averaged_tt_recursion_result}
\mathsf Z_m
=
\frac{1}{\prod_{\ell=1}^{m}n_\ell}
\sum_{j_1=1}^{n_1}\cdots\sum_{j_m=1}^{n_m}
\bigl[
\mathsf A_1(j_1)\cdots\mathsf A_m(j_m)
\bigr]^\top
\bigl[
\mathsf B_1(j_1)\cdots\mathsf B_m(j_m)
\bigr].
\end{equation}
If the $m$-th coordinate is pinned to a fixed sample $i_m$, then
\begin{equation}\label{eq:pinned_update}
\frac{1}{n_m}
\sum_{j_m=1}^{n_m}
\delta_{j_mi_m}\,
\mathsf A_m(j_m)^\top
\mathsf Z_{m-1}
\mathsf B_m(j_m)
=
\frac{1}{n_m}
\mathsf A_m(i_m)^\top
\mathsf Z_{m-1}
\mathsf B_m(i_m).
\end{equation}
\end{lemma}

\begin{proof}
We prove \eqref{eq:averaged_tt_recursion_result} by induction on $m$.
For $m=1$, we obtain
\begin{align*}
\mathsf Z_1
=
\frac{1}{n_1}
\sum_{j_1=1}^{n_1}
\mathsf A_1(j_1)^\top\mathsf B_1(j_1),
\end{align*}
which is the base case for \eqref{eq:averaged_tt_recursion_result}.
Suppose that it holds at level $m-1$. Substituting the induction hypothesis
into \eqref{eq:averaged_tt_recursion} gives
\begin{align*}
\mathsf Z_m
&=
\frac{1}{n_m}
\sum_{j_m=1}^{n_m}
\mathsf A_m(j_m)^\top
\left[
\frac{1}{\prod_{\ell=1}^{m-1}n_\ell}
\sum_{j_1=1}^{n_1}\cdots
\sum_{j_{m-1}=1}^{n_{m-1}}
\right.
\\[-1mm]
&\hspace{35mm}\left.
\bigl[
\mathsf A_1(j_1)\cdots
\mathsf A_{m-1}(j_{m-1})
\bigr]^\top
\bigl[
\mathsf B_1(j_1)\cdots
\mathsf B_{m-1}(j_{m-1})
\bigr]
\right]
\mathsf B_m(j_m)
\\
&=
\frac{1}{\prod_{\ell=1}^{m}n_\ell}
\sum_{j_1=1}^{n_1}\cdots\sum_{j_m=1}^{n_m}
\bigl[
\mathsf A_1(j_1)\cdots\mathsf A_m(j_m)
\bigr]^\top
\bigl[
\mathsf B_1(j_1)\cdots\mathsf B_m(j_m)
\bigr],
\end{align*}
where we used linearity of the finite sums and associativity of matrix
multiplication. It completes the induction step and proves
\eqref{eq:averaged_tt_recursion_result}. Formula
\eqref{eq:pinned_update} follows directly from the Kronecker delta.
\end{proof}
\noindent 
Lemma~\ref{lem:averaged_tt_recursion} shows that the tensor-grid average of
the product of two TT contractions can be computed one coordinate at a time.
We refer to such an averaged product as a double-layer TT contraction. The two
layers may have different rank sequences. We remark beforehand that two
instances occur below: $p_m=q_m=r_m$ for the compressed Gramian in
Theorem~\ref{thm:tt_gramian_forward_backward}, and $p_m=r_m$, with $(q_m)$ the TT ranks of the collocated residual, for the compressed residual in Proposition~\ref{prop:tt_compressed_gradient}.

Motivated by the recursive formula \eqref{eq:averaged_tt_recursion}, we
distinguish the two possible directions in which an averaged double-layer
contraction can be propagated through the $m$-th pair of TT cores. To formalize this distinction, we introduce the forward and backward transfer operators. We note that these transfer operator sweeps are standard environment contractions used in the tensor-train and density-matrix renormalization group literature; cf. \cite{HoltzRohwedderSchneider2012,oseledets2011tensor}.
\begin{proposition}
\label{prop:tt_transfer_operators}
Let $\mathsf A_m(j_m)\in\R^{p_{m-1}\times p_m}$ and
$\mathsf B_m(j_m)\in\R^{q_{m-1}\times q_m}$ be as in
Lemma~\ref{lem:averaged_tt_recursion}. For
\[
\mathsf X\in\R^{p_{m-1}\times q_{m-1}},
\qquad
\mathsf Y\in\R^{p_m\times q_m},
\]
define the forward transfer operator
\begin{equation}\label{eq:forward_transfer_operator}
\Tfwd_m
:
\R^{p_{m-1}\times q_{m-1}}
\to
\R^{p_m\times q_m},
\qquad
\Tfwd_m(\mathsf X)
:=
\frac{1}{n_m}
\sum_{j_m=1}^{n_m}
\bigl(\mathsf A_m(j_m)\bigr)^\top
\mathsf X
\,\mathsf B_m(j_m),
\end{equation}
and the backward transfer operator
\begin{equation}\label{eq:backward_transfer_operator}
\Tbwd_m
:
\R^{p_m\times q_m}
\to
\R^{p_{m-1}\times q_{m-1}},
\qquad
\Tbwd_m(\mathsf Y)
:=
\frac{1}{n_m}
\sum_{j_m=1}^{n_m}
\mathsf A_m(j_m)
\,\mathsf Y
\bigl(\mathsf B_m(j_m)\bigr)^\top.
\end{equation}
With respect to the Frobenius inner product these operators are adjoint:
\begin{equation}\label{eq:transfer_adjoint_identity}
\big\langle
\Tfwd_m(\mathsf X),
\mathsf Y
\big\rangle_{\mathrm F}
=
\big\langle
\mathsf X,
\Tbwd_m(\mathsf Y)
\big\rangle_{\mathrm F},
\end{equation}
where
$\langle\mathsf U,\mathsf V\rangle_{\mathrm F}
:=\operatorname{tr}(\mathsf U^\top\mathsf V)$
for matrices $\mathsf U,\mathsf V$ of the same size.
\end{proposition}

\begin{proof}
By the definition of the Frobenius inner product and cyclicity of the trace,
\begin{align*}
\big\langle
\Tfwd_m(\mathsf X),
\mathsf Y
\big\rangle_{\mathrm F}
&=
\frac{1}{n_m}
\sum_{j_m=1}^{n_m}
\operatorname{tr}
\left(
\bigl[
\bigl(\mathsf A_m(j_m)\bigr)^\top
\mathsf X
\,\mathsf B_m(j_m)
\bigr]^\top
\mathsf Y
\right)
\\
&=
\frac{1}{n_m}
\sum_{j_m=1}^{n_m}
\operatorname{tr}
\left(
\mathsf X^\top
\mathsf A_m(j_m)
\,\mathsf Y
\bigl(\mathsf B_m(j_m)\bigr)^\top
\right)=
\big\langle
\mathsf X,
\Tbwd_m(\mathsf Y)
\big\rangle_{\mathrm F},
\end{align*}
which proves \eqref{eq:transfer_adjoint_identity}. 
\end{proof}
\noindent
Throughout, we write $\Tfwd_m^{(\beta,\gamma)}$ and
$\Tbwd_m^{(\beta,\gamma)}$ for the instance
$(\mathsf A_m,\mathsf B_m)=\bigl(\mathsf G_m^{(\beta)},\mathsf G_m^{(\gamma)}\bigr)$
of \eqref{eq:forward_transfer_operator}--\eqref{eq:backward_transfer_operator},
in which both layers are operated trains and $p_m=q_m=r_m$, and
$\Tfwd_m^{(\beta,\rho)}$, $\Tbwd_m^{(\beta,\rho)}$ for the mixed instance
$(\mathsf A_m,\mathsf B_m)=\bigl(\mathsf G_m^{(\beta)},\mathsf Y_m\bigr)$, in
which the second layer is the collocated residual train of
\eqref{eq:tt_residual_representation_gradient}, with $p_m=r_m$ and $(q_m)$
denoting its TT ranks. The adjoint identity
\eqref{eq:transfer_adjoint_identity} then applies to both.

The partial train matrices \eqref{eq:partial trains}  allow us to define the left and right interface matrices. For an operator pair $(\beta,\gamma)$,
we define
\begin{align}
\mathsf L_{<s}^{(\beta,\gamma)}
:=
\frac{1}{N_{<s}}
\bigl(\mathsf U_{<s}^{(\beta)}\bigr)^\top
\mathsf U_{<s}^{(\gamma)}
\in\R^{r_{s-1}\times r_{s-1}}, \qquad
\mathsf R_{>s}^{(\beta,\gamma)}
:=
\frac{1}{N_{>s}}
\bigl(\mathsf U_{>s}^{(\beta)}\bigr)^\top
\mathsf U_{>s}^{(\gamma)}
\in\R^{r_s\times r_s}.
\end{align}
Proposition~\ref{prop:tt_transfer_operators} gives the compact recursions
\begin{equation}\label{eq:TT_recursions_interfaces}
\begin{aligned}
\mathsf L_{<1}^{(\beta,\gamma)}
&=1, \quad
&
\mathsf L_{<m+1}^{(\beta,\gamma)}
&=
\Tfwd_m^{(\beta,\gamma)}
\bigl(\mathsf L_{<m}^{(\beta,\gamma)}\bigr),
&& m=1,\ldots,d-1,
\\
\mathsf R_{>d}^{(\beta,\gamma)}
&=1, \quad
&
\mathsf R_{>m-1}^{(\beta,\gamma)}
&=
\Tbwd_m^{(\beta,\gamma)}
\bigl(\mathsf R_{>m}^{(\beta,\gamma)}\bigr),
&& m=d,\ldots,2.
\end{aligned}
\end{equation}
Thus, one forward sweep computes all left interfaces, and one backward
sweep with the adjoint transfers computes all right interfaces. 

With this rather technical preparation in place, we are ready to characterize the operator-indexed blocks of the compressed Gramian
\eqref{eq:compressed_gramian_operator_blocks}.

\begin{theorem}[Algebraic characterization of the compressed TT Gramian]
\label{thm:tt_gramian_forward_backward}
For every operator pair $(\beta,\gamma)$, the following statements hold.

\begin{enumerate}
\item For every $s=1,\ldots,d$, we have
\begin{equation}\label{eq:tt_gramian_diagonal_final}
\Upsilon_{ss}^{(\beta,\gamma)}
=
\frac{1}{n_s}I_{n_s}\otimes
\left(
\mathsf L_{<s}^{(\beta,\gamma)}\otimes
\mathsf R_{>s}^{(\beta,\gamma)}
\right)
\in
\R^{n_sr_{s-1}r_s\times n_sr_{s-1}r_s}.
\end{equation}

\item Let $s<t$, and fix active samples $i_s,i_t$ and bond pairs
$(\alpha_{s-1},\alpha_s)$ and
$(\alpha'_{t-1},\alpha'_t)$. Define
\begin{align}
\mathsf Z_s
&:=
\frac{1}{n_s}
\mathsf M_{\alpha_{s-1}\alpha_s}^\top
\mathsf L_{<s}^{(\beta,\gamma)}
\mathsf G_s^{(\gamma)}(i_s)
\in\R^{r_s\times r_s},
\label{eq:bridge_initialization}
\\
\mathsf Z_m
&:=
\Tfwd_m^{(\beta,\gamma)}(\mathsf Z_{m-1})
\in\R^{r_m\times r_m},
\qquad s<m<t,
\label{eq:bridge_propagation}
\\
\mathsf Z_t
&:=
\frac{1}{n_t}
\bigl(\mathsf G_t^{(\beta)}(i_t)\bigr)^\top
\mathsf Z_{t-1}
\mathsf M_{\alpha'_{t-1}\alpha'_t}
\in\R^{r_t\times r_t}.
\label{eq:bridge_termination}
\end{align}
Then $\Upsilon_{st}^{(\beta,\gamma)}
\in
\R^{
n_sr_{s-1}r_s
\times
n_tr_{t-1}r_t}$ with the entries given by
\begin{equation}\label{eq:tt_gramian_offdiagonal_final}
\Upsilon_{st}^{(\beta,\gamma)}
\bigl(
\overline{i_s\alpha_{s-1}\alpha_s},
\overline{i_t\alpha'_{t-1}\alpha'_t}
\bigr)
=
\big\langle
\mathsf Z_t,
\mathsf R_{>t}^{(\beta,\gamma)}
\big\rangle_{\mathrm F},
\end{equation}
If $t=s+1$, the propagation step
\eqref{eq:bridge_propagation} is absent. The opposite block is obtained from
\begin{equation}\label{eq:gramian_block_symmetry}
\Upsilon_{ts}^{(\gamma,\beta)}
=
\bigl(\Upsilon_{st}^{(\beta,\gamma)}\bigr)^\top.
\end{equation}
\end{enumerate}
\end{theorem}

\begin{proof}

Using Remark~\ref{rem:active_ordering},
\eqref{eq:tt_E_matrix}, and the Kronecker identity, we obtain
\begin{align*}
\Upsilon_{ss}^{(\beta,\gamma)}
= \frac{1}{N_\D}
\bigl(\Psi_s^{(\beta)}\bigr)^\top
\Psi_s^{(\gamma)} &= \frac{1}{n_s}I_{n_s}\otimes
\left[
\frac{1}{N_{-s}}
\bigl(\mathsf E_s^{(\beta)}\bigr)^\top
\mathsf E_s^{(\gamma)}
\right]
\\
&=
\frac{1}{n_s}I_{n_s}\otimes
\left[
\frac{1}{N_{<s}N_{>s}}
\bigl(
\mathsf U_{<s}^{(\beta)}
\otimes
\mathsf U_{>s}^{(\beta)}
\bigr)^\top
\bigl(
\mathsf U_{<s}^{(\gamma)}
\otimes
\mathsf U_{>s}^{(\gamma)}
\bigr)
\right]
\\
&=
\frac{1}{n_s}I_{n_s}\otimes
\left(
\mathsf L_{<s}^{(\beta,\gamma)}
\otimes
\mathsf R_{>s}^{(\beta,\gamma)}
\right),
\end{align*}
which proves \eqref{eq:tt_gramian_diagonal_final}.

We proceed with the computation of the off-diagonal blocks \eqref{eq:tt_gramian_offdiagonal_final}. Without loss of generality, assume that $s<t$ and fix the two column indices
appearing in \eqref{eq:tt_gramian_offdiagonal_final}. Denote the
corresponding columns by
\[
\bm\psi_s^{(\beta)}
:=
\Psi_s^{(\beta)}
\bigl(:,\overline{i_s\alpha_{s-1}\alpha_s}\bigr),
\qquad
\bm\psi_t^{(\gamma)}
:=
\Psi_t^{(\gamma)}
\bigl(:,\overline{i_t\alpha_{t-1}'\alpha_t'}\bigr).
\]
By \eqref{eq:Psi_beta} and
\eqref{eq:tt_E_matrix_unit_representation}, their entries are given by
\begin{equation}\label{TT inactive columns}
\begin{aligned}
\bm\psi_s^{(\beta)}(\overline{j_1\ldots j_d})
&=
\mathsf G_1^{(\beta)}(j_1)\cdots
\mathsf G_{s-1}^{(\beta)}(j_{s-1})
\bigl[
\delta_{j_si_s}\,
\mathsf M_{\alpha_{s-1}\alpha_s}
\bigr]
\mathsf G_{s+1}^{(\beta)}(j_{s+1})\cdots
\mathsf G_d^{(\beta)}(j_d),
\\
\bm\psi_t^{(\gamma)}(\overline{j_1\ldots j_d})
&=
\mathsf G_1^{(\gamma)}(j_1)\cdots
\mathsf G_{t-1}^{(\gamma)}(j_{t-1})
\bigl[
\delta_{j_ti_t}\,
\mathsf M_{\alpha_{t-1}'\alpha_t'}
\bigr]
\mathsf G_{t+1}^{(\gamma)}(j_{t+1})\cdots
\mathsf G_d^{(\gamma)}(j_d).
\end{aligned}
\end{equation}
Consequently, an entry of the off-diagonal operator-indexed Gramian block is given by
\begin{align*}
&
\Upsilon_{st}^{(\beta,\gamma)}
\bigl(
\overline{i_s\alpha_{s-1}\alpha_s},
\overline{i_t\alpha_{t-1}'\alpha_t'}
\bigr)
=
\frac{1}{N_\D}
\bigl(\bm\psi_s^{(\beta)}\bigr)^\top
\bm\psi_t^{(\gamma)}
=
\frac{1}{N_\D}
\sum_{j_1=1}^{n_1}\cdots\sum_{j_d=1}^{n_d}
\bm\psi_s^{(\beta)}(\overline{j_1\ldots j_d})
\bm\psi_t^{(\gamma)}(\overline{j_1\ldots j_d}).
\end{align*}
Observe that this is nothing but the averaged double-layer contraction 
\eqref{eq:averaged_tt_recursion_result} of the two tensor trains
\eqref{TT inactive columns}. We evaluate it by following the
diagram in Fig.~\ref{Diag: Double-layer contraction diagram for the off-diagonal Gramian block} and contracting the two trains from left to right.

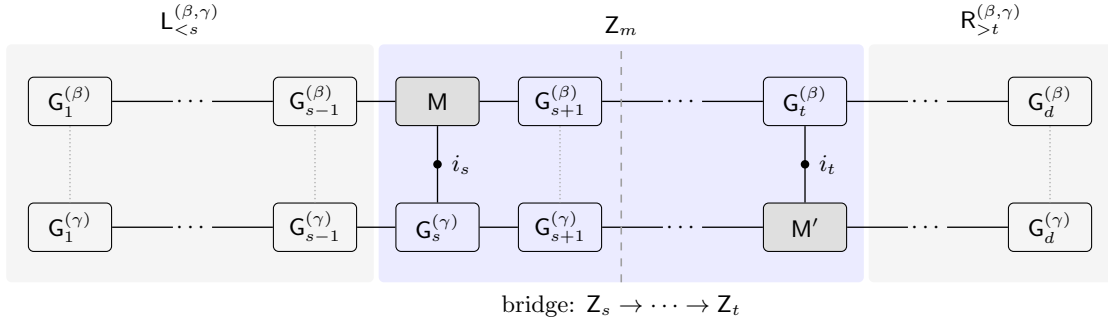
\begin{figure}[H]
\centering

\begin{tikzpicture}[
  x=1.62cm,y=1.30cm,>=Latex,
  core/.style={draw,rounded corners=2pt,minimum width=1.10cm,
               minimum height=0.64cm,inner sep=1pt,font=\footnotesize},
  mu/.style={core,fill=gray!25},
  dots/.style={font=\footnotesize,inner sep=1pt},
  bond/.style={line width=0.5pt},
  share/.style={densely dotted,gray!65,line width=0.5pt},
  cut/.style={dashed,gray!75,line width=0.6pt},
  reg/.style={fill=gray!8,rounded corners=2pt},
  bridge/.style={fill=blue!8,rounded corners=2pt},
  lab/.style={font=\footnotesize,inner sep=1pt}]

\fill[reg]    (0.48,-0.56) rectangle (3.48,1.86);
\fill[bridge] (3.52,-0.56) rectangle (7.48,1.86);
\fill[reg]    (7.52,-0.56) rectangle (9.52,1.86);

\node[core] (t1) at (1,1.28) {$\mathsf G_1^{(\beta)}$};
\node[dots] (t2) at (2,1.28) {$\cdots$};
\node[core] (t3) at (3,1.28) {$\mathsf G_{s-1}^{(\beta)}$};
\node[mu]   (t4) at (4,1.28) {$\mathsf M$};
\node[core] (t5) at (5,1.28) {$\mathsf G_{s+1}^{(\beta)}$};
\node[dots] (t6) at (6,1.28) {$\cdots$};
\node[core] (t7) at (7,1.28) {$\mathsf G_{t}^{(\beta)}$};
\node[dots] (t8) at (8,1.28) {$\cdots$};
\node[core] (t9) at (9,1.28) {$\mathsf G_{d}^{(\beta)}$};

\node[core] (b1) at (1,0) {$\mathsf G_1^{(\gamma)}$};
\node[dots] (b2) at (2,0) {$\cdots$};
\node[core] (b3) at (3,0) {$\mathsf G_{s-1}^{(\gamma)}$};
\node[core] (b4) at (4,0) {$\mathsf G_{s}^{(\gamma)}$};
\node[core] (b5) at (5,0) {$\mathsf G_{s+1}^{(\gamma)}$};
\node[dots] (b6) at (6,0) {$\cdots$};
\node[mu]   (b7) at (7,0) {$\mathsf M'$};
\node[dots] (b8) at (8,0) {$\cdots$};
\node[core] (b9) at (9,0) {$\mathsf G_{d}^{(\gamma)}$};

\foreach \a/\b in {1/2,2/3,3/4,4/5,5/6,6/7,7/8,8/9} {
  \draw[bond] (t\a.east)--(t\b.west);
  \draw[bond] (b\a.east)--(b\b.west);
}

\foreach \k in {1,3,5,9} {
  \draw[share] (t\k.south)--(b\k.north);
}

\foreach \k/\lbl in {4/{i_s},7/{i_t}} {
  \draw[bond] (t\k.south)--(b\k.north);
  \fill (\k,0.64) circle (1.6pt);
  \node[lab,anchor=west] at (\k+0.09,0.64) {$\lbl$};
}

\draw[cut] (5.5,-0.56)--(5.5,1.86);

\node[lab,anchor=south] at (1.98,1.92)
  {$\mathsf L_{<s}^{(\beta,\gamma)}$};

\node[lab,anchor=south] at (5.50,1.92)
  {$\mathsf Z_m$};

\node[lab,anchor=south] at (8.52,1.92)
  {$\mathsf R_{>t}^{(\beta,\gamma)}$};

\node[lab,anchor=north] at (5.50,-0.66)
  {bridge: $\mathsf Z_s\to\cdots\to\mathsf Z_t$};

\end{tikzpicture}
\caption{Double-layer contraction diagram for the off-diagonal Gramian block}
\label{Diag: Double-layer contraction diagram for the off-diagonal Gramian block}
\end{figure}

\medskip
\noindent\emph{Step 1:} We start by computing the averaged double-layer contraction of the left shaded region of the diagram. For $m<s$, the two layers contain the ordinary operated cores $\mathsf G_m^{(\beta)}(j_m)$ and $\mathsf G_m^{(\gamma)}(j_m)$, respectively. Repeated application of the forward
transfer operators \eqref{eq:forward_transfer_operator} yields
\begin{equation}\label{eq:offdiag_left_region}
\mathsf L_{<s}^{(\beta,\gamma)}
=
\Tfwd_{s-1}^{(\beta,\gamma)}
\circ\cdots\circ
\Tfwd_1^{(\beta,\gamma)}
\bigl(\mathsf L_{<1}^{(\beta,\gamma)}\bigr).
\end{equation}
As a result, the matrix $\mathsf L_{<s}^{(\beta,\gamma)}
\in\R^{r_{s-1}\times r_{s-1}}$ defined by
\eqref{eq:TT_recursions_interfaces} contains the double-layer contraction of
all coordinates strictly to the left of $s$.

\medskip
\noindent\emph{Step 2:} We initialize the bridge at the first active coordinate $s$. Here, the first
layer contains $\delta_{j_si_s}\mathsf M_{\alpha_{s-1}\alpha_s}$, while the
second contains $\mathsf G_s^{(\gamma)}(j_s)$. The Kronecker delta reduces
the local collocation average to the single contribution $j_s=i_s$.
Thus, \eqref{eq:pinned_update} yields
$\mathsf Z_s\in\R^{r_s\times r_s}$ in
\eqref{eq:bridge_initialization}.

\medskip
\noindent\emph{Step 3:} We now propagate through the bridge.
For every $s<m<t$, the first and second layers contain the ordinary operated
cores $\mathsf G_m^{(\beta)}(j_m)$ and
$\mathsf G_m^{(\gamma)}(j_m)$, respectively. Successive application of the forward transfer operators \eqref{eq:forward_transfer_operator} to these cores
propagates the bridge state $\mathsf Z_s$ from coordinate $s$ to coordinate
$m$, giving
\begin{equation}
\mathsf Z_m
=
\Tfwd_m^{(\beta,\gamma)}
\circ\cdots\circ
\Tfwd_{s+1}^{(\beta,\gamma)}
\bigl(\mathsf Z_s\bigr),
\qquad
s<m<t.
\end{equation}
which is obtained by iterating \eqref{eq:bridge_propagation}.

\medskip
\noindent\emph{Step 4:} We are ready to terminate the bridge at the second active coordinate $t$. Proceeding as
in Step~2, the first layer contains $\mathsf G_t^{(\beta)}(j_t)$, whereas
the second contains
$\delta_{j_ti_t}\mathsf M_{\alpha_{t-1}'\alpha_t'}$. The Kronecker delta
again reduces the local collocation average to the single contribution
$j_t=i_t$. Hence, \eqref{eq:pinned_update} yields
$\mathsf Z_t\in\R^{r_t\times r_t}$ defined by
\eqref{eq:bridge_termination}.

\medskip
\noindent\emph{Step 5:} Finally, we close the accumulated contraction against the right interface $\mathsf R_{>t}^{(\beta,\gamma)}$. Rather than propagating $\mathsf Z_t$
forward through the tail $m>t$, we contract the same tail backward using
\eqref{eq:transfer_adjoint_identity}. Repeated application of the adjoint
identity yields
\begin{align*}
\big\langle
\Tfwd_d^{(\beta,\gamma)}
\circ\cdots\circ
\Tfwd_{t+1}^{(\beta,\gamma)}(\mathsf Z_t),
\mathsf R_{>d}^{(\beta,\gamma)}
\big\rangle_{\mathrm F}=
\big \langle
\mathsf Z_t,
\Tbwd_{t+1}^{(\beta,\gamma)}
\circ\cdots\circ
\Tbwd_d^{(\beta,\gamma)}
\bigl(\mathsf R_{>d}^{(\beta,\gamma)}\bigr)
\big\rangle_{\mathrm F}=
\big\langle
\mathsf Z_t,
\mathsf R_{>t}^{(\beta,\gamma)}
\big\rangle_{\mathrm F}.
\end{align*}
which proves \eqref{eq:tt_gramian_offdiagonal_final}.

Observe that if $t=s+1$, the bridge propagation
\eqref{eq:bridge_propagation} is absent and the bridge initialized at $s$
is terminated immediately at $t$. Finally, recall our initial assumption $s<t$. If $s>t$, interchanging the two columns in the defining dot product between $\bm\psi_s^{(\beta)}$ and $\bm\psi_t^{(\gamma)}$ places the smaller active coordinate first and gives $\Upsilon_{st}^{(\beta,\gamma)}
=\bigl(\Upsilon_{ts}^{(\gamma,\beta)}\bigr)^\top$. Thus, the case $s>t$ follows from the case $s<t$, which proves
\eqref{eq:gramian_block_symmetry} and completes the characterization of the off-diagonal blocks.
\end{proof}

It remains to characterize the compressed residual
\eqref{eq:compressed_residual_operator_blocks} in the TT format, from which the gradient
expression follows by \eqref{eq:compressed_gradient}. The proof follows the
same strategy as that of the preceding theorem.
\begin{proposition}[TT representation of the compressed residual and gradient]
\label{prop:tt_compressed_gradient}
Suppose that the collocated residual admits the TT representation
\begin{equation}\label{eq:tt_residual_representation_gradient}
\rvec(\theta)(\overline{i_1\ldots i_d})
=
\mathsf Y_1(i_1)\cdots\mathsf Y_d(i_d),
\qquad
\mathsf Y_m(i_m)\in\R^{q_{m-1}\times q_m},
\qquad q_0=q_d=1.
\end{equation}
For every operator term $\beta$, define the mixed left and right interfaces
by
\begin{equation}\label{eq:tt_residual_interfaces}
\begin{aligned}
\mathsf L_{<1}^{(\beta,\rho)}
&=1,
&
\mathsf L_{<m+1}^{(\beta,\rho)}
&:=
\frac{1}{n_m}
\sum_{j_m=1}^{n_m}
\bigl(\mathsf G_m^{(\beta)}(j_m)\bigr)^\top
\mathsf L_{<m}^{(\beta,\rho)}
\mathsf Y_m(j_m),
\\
\mathsf R_{>d}^{(\beta,\rho)}
&=1,
&
\mathsf R_{>m-1}^{(\beta,\rho)}
&:=
\frac{1}{n_m}
\sum_{j_m=1}^{n_m}
\mathsf G_m^{(\beta)}(j_m)
\mathsf R_{>m}^{(\beta,\rho)}
\mathsf Y_m(j_m)^\top,
\end{aligned}
\end{equation}
where
$\mathsf L_{<s}^{(\beta,\rho)}\in
\R^{r_{s-1}\times q_{s-1}}$ and
$\mathsf R_{>s}^{(\beta,\rho)}\in
\R^{r_s\times q_s}$. For $i_s=1,\ldots,n_s$, set
\begin{equation}\label{eq:tt_gradient_local_matrix}
\mathsf P_s^{(\beta)}(i_s)
:=
\frac{1}{n_s}
\mathsf L_{<s}^{(\beta,\rho)}
\mathsf Y_s(i_s)
\bigl(\mathsf R_{>s}^{(\beta,\rho)}\bigr)^\top
\in\R^{r_{s-1}\times r_s}.
\end{equation}
Then the operator-indexed compressed residual satisfies
\begin{equation}\label{eq:tt_compressed_rhs}
\rho_{\mathrm c,s}^{(\beta)}
\bigl(\overline{i_s\alpha_{s-1}\alpha_s}\bigr)
=
\mathsf P_s^{(\beta)}(i_s)
(\alpha_{s-1},\alpha_s).
\end{equation}
\end{proposition}

\begin{proof}
Fix $s$, $\beta$, and a the row index
$\overline{i_s\alpha_{s-1}\alpha_s}$. From \eqref{eq:Psi_beta},
\eqref{eq:compressed_residual_operator_blocks}, and
\eqref{eq:tt_E_matrix_unit_representation},
\begin{align*}
\rho_{\mathrm c,s}^{(\beta)}
\bigl(\overline{i_s\alpha_{s-1}\alpha_s}\bigr)
&=
\frac{1}{N_\D}
\sum_{j_1=1}^{n_1}\cdots\sum_{j_d=1}^{n_d}
\Big[
\mathsf G_1^{(\beta)}(j_1)\cdots
 \\
&\mathsf G_{s-1}^{(\beta)}(j_{s-1})\delta_{j_si_s}\mathsf M_{\alpha_{s-1}\alpha_s}
\mathsf G_{s+1}^{(\beta)}(j_{s+1})\cdots
\mathsf G_d^{(\beta)}(j_d)
\Big]
\Big[
\mathsf Y_1(j_1)\cdots\mathsf Y_d(j_d)
\Big].
\end{align*}
Observe that this is an averaged double-layer TT contraction between the
operated train and the residual train. By Proposition~\ref{prop:tt_transfer_operators}, the mixed interface
recursions \eqref{eq:tt_residual_interfaces} satisfy the adjoint identity \eqref{eq:transfer_adjoint_identity}, which we use in  Step~3 below.

We evaluate the averaged double-layer TT contraction by following the diagram in 
Fig. \ref{Diag: Double-layer contraction diagram for the compressed residual} and contracting the two trains from left to right.

\begin{figure}[H]
\centering

\begin{tikzpicture}[
  x=1.68cm,y=1.32cm,>=Latex,
  core/.style={draw,rounded corners=2pt,minimum width=1.10cm,
               minimum height=0.64cm,inner sep=1pt,font=\footnotesize},
  mu/.style={core,fill=gray!25},
  dots/.style={font=\footnotesize,inner sep=1pt},
  bond/.style={line width=0.5pt},
  share/.style={densely dotted,gray!65,line width=0.5pt},
  pin/.style={line width=0.55pt},
  reg/.style={fill=gray!8,rounded corners=2pt},
  local/.style={fill=blue!8,rounded corners=2pt},
  lab/.style={font=\footnotesize,inner sep=1pt}
]

\fill[reg]   (0.48,-0.56) rectangle (3.48,1.86);
\fill[local] (3.52,-0.56) rectangle (4.48,1.86);
\fill[reg]   (4.52,-0.56) rectangle (7.52,1.86);

\node[core] (t1) at (1,1.28) {$\mathsf G_1^{(\beta)}$};
\node[dots] (t2) at (2,1.28) {$\cdots$};
\node[core] (t3) at (3,1.28) {$\mathsf G_{s-1}^{(\beta)}$};
\node[mu]   (t4) at (4,1.28) {$\mathsf M_{\alpha_{s-1}\alpha_s}$};
\node[core] (t5) at (5,1.28) {$\mathsf G_{s+1}^{(\beta)}$};
\node[dots] (t6) at (6,1.28) {$\cdots$};
\node[core] (t7) at (7,1.28) {$\mathsf G_d^{(\beta)}$};

\node[core] (b1) at (1,0) {$\mathsf Y_1$};
\node[dots] (b2) at (2,0) {$\cdots$};
\node[core] (b3) at (3,0) {$\mathsf Y_{s-1}$};
\node[core] (b4) at (4,0) {$\mathsf Y_s(i_s)$};
\node[core] (b5) at (5,0) {$\mathsf Y_{s+1}$};
\node[dots] (b6) at (6,0) {$\cdots$};
\node[core] (b7) at (7,0) {$\mathsf Y_d$};

\foreach \a/\b in {1/2,2/3,3/4,4/5,5/6,6/7} {
  \draw[bond] (t\a.east)--(t\b.west);
  \draw[bond] (b\a.east)--(b\b.west);
}

\foreach \k in {1,3,5,7} {
  \draw[share] (t\k.south)--(b\k.north);
}

\draw[pin] (t4.south)--(b4.north);
\fill (4,0.64) circle (1.6pt);
\node[lab,anchor=west] at (4.10,0.64) {$i_s$};

\node[lab,anchor=south] at (2.00,1.92)
  {$\mathsf L_{<s}^{(\beta,\rho)}$};

\node[lab,anchor=south] at (4.00,1.92)
  {active slot $s$};

\node[lab,anchor=south] at (6.00,1.92)
  {$\mathsf R_{>s}^{(\beta,\rho)}$};

\end{tikzpicture}
\caption{Double-layer contraction diagram for the compressed residual}
\label{Diag: Double-layer contraction diagram for the compressed residual}
\end{figure}
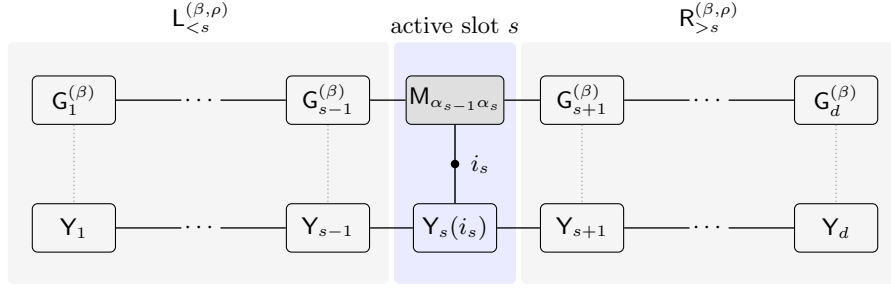

\medskip
\noindent\emph{Step 1.}
We start by computing the averaged double-layer contraction of the
coordinates strictly to the left of $s$. For $m<s$, the two layers contain
the ordinary cores
$\mathsf G_m^{(\beta)}(j_m)$ and $\mathsf Y_m(j_m)$, respectively.
Repeated application of the corresponding forward transfer operators yields
\begin{equation}\label{eq:gradient_left_region}
\mathsf L_{<s}^{(\beta,\rho)}
=
\Tfwd_{s-1}^{(\beta,\rho)}
\circ\cdots\circ
\Tfwd_1^{(\beta,\rho)}
\bigl(\mathsf L_{<1}^{(\beta,\rho)}\bigr).
\end{equation}

\medskip
\noindent\emph{Step 2.}
At the active coordinate $s$, the operated layer contains
$\delta_{j_si_s}\mathsf M_{\alpha_{s-1}\alpha_s}$, while the residual
layer contains $\mathsf Y_s(j_s)$. Formula \eqref{eq:pinned_update} then yields
\begin{equation}\label{eq:gradient_active_contraction}
\mathrm Z_{s}^{\rho}
=
\frac{1}{n_s}
\mathsf M_{\alpha_{s-1}\alpha_s}^{\top}
\mathsf L_{<s}^{(\beta,\rho)}
\mathsf Y_s(i_s).
\end{equation}

\medskip
\noindent\emph{Step 3.}
Finally, we close the accumulated contraction against the right interface
$\mathsf R_{>s}^{(\beta,\rho)}$. Instead of propagating
$\mathrm Z_s^\rho$ forward through the tail $m>s$, we contract this tail
backward. The repeated application of \eqref{eq:transfer_adjoint_identity}
identity yields
\begin{align*}
\big\langle
\Tfwd_d^{(\beta,\rho)}
\circ\cdots\circ
\Tfwd_{s+1}^{(\beta,\rho)}
\bigl(\mathrm Z_s^\rho\bigr),
\mathsf R_{>d}^{(\beta,\rho)}
\big\rangle_{\mathrm F}
=
\big\langle
\mathrm Z_s^\rho,
\Tbwd_{s+1}^{(\beta,\rho)}
\circ\cdots\circ
\Tbwd_d^{(\beta,\rho)}
\bigl(\mathsf R_{>d}^{(\beta,\rho)}\bigr)
\big\rangle_{\mathrm F}
=
\big\langle
\mathrm Z_s^\rho,
\mathsf R_{>s}^{(\beta,\rho)}
\big\rangle_{\mathrm F}.
\end{align*}
We complete the proof by computing 
\begin{align*}
\rho_{\mathrm c,s}^{(\beta)}
\bigl(\overline{i_s\alpha_{s-1}\alpha_s}\bigr)
=
\big\langle
\mathrm Z_s^\rho,
\mathsf R_{>s}^{(\beta,\rho)}
\big\rangle_{\mathrm F}=
\frac{1}{n_s}
\left[
\mathsf L_{<s}^{(\beta,\rho)}
\mathsf Y_s(i_s)
\bigl(\mathsf R_{>s}^{(\beta,\rho)}\bigr)^\top
\right]
(\alpha_{s-1},\alpha_s),
\end{align*}
where the last identity follows from
$\mathsf M_{\alpha_{s-1}\alpha_s}
=\mathrm e_{\alpha_{s-1}}\mathrm e_{\alpha_s}^{\top}$.
\end{proof}

\begin{remark}\label{rem:tt_residual_sum}
If the residual \eqref{eq:tt_residual_representation_gradient} is represented as a finite sum of TT trains, the construction
in Proposition~\ref{prop:tt_compressed_gradient} is applied to each train and the resulting compressed residuals are added. 
\end{remark}

Finally, we note that the loss in the tensor train format can also be
evaluated directly from the residual cores in
\eqref{eq:tt_residual_representation_gradient}. This amounts to an averaged double-layer TT
contraction of the residual train with itself. With
$\mathsf A_m=\mathsf B_m=\mathsf Y_m$, the recursion in
Lemma~\ref{lem:averaged_tt_recursion} becomes
\begin{equation}\label{eq:tt_loss_recursion}
\mathsf S_0:=1,
\qquad
\mathsf S_m
:=
\frac{1}{n_m}
\sum_{i_m=1}^{n_m}
\bigl(\mathsf Y_m(i_m)\bigr)^\top
\mathsf S_{m-1}\mathsf Y_m(i_m)
\in\R^{q_m\times q_m},
\qquad m=1,\ldots,d.
\end{equation}
Since $q_d=1$, the final contraction $\mathsf S_d$ is scalar, and
\eqref{eq:averaged_tt_recursion_result} yields
\begin{equation}\label{eq:tt_loss_contraction}
\widehat{\mathcal L}(\theta)
=
\frac{1}{2N_\D}\|\rvec(\theta)\|_2^2
=
\frac12\mathsf S_d.
\end{equation}
Algorithm~\ref{alg:tt_assembly} then summarizes the TT assembly of the
compressed Gauss--Newton system.

\begin{algorithm}[H]
\caption{TT assembly for the compressed Gauss--Newton step}
\label{alg:tt_assembly}
\small
\begin{algorithmic}[1]
\setlength{\itemsep}{0.22em}
\REQUIRE Current parameters $\theta$; factor grids $\X_s$; TT ranks
$(r_0,\ldots,r_d)$; univariate operators $\A_s^{(\beta)}$ in the
separated representation of $\A$; source data admitting the residual
representation \eqref{eq:tt_residual_representation_gradient}
\STATE Evaluate $\mathsf G_s^{(\beta)}$ and obtain
$\Cmat_s^{(\beta)}$ for all $s$
and $\beta$ by automatic differentiation using
\eqref{eq:tt_collocated_objects}--\eqref{eq:tt_Cmat}.
\STATE $\Cmat_s\leftarrow\operatorname{col}_{\beta=1}^{R_\A}
\Cmat_s^{(\beta)}$, \quad
$N_s\leftarrow R_\A n_s r_{s-1}r_s$, for all $s$.
\STATE Form the residual cores $\mathsf Y_m$ from the operated trains
and prescribed source data.
\FOR{$\beta,\gamma=1,\ldots,R_\A$}
  \STATE Compute all $\mathsf L_{<s}^{(\beta,\gamma)}$ and
  $\mathsf R_{>s}^{(\beta,\gamma)}$ by the two sweeps
  \eqref{eq:TT_recursions_interfaces}.
  \STATE Assemble the diagonal blocks using
  \eqref{eq:tt_gramian_diagonal_final}.
  \FOR{$s=1,\ldots,d-1$ and each $(i_s,a_s)$}
    \STATE Initialize $\mathsf Z\leftarrow
    n_s^{-1}\mathsf M_{a_s}^{\top}
    \mathsf L_{<s}^{(\beta,\gamma)}
    \mathsf G_s^{(\gamma)}(i_s)$.
    \FOR{$t=s+1,\ldots,d$}
      \STATE For every $(i_t,a_t')$, set
      $\displaystyle
      \Upsilon_{st}^{(\beta,\gamma)}
      (\overline{i_sa_s},\overline{i_ta_t'})
      \leftarrow
      \frac{1}{n_t}
      \big\langle
      (\mathsf G_t^{(\beta)}(i_t))^\top\mathsf Z\mathsf M_{a_t'},
      \mathsf R_{>t}^{(\beta,\gamma)}
      \big\rangle_{\mathrm F}$.
      \IF{$t<d$}
        \STATE $\mathsf Z\leftarrow
        \Tfwd_t^{(\beta,\gamma)}(\mathsf Z)$.
      \ENDIF
    \ENDFOR
  \ENDFOR
  \STATE Fill the opposite blocks by
  $\Upsilon_{ts}^{(\gamma,\beta)}
  \leftarrow(\Upsilon_{st}^{(\beta,\gamma)})^\top$ for $s<t$. 
\ENDFOR
\FOR{$\beta=1,\ldots,R_\A$}
  \STATE Compute the mixed interfaces \eqref{eq:tt_residual_interfaces},
  then $\mathsf P_s^{(\beta)}(i_s)$ and
  $\rho_{\mathrm c,s}^{(\beta)}$ using
  \eqref{eq:tt_gradient_local_matrix}--\eqref{eq:tt_compressed_rhs}.
\ENDFOR
\STATE Stack $\rho_{\mathrm c,s}^{(\beta)}$ and the Gramian blocks
in the order $(s,\beta,i_s,a_s)$ to obtain
$\rho_{\mathrm c}$ and $\Upsilon$.
\STATE Compute
$\nabla_{\theta^s}\widehat{\mathcal L}(\theta)
\leftarrow\sum_{\beta=1}^{R_\A}
\bigl(\Cmat_s^{(\beta)}\bigr)^\top
\rho_{\mathrm c,s}^{(\beta)}$,
for $s=1,\ldots,d$.
\STATE Stack the coordinate gradients to obtain
$\nabla_\theta\widehat{\mathcal L}(\theta)
\leftarrow\operatorname{col}_{s=1}^d
\nabla_{\theta^s}\widehat{\mathcal L}(\theta)$.
\STATE Evaluate $\widehat{\mathcal L}(\theta)$ using
\eqref{eq:tt_loss_recursion}--\eqref{eq:tt_loss_contraction}.
\ENSURE $\{\Cmat_s\}_{s=1}^d$, $\Upsilon$, $\rho_{\mathrm c}$,
$\nabla_\theta\widehat{\mathcal L}(\theta)$,
and $\widehat{\mathcal L}(\theta)$, with
$\Neff=R_\A\sum_s n_sr_{s-1}r_s$
\end{algorithmic}
\end{algorithm}

\section{Examples and numerical results}\label{sec:numerics}
\phantom{.}
In this section, we consider three test problems: a high-dimensional
Poisson problem with large canonical rank but bounded TT rank, illustrating
the advantage of the TT format over the CP format; a high-dimensional
diffusion benchmark \cite{jnini2025dual}, which also includes softly imposed
boundary and initial conditions treated here within the TT-compressed
setting; and two high-dimensional parametric Darcy problems with coefficients
of varying parametric complexity.

All considered problems fall within the class characterized by
Proposition~\ref{prop:lc_scope_finite_order}. Accordingly, for the operator
characterizations, we choose univariate spaces $H_s$ for which the derivative
maps required by Proposition~\ref{prop:lc_scope_finite_order} are bounded.
In addition, Proposition~\ref{prop:lc_scope_finite_order} renders
Assumption~\ref{ass:local_differential_channels} applicable, so that
Theorem~\ref{prop:local_channel_compression} can be used for the algebraic
characterization of the corresponding compressed matrices.

For the derivative orders occurring below, the sensitivity matrices
\eqref{eq:local_channel_sensitivities} specialize to
\begin{equation}\label{eq:numerical_sensitivity_matrices}
\mathsf H_{s,\ell}\bigl(\overline{i_sa_s},p\bigr)
:=
\begin{cases}
\bigl(\partial_{z_s}^{\ell}\partial_{\theta_p^s}u_s^{a_s}\bigr)(z_{s,i_s}),
&\text{CP},\\[1mm]
\bigl(\partial_{z_s}^{\ell}\partial_{\theta_p^s}G_s\bigr)(z_{s,i_s})
(\alpha_{s-1},\alpha_s),&\text{TT},
\end{cases}
\qquad
\mathsf H_{s,\ell}\in\R^{n_sk_s\times P_s},
\end{equation}
where $k_s=R$ and $a_s=\alpha$ for CP, while
$k_s=r_{s-1}r_s$ and
$a_s=\overline{\alpha_{s-1}\alpha_s}$ for TT.
For each problem, we provide an explicit description of the resulting matrices, both for implementation and to illustrate the abstract theory in concrete settings.

In our numerical experiments, we assess the approximation quality using either the relative $L^2$ or $H^1$ errors
\begin{equation*}
e_{L^2}
:=
\frac{\|u_\theta-u^\star\|_{L^2(\mathcal{D})}}
{\|u^\star\|_{L^2(\mathcal{D})}},
\qquad
e_{H^1}
:=
\left(
\frac{
\|u_\theta-u^\star\|_{L^2(\mathcal{D})}^2
+
\|\nabla_x u_\theta-\nabla_x u^\star\|_{L^2(\mathcal{D})}^2
}{
\|u^\star\|_{L^2(\mathcal{D})}^2
+
\|\nabla_x u^\star\|_{L^2(\mathcal{D})}^2
}
\right)^{1/2},
\end{equation*}
where $\nabla_x$ denotes differentiation with respect to the spatial variables. For all problems, we construct an analytical solution $u^\star$ admitting a
finite separable or low-rank tensor representation. The respective errors are
computed using fixed tensor-product midpoint rules, with the required integrals
evaluated through successive one-dimensional contractions of the separable
factors or TT cores, thereby avoiding the evaluation of either $u^\star$ or
$u_\theta$ on the full tensor-product grid.

Unless otherwise specified, homogeneous Dirichlet conditions are imposed as hard constraints
\begin{align}\label{eq:SDF num}
\ell_{\rm bc}(x)
=\sqrt{30}\,x(1-x), \qquad
G_m(x_m;\theta^m)
=
\ell_{\rm bc}(x_m)\widetilde G_m(x_m;\theta^m),
\qquad m=1,\ldots,d_x.
\end{align}
The above boundary cutoff function satisfies \eqref{eq:separable_cutoff}. All networks use $\tanh$ activations in the hidden layers and a linear output layer. At each iteration of the Gauss--Newton algorithm, we use the same logarithmic line search with $31$ positive candidates ranging from $1$ to $10^{-6}$. The experiments were performed on an NVIDIA A100 GPU with 80 GB of VRAM using 64-bit floating-point precision.

\subsection{High-dimensional Poisson problem}
Let $\D=\Omega=(0,1)^{\dsp}$, where $d=\dsp=25$. We consider
\begin{equation}\label{eq:numerical_poisson}
-\Delta u_K^\star=f_K\quad\text{in }\D,
\qquad u_K^\star=0\quad\text{on }\partial\D.
\end{equation}
According to Proposition \ref{prop:poisson_separable_structure}, the Laplace operator satisfies Assumption~\ref{ass:operator}. Here, Proposition~\ref{prop:lc_scope_finite_order} also
applies, and for every active coordinate $s$ we obtain
\begin{equation}\label{eq:poisson_local_channels}
\Lambda_s=\{0,2\},
\qquad
\mathcal D_{s,0}=I,
\qquad
\mathcal D_{s,2}=\partial_{z_s}^2,
\qquad
b_{s,0}^{(\beta)}=1-\delta_{s\beta},
\qquad
b_{s,2}^{(\beta)}=-\delta_{s\beta}.
\end{equation}
Hence Assumption~\ref{ass:local_differential_channels} holds with
$L_s=2$. To compare CP and TT formats, we construct exact solutions of
\eqref{eq:numerical_poisson} with increasing canonical rank but uniformly
bounded maximal TT rank.

\begin{proposition}
\label{prop:paired_target_ranks}
Let $K=2^q$ with $q\geq 1$ and $2q\leq d$, and consider the function
\begin{equation}\label{eq:paired_target_rank_prop}
u_K^\star(z)
=
\frac{1}{\sqrt K}
\prod_{\ell=1}^{q}
\left[
\varphi_1(z_{2\ell-1})\varphi_1(z_{2\ell})
+
\varphi_2(z_{2\ell-1})\varphi_2(z_{2\ell})
\right]
\prod_{m=2q+1}^{d}\varphi_1(z_m),
\end{equation}
where $\varphi_1$ and $\varphi_2$ are linearly independent and orthogonal in
$L^2(0,1)$. Then 
\begin{equation}\label{eq:paired_target_cp_rank}
\operatorname{rank}_{\mathrm{CP}}(u_K^\star)=K.
\end{equation}
Moreover, $u_K^\star$ admits a TT representation with TT ranks
$(r_m)_{m=0}^{d}$ satisfying
\begin{equation}\label{eq:paired_target_tt_upper}
\max_{m=1,\ldots,d-1} r_m\leq 2.
\end{equation}
\end{proposition}

\begin{proof}
We first characterize the canonical rank. Expanding the brackets in
\eqref{eq:paired_target_rank_prop} yields
\begin{equation}\label{eq:paired_target_cp_expansion}
u_K^\star(z)
=
\frac{1}{\sqrt K}
\sum_{\bm\nu\in\{1,2\}^{q}}
\left[
\prod_{\ell=1}^{q}
\varphi_{\nu_\ell}(z_{2\ell-1})
\varphi_{\nu_\ell}(z_{2\ell})
\right]
\prod_{m=2q+1}^{d}\varphi_1(z_m).
\end{equation}
Since $\{1,2\}^{q}$ contains $2^q=K$ elements,
\eqref{eq:paired_target_cp_expansion} is a canonical representation with
$K$ rank-one terms. Therefore, we get the upper bound 
\begin{equation}\label{eq:paired_target_cp_upper}
\operatorname{rank}_{\mathrm{CP}}(u_K^\star)\leq K.
\end{equation}
Let $z_{\mathrm o}:=(z_1,z_3,\ldots,z_{2q-1})$ and $z_{\mathrm r}:=(z_2,z_4,\ldots,z_{2q},z_{2q+1},\ldots,z_d)$ and, for $\bm\nu=(\nu_1,\ldots,\nu_q)\in\{1,2\}^{q}$, define
\begin{align*}
v_{\bm\nu}(z_{\mathrm o})
:=
\prod_{\ell=1}^{q}
\varphi_{\nu_\ell}(z_{2\ell-1}), \qquad
w_{\bm\nu}(z_{\mathrm r})
:=
\left[
\prod_{\ell=1}^{q}
\varphi_{\nu_\ell}(z_{2\ell})
\right]
\prod_{m=2q+1}^{d}\varphi_1(z_m).
\end{align*}
Therefore, setting $\D_{\mathrm o}
:=
\prod_{\ell=1}^{q}\D_{2\ell-1}$ and $
\D_{\mathrm r}
:=
\big(\prod_{\ell=1}^{q}\D_{2\ell} \big)
\times
\big(\prod_{m=2q+1}^{d}\D_m\big)$ such that $v_{\bm\nu}\in L^2(\D_{\mathrm o})$ and
$w_{\bm\nu}\in L^2(\D_{\mathrm r})$, \eqref{eq:paired_target_cp_expansion} can be written in the bipartite tensor form
\begin{equation*}
u_K^\star
=
\frac{1}{\sqrt K}
\sum_{\bm\nu\in\{1,2\}^{q}}
v_{\bm\nu}\otimes w_{\bm\nu}.
\end{equation*}
This corresponds to the matricization obtained by grouping the odd
coordinates against the remaining ones in the sense of
\cite[Def.~5.3]{hackbusch2012tensor}. Since $\varphi_1$ and $\varphi_2$
are orthogonal, both families
$\{v_{\bm\nu}\}_{\bm\nu\in\{1,2\}^{q}}$ and
$\{w_{\bm\nu}\}_{\bm\nu\in\{1,2\}^{q}}$ consist of $K=2^q$ mutually
orthogonal, and hence linearly independent, functions. Therefore, the
matricization has matrix rank $K$; cf.
\cite[Def.~5.7]{hackbusch2012tensor}. The bound between the matricization rank and the canonical tensor rank
\cite[Rem.~6.24]{hackbusch2012tensor} yields 
\begin{align*}
K \leq \operatorname{rank}_{\mathrm{CP}}(u_K^\star).
\end{align*}
Combined with \eqref{eq:paired_target_cp_upper}, it proves \eqref{eq:paired_target_cp_rank}

We proceed by constructing a TT representation of \eqref{eq:paired_target_cp_expansion}. First, set 
\begin{equation}\label{eq:paired_target_tt_ranks}
r_0=r_d=1,
\qquad
r_{2\ell-1}=2,
\qquad
r_{2\ell}=1,
\qquad
\ell=1,\ldots,q,
\end{equation}
and $r_m=1$ for $m>2q$. For the first pair $(z_1,z_2)$, we define
\begin{equation*}
G_1(z_1)
=
\frac{1}{\sqrt K}
\begin{bmatrix}
\varphi_1(z_1) & \varphi_2(z_1)
\end{bmatrix},
\qquad
G_2(z_2)
=
\begin{bmatrix}
\varphi_1(z_2)\\
\varphi_2(z_2)
\end{bmatrix}.
\end{equation*}
For $\ell=2,\ldots,q$, set
\begin{equation*}
G_{2\ell-1}(z_{2\ell-1})
=
\begin{bmatrix}
\varphi_1(z_{2\ell-1}) &
\varphi_2(z_{2\ell-1})
\end{bmatrix},
\qquad
G_{2\ell}(z_{2\ell})
=
\begin{bmatrix}
\varphi_1(z_{2\ell})\\
\varphi_2(z_{2\ell})
\end{bmatrix},
\end{equation*}
while for the remaining coordinates we use the scalar cores
\begin{equation*}
G_m(z_m)=\varphi_1(z_m),
\qquad
m=2q+1,\ldots,d.
\end{equation*}
For each coordinate pair $(z_{2\ell-1},z_{2\ell})$, $\ell=2,\ldots,q$, we obtain
\begin{equation*}
G_{2\ell-1}(z_{2\ell-1})G_{2\ell}(z_{2\ell})
=
\varphi_1(z_{2\ell-1})\varphi_1(z_{2\ell})
+
\varphi_2(z_{2\ell-1})\varphi_2(z_{2\ell}).
\end{equation*}
Multiplying the resulting cores successively yields
\eqref{eq:paired_target_rank_prop}. Since all intermediate dimensions in
\eqref{eq:paired_target_tt_ranks} are either one or two, this provides a TT representation satisfying \eqref{eq:paired_target_tt_upper}
\end{proof}

In the next proposition, we provide the matrices required for the compressed Gauss--Newton method in both the CP and TT formats.
\begin{proposition}[Local-channel matrices for Poisson]
\label{prop:poisson_channel_matrices}
For the Poisson channels \eqref{eq:poisson_local_channels},
Theorem~\ref{prop:local_channel_compression} admits the following explicit
realizations for the CP and TT parametrizations.
Using 
\eqref{eq:numerical_sensitivity_matrices}, define
\begin{equation}\label{eq:poisson_channel_stacking}
\Psi_s^{\mathrm{ch}}
:=
\begin{bmatrix}
\Psi_{s,0}^{\mathrm{ch}}&
\Psi_{s,2}^{\mathrm{ch}}
\end{bmatrix},
\qquad
\Cmat_s^{\mathrm{ch}}
:=
\operatorname{col}(\mathsf H_{s,0},\mathsf H_{s,2}),
\qquad
N_s^{\mathrm{ch}}=2n_sk_s.
\end{equation}

\begin{enumerate}[label=\textup{(\roman*)}]

\item
For the CP format, $k_s=R$, and the direction matrices are
\begin{equation}\label{eq:poisson_cp_direction_columns}
\begin{aligned}
\Psi_{s,0}^{\mathrm{ch}}
\bigl(i,\overline{j_s\alpha}\bigr)
&=
-\delta_{i_sj_s}
\sum_{\beta\ne s}
(u_\beta^\alpha)''(z_{\beta,i_\beta})
\prod_{m\ne s,\beta}
u_m^\alpha(z_{m,i_m}),
\\
\Psi_{s,2}^{\mathrm{ch}}
\bigl(i,\overline{j_s\alpha}\bigr)
&=
-\delta_{i_sj_s}
\prod_{m\ne s}
u_m^\alpha(z_{m,i_m}),
\end{aligned}
\end{equation}
with
$\Psi_{s,\ell}^{\mathrm{ch}}\in\R^{N_\D\times n_sR}$.

\item
For the TT format, $k_s=r_{s-1}r_s$ and
$a_s=\overline{\alpha_{s-1}\alpha_s}$.
Let $\mathsf M_{a_s}$ be the corresponding matrix unit
\eqref{eq:tt_matrix_units}, and define
\begin{equation*}
\mathcal B_m^\Delta
:=
\begin{bmatrix}
G_m&-G_m''\\
0&G_m
\end{bmatrix}
\in\R^{2r_{m-1}\times2r_m},
\qquad
\mathsf P_{s,0,a_s}^{\Delta}
:=
\begin{bmatrix}
\mathsf M_{a_s}&0\\
0&\mathsf M_{a_s}
\end{bmatrix},
\qquad
\mathsf P_{s,2,a_s}^{\Delta}
:=
\begin{bmatrix}
0&-\mathsf M_{a_s}\\
0&0
\end{bmatrix}.
\end{equation*}
Furthermore,
$\mathcal Q_{<s}^{\Delta}\in\R^{1\times2r_{s-1}}$ and
$\mathcal Q_{>s}^{\Delta}\in\R^{2r_s\times1}$ are defined by
\begin{equation}\label{eq:poisson_partial_contractions}
\mathcal Q_{<s}^{\Delta}
:=
\begin{bmatrix}1&0\end{bmatrix}
\prod_{m<s}\mathcal B_m^\Delta,
\qquad
\mathcal Q_{>s}^{\Delta}
:=
\left(
\prod_{m>s}\mathcal B_m^\Delta
\right)
\begin{bmatrix}0\\1\end{bmatrix}.
\end{equation}
For $\ell\in\{0,2\}$, the direction matrices are
\begin{equation}\label{eq:poisson_tt_direction_columns}
\Psi_{s,\ell}^{\mathrm{ch}}
\bigl(i,\overline{j_sa_s}\bigr)
=
\delta_{i_sj_s}\,
\mathcal Q_{<s}^{\Delta}(z_{<s,i_{<s}})
\mathsf P_{s,\ell,a_s}^{\Delta}
\mathcal Q_{>s}^{\Delta}(z_{>s,i_{>s}}),
\end{equation}
with
$\Psi_{s,\ell}^{\mathrm{ch}}\in\R^{N_\D\times n_sk_s}$.
\end{enumerate}
Consequently, the respective effective compressed dimensions are given by
\begin{equation}\label{eq:poisson_local_channel_dimensions}
\Neff^{\mathrm{ch,CP}}
=
2R\sum_{s=1}^{d}n_s,
\qquad
\Neff^{\mathrm{ch,TT}}
=
2\sum_{s=1}^{d}n_sr_{s-1}r_s.
\end{equation}
\end{proposition}

\begin{proof}
For CP, the $\beta$-th separated operator term acts as
$-\partial_{z_\beta}^2$ in coordinate $\beta$ and as the identity in
all other coordinates. By \eqref{eq:cp_F_note} we therefore obtain
\begin{equation*}
\mathsf F_m^{(\beta)}(i_m,\alpha)
=
\begin{cases}
-(u_\beta^\alpha)''(z_{\beta,i_\beta}),&m=\beta,\\
u_m^\alpha(z_{m,i_m}),&m\ne\beta.
\end{cases}
\end{equation*}
Using the Khatri--Rao representation \eqref{eq:cp_E_note}, the inactive
contractions satisfy
\begin{equation}\label{eq:poisson_cp_inactive_contractions}
\mathsf E_s^{(\beta)}
\bigl(\overline{i_{-s}},\alpha\bigr)
=
\begin{cases}
\displaystyle
-(u_\beta^\alpha)''(z_{\beta,i_\beta})
\prod_{m\ne s,\beta}u_m^\alpha(z_{m,i_m}),
&\beta\ne s,\\[3mm]
\displaystyle
\prod_{m\ne s}u_m^\alpha(z_{m,i_m}),
&\beta=s.
\end{cases}
\end{equation}
Substituting \eqref{eq:poisson_cp_inactive_contractions} and
\eqref{eq:poisson_local_channels} into
\eqref{eq:local_channel_direction_columns} gives
\begin{align*}
\Psi_{s,0}^{\mathrm{ch}}
\bigl(i,\overline{j_s\alpha}\bigr)
&=
\delta_{i_sj_s}
\sum_{\beta=1}^{d}
(1-\delta_{s\beta})
\mathsf E_s^{(\beta)}
\bigl(\overline{i_{-s}},\alpha\bigr)
=
-\delta_{i_sj_s}
\sum_{\beta\ne s}
(u_\beta^\alpha)''(z_{\beta,i_\beta})
\prod_{m\ne s,\beta}
u_m^\alpha(z_{m,i_m}),
\\
\Psi_{s,2}^{\mathrm{ch}}
\bigl(i,\overline{j_s\alpha}\bigr) &=
-\delta_{i_sj_s}
\sum_{\beta=1}^{d}
\delta_{s\beta}
\mathsf E_s^{(\beta)}
\bigl(\overline{i_{-s}},\alpha\bigr)
=
-\delta_{i_sj_s}
\prod_{m\ne s}u_m^\alpha(z_{m,i_m}),
\end{align*}
which proves \eqref{eq:poisson_cp_direction_columns}.

For the TT format, ordinary block multiplication gives
\begin{equation*}
\begin{bmatrix}1&0\end{bmatrix}
\mathcal B_1^\Delta\cdots\mathcal B_d^\Delta
\begin{bmatrix}0\\1\end{bmatrix}
=
-\sum_{\beta=1}^{d}
G_1\cdots G_{\beta-1}G_\beta''G_{\beta+1}\cdots G_d
=
-\Delta u_\theta.
\end{equation*}
For the zeroth-order channel, $\beta\ne s$, and splitting into
$\beta<s$ and $\beta>s$ yields
\begin{align*}
\mathcal Q_{<s}^{\Delta}
\mathsf P_{s,0,a_s}^{\Delta}
\mathcal Q_{>s}^{\Delta}
&=
-\sum_{\beta<s}
G_1\cdots G_{\beta-1}G_\beta''G_{\beta+1}\cdots
G_{s-1}\mathsf M_{a_s}G_{s+1}\cdots G_d
\\
&\quad
-\sum_{\beta>s}
G_1\cdots G_{s-1}\mathsf M_{a_s}
G_{s+1}\cdots G_{\beta-1}G_\beta''G_{\beta+1}\cdots G_d.
\end{align*}
This yields $\Psi_{s,0}^{\mathrm{ch}}$ associated with
$\mathsf H_{s,0}$. For $\beta=s$, we have
\begin{equation*}
\mathcal Q_{<s}^{\Delta}
\mathsf P_{s,2,a_s}^{\Delta}
\mathcal Q_{>s}^{\Delta}
=
-
G_1\cdots G_{s-1}\mathsf M_{a_s}G_{s+1}\cdots G_d,
\end{equation*}
which yields $\Psi_{s,2}^{\mathrm{ch}}$ associated with
$\mathsf H_{s,2}$. Thus \eqref{eq:poisson_tt_direction_columns}
coincides with \eqref{eq:local_channel_direction_columns} for
$b_{s,0}^{(\beta)}=1-\delta_{s\beta}$ and
$b_{s,2}^{(\beta)}=-\delta_{s\beta}$. 

The effective compressed dimension count for both cases follows from Theorem~\ref{prop:local_channel_compression}.
\end{proof}

Let $\varphi_1(x)=\sin(\pi x)$ and $\varphi_2(x)=\sin(2\pi x)$ in
\eqref{eq:paired_target_rank_prop}, and set $f_K:=-\Delta u_K^\star$ in
\eqref{eq:numerical_poisson}. We conduct numerical experiments for
$K\in\{4,8,16,32\}$ using two rank configurations: a CP model of rank $R=4$
together with a TT model of boundary ranks $r_0=r_d=1$ and uniform interior rank
$r=2$, and a second configuration with $R=9$ and $r=3$. Each univariate factor
or core is a fully connected $\tanh$ network $1\to64\to64\to k_s$, with $k_s=R$
in the CP and $k_s=r_{s-1}r_s$ in the TT case. Per coordinate we draw $n=32$
collocation points independently from the uniform distribution, fixed for the
whole run. Each model is trained for $1000$ damped Gauss--Newton iterations in
the compressed space with $\mu=10^{-5}$.

\begin{figure}
    \centering
    \includegraphics[width=0.7\linewidth]{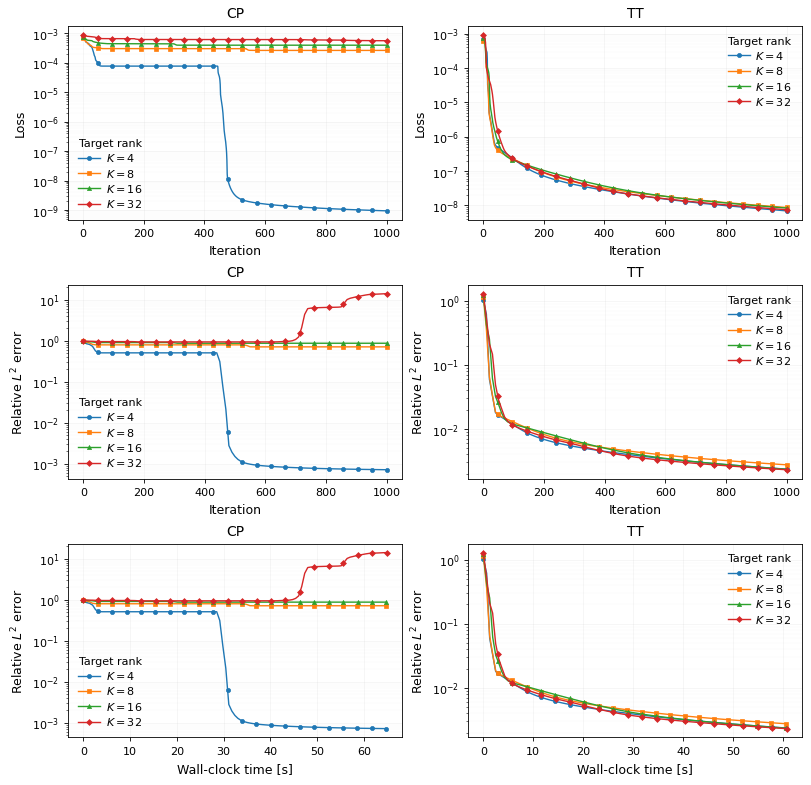}
    \caption{High-dimensional Poisson problem: CP and TT compressed Gauss--Newton convergence for canonical rank $R=4$ and TT rank $r=2$}
    \label{fig:ablation_poisson_1}
\end{figure}

The experiments in Fig.~\ref{fig:ablation_poisson_1} and
Fig.~\ref{fig:ablation_poisson_2} reflect the different rank requirements
established in Proposition~\ref{prop:paired_target_ranks}: the target $u_K^\star$
has canonical rank exactly $K$ by \eqref{eq:paired_target_cp_rank}, while its TT
ranks are bounded by two uniformly in $K$, see \eqref{eq:paired_target_tt_upper}.
Accordingly, in Fig.~\ref{fig:ablation_poisson_1} the CP model with $R=4$
accurately resolves only the $K=4$ target, while for $K>4$, where $R<K$, the
error stays at its initial level. In Fig.~\ref{fig:ablation_poisson_2},
increasing the CP rank to $R=9$ additionally allows the model to resolve $K=8$,
whereas the errors for $K=16$ and $K=32$ remain large, as expected. In contrast, and in agreement with
\eqref{eq:paired_target_tt_upper}, the TT model converges to relative errors of
about $2\cdot10^{-3}$ for all $K$ already with $r=2$, and increasing the TT rank
to $r=3$ yields no significant gain in accuracy while increasing the
computational time by roughly a factor of four. We emphasize that the $r=2$
configuration is also very fast: the full $1000$ Gauss--Newton iterations take
about $60$ seconds. Observe that most of the accuracy is reached
much earlier, the relative error falling below $10^{-2}$ within the first
$5$ seconds.

\begin{figure}[H]
    \centering
    \includegraphics[width=0.7\linewidth]{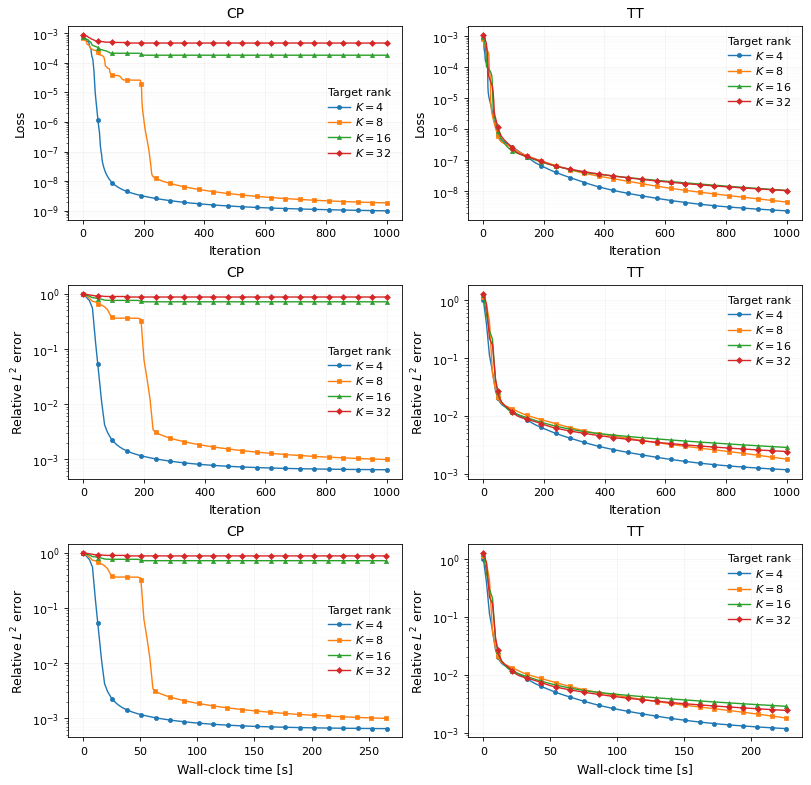}
    \caption{High-dimensional Poisson problem: CP and TT compressed Gauss--Newton convergence for canonical rank $R=9$ and TT rank $r=3$.}
    \label{fig:ablation_poisson_2}
\end{figure}


\subsection{The $10+1$-dimensional heat equation}

We consider the heat equation on $\Omega=(0,1)^{10}$ and $t\in(0,1)$,
with diffusion coefficient $\kappa=\tfrac14$:
\begin{equation*}
\partial_tu-\kappa\Delta_xu=0,
\qquad
u(x,0)=\sum_{m=1}^{10}\sin(2\pi x_m).
\end{equation*}
The Dirichlet data are given by the trace of the analytical solution
\begin{equation*}
u^\star(x,t)
=
e^{-4\pi^2\kappa t}
\sum_{m=1}^{10}\sin(2\pi x_m).
\end{equation*}
For the strong formulation, we use the following function space setting
\begin{align*}
U
&:=
L^2\bigl((0,1);H^2(\Omega)\bigr)
\cap
H^1\bigl((0,1);L^2(\Omega)\bigr),
\qquad
H
:=
\left(\bigotimes_{m=1}^{10}H^2(0,1)\right)
\otimes H^1(0,1),
\end{align*}
equipped with the norms
\begin{align*}
\|v\|_U^2
&:=
\|v\|_{L^2((0,1);H^2(\Omega))}^2
+
\|\partial_t v\|_{L^2(\Omega\times(0,1))}^2, \quad
\|v\|_H^2
:=
\sum_{\alpha\in\{0,1,2\}^{10}}
\sum_{j=0}^{1}
\|D_x^\alpha\partial_t^j v\|_{L^2(\Omega\times(0,1))}^2.
\end{align*}
The continuous embedding $H\hookrightarrow U$ and its density, and hence $H\overset{\mathrm d}{\hookrightarrow}U$, follow along the lines of the arguments given in the abstract section.

Writing $z=(x_1,\ldots,x_{10},t)$ and $d=11$, set
\begin{equation*}
(\ell_s,\gamma_s)
:=
\begin{cases}
(2,-\kappa),&s\le10,\\
(1,1),&s=11.
\end{cases}
\end{equation*}
Then the heat operator $\A:U\to L^2(\Omega\times(0,1))$ is bounded,
and its restriction to $H$ admits the representation
\begin{equation*}
\left.\A\right|_H
=
\sum_{\beta=1}^{d}
I^{\otimes(\beta-1)}
\otimes
\gamma_\beta\partial_{z_\beta}^{\ell_\beta}
\otimes
I^{\otimes(d-\beta)}.
\end{equation*}
Proposition~\ref{prop:lc_scope_finite_order} applies with
$a_{\nu^{(\beta)}}=\gamma_\beta$ and
$\nu_m^{(\beta)}=\ell_\beta\delta_{m\beta}$,
yielding $R_\A=d=11$ and
\begin{equation*}
\Lambda_s=\{0,\ell_s\},
\qquad
\mathcal D_{s,\ell}=\partial_{z_s}^{\ell},
\qquad
b_{s,0}^{(\beta)}=1-\delta_{s\beta},
\qquad
b_{s,\ell_s}^{(\beta)}=\gamma_s\delta_{s\beta}.
\end{equation*}

For this problem, the boundary and initial conditions are imposed softly
through additional loss terms. The lateral faces $\Lat_s^\sigma$,
$s=1,\ldots,d_x$, $\sigma\in\{-,+\}$, with
$z_s^-=0$ and $z_s^+=1$, are discretized by the trace grids
\eqref{eq:boundary_face_grid}. The initial slice $\Gamma_0=(0,1)^{d_x}\times\{0\}$ is treated by the same fixed-coordinate construction, with the temporal
coordinate $z_d=t$ fixed at $t=0$:
\begin{align*}
\X_{\Gamma_0}
=
\X_1\times\cdots\times\X_{d_x}\times\{0\},
\qquad
N_0:=|\X_{\Gamma_0}|=n^{d-1}.
\end{align*}
Since $n_1=\cdots=n_d=n$, we also have
$N_{s,\sigma}=n^{d-1}$ for every lateral face. Setting
$u_0(x):=u^\star(x,0)$, the equal-weight collocation discretization
\eqref{eq:discrete_loss} and \eqref{eq:discrete_face_quadrature} gives
\begin{align*}
\widehat{\mathcal L}(\theta)
&=
\frac{1}{2N_\D}
\sum_{z\in\X_\D}
\bigl|
\partial_tu_\theta(z)-\kappa\Delta_xu_\theta(z)
\bigr|^2
\\
&\quad+
\frac{\lambda_{\mathrm b}}{2d_x}
\sum_{s=1}^{d_x}
\sum_{\sigma\in\{-,+\}}
\frac{1}{2N_{s,\sigma}}
\sum_{z\in\X_{\Lat_s^\sigma}}
\bigl|u_\theta(z)-u^\star(z)\bigr|^2
+
\frac{\lambda_0}{2N_0}
\sum_{z\in\X_{\Gamma_0}}
\bigl|u_\theta(z)-u_0(x)\bigr|^2,
\end{align*}
with loss weights $\lambda_{\mathrm b}=10$ and $\lambda_0=5$.



For the TT ansatz \eqref{eq:tt_ansatz}, set
$k_s=r_{s-1}r_s$ and $\Lambda_s=\{0,\ell_s\}$, with
$\mathcal D_{s,0}=I$. Since the boundary and initial condition residuals
use only the identity channel, no additional channels are required.
Following \eqref{eq:boundary_common_sensitivities}, we evaluate
\eqref{eq:numerical_sensitivity_matrices} on the augmented grids
$\widehat\X_s$ and define
\begin{equation}\label{eq:heat_channel_stacking}
\Psi_s^{\mathrm{ch},\D}
:=
\begin{bmatrix}
\Psi_{s,0}^{\mathrm{ch},\D}&
\Psi_{s,\ell_s}^{\mathrm{ch},\D}
\end{bmatrix},
\qquad
\Cmat_s^{\mathrm{ch}}
:=
\operatorname{col}(\mathsf H_{s,0},\mathsf H_{s,\ell_s}),
\qquad
N_s^{\mathrm{ch}}=2\widehat n_sk_s.
\end{equation}
With the matrix units $\mathsf M_{a_s}$ from
\eqref{eq:tt_matrix_units}, define
\begin{equation*}
\mathcal B_m^{\mathrm h}
:=
\begin{bmatrix}
G_m&\gamma_m\partial_{z_m}^{\ell_m}G_m\\
0&G_m
\end{bmatrix},
\qquad
\mathsf P_{s,0,a_s}^{\mathrm h}
:=
\begin{bmatrix}
\mathsf M_{a_s}&0\\
0&\mathsf M_{a_s}
\end{bmatrix},
\qquad
\mathsf P_{s,\ell_s,a_s}^{\mathrm h}
:=
\gamma_s
\begin{bmatrix}
0&\mathsf M_{a_s}\\
0&0
\end{bmatrix}.
\end{equation*}
Define $\mathcal Q_{<s}^{\mathrm h}\in\R^{1\times2r_{s-1}}$ and
$\mathcal Q_{>s}^{\mathrm h}\in\R^{2r_s\times1}$ by
\eqref{eq:poisson_partial_contractions}, replacing
$\mathcal B_m^\Delta$ with $\mathcal B_m^{\mathrm h}$. Then
\begin{equation*}
\Psi_{s,\ell}^{\mathrm{ch},\D}
\bigl(i,\overline{j_sa_s}\bigr)
=
\delta_{i_sj_s}\,
\mathcal Q_{<s}^{\mathrm h}(z_{<s,i_{<s}})
\mathsf P_{s,\ell,a_s}^{\mathrm h}
\mathcal Q_{>s}^{\mathrm h}(z_{>s,i_{>s}}),
\qquad
\ell\in\Lambda_s,
\end{equation*}
where $1\le j_s\le\widehat n_s$; the columns corresponding to appended
endpoint evaluations vanish for the interior residual. The proof is
identical to the TT argument in
Proposition~\ref{prop:poisson_channel_matrices}, with
$\gamma_m\partial_{z_m}^{\ell_m}$ in place of
$-\partial_{z_m}^2$. The lateral and initial-condition contributions are assembled by the
common-sensitivity construction of
Appendix~\ref{app:tt_boundary_contractions}; the initial slice corresponds
to the same construction with the temporal coordinate fixed at $t=0$.

It follows from \eqref{eq:boundary_common_sensitivities} that the compressed dimension for this problem is given by
\begin{equation*}
\Neff^{\mathrm{ch}}
=
2\sum_{s=1}^{d}\widehat n_s r_{s-1}r_s
=
2\sum_{s=1}^{d}n_s r_{s-1}r_s
+
2\sum_{s=1}^{d}(\widehat n_s-n_s)r_{s-1}r_s.
\end{equation*}
The boundary and initial condition terms increase the compressed
dimension only through the additional endpoint evaluations in
\eqref{eq:boundary_face_appendix}, rather than introducing separate
compressed spaces for the individual faces. We refer to
Appendix~\ref{app:tt_boundary_contractions} for further details regarding this issue.

We use a TT-PINN with $r_0=r_d=1$ and uniform interior rank $r=2$. Each TT
core is a fully connected $\tanh$ network
$1\to128\to128\to128\to128\to r_{m-1}r_m$. For $d=11$, the model has
$P=552{,}872$ parameters. We use
$n_1=\cdots=n_d=n=32$ interior collocation points per coordinate, drawn
independently from the uniform distribution and fixed throughout the run.
The concrete count in the effective
compressed dimension yields the number
\begin{align*}
\Neff^{\mathrm{ch}}
=
2\sum_{s=1}^{d}\widehat n_s r_{s-1}r_s
=
2\cdot34\cdot40
=
2720,
\end{align*}
where $\widehat n_s=n+2=34$. For the damping, we set $\mu=10^{-4}$. The ablation consists of $500$
Gauss--Newton iterations for each of $10$ random seeds. As a first-order
baseline, we compare against Adam, where the residual, loss, and gradient are
evaluated in the same compressed TT form, so that the two methods differ only
in how the update is computed. For Adam, we use the standard hyperparameters
$\beta_1^{\mathrm{Adam}}=0.9$, $\beta_2^{\mathrm{Adam}}=0.999$, and
$\epsilon^{\mathrm{Adam}}=10^{-8}$, together with the exponentially decaying
learning-rate schedule
\begin{align*}
\eta_k^{\mathrm{Adm}}
=
\eta_0^{\mathrm{Adm}}
\bigl(\eta_{\mathrm{end}}^{\mathrm{Adm}}/\eta_0^{\mathrm{Adm}}\bigr)^{k/K},
\qquad
\eta_0^{\mathrm{Adm}}=10^{-3},
\quad
\eta_{\mathrm{end}}^{\mathrm{Adm}}=10^{-5},
\end{align*}
over $K=50\,000$ iterations. The relative errors in the convergence curves
of Fig.~\ref{fig:ablation_heat_GN} are computed on midpoint factor grids with
$n_{\rm err}=256$ points per coordinate.

Both compressed Gauss--Newton and TT-Adam perform well with very competitive compute times due to TT compression, as shown in Fig.~\ref{fig:ablation_heat_GN}. However, the compressed Gauss--Newton reaches a median relative error of about $1.5\cdot10^{-4}$ after $500$ iterations, compared with $7\cdot10^{-4}$ for Adam after $50,000$ iterations. It is also substantially faster, requiring about $16$ seconds versus $73$ seconds for Adam. The error curves also differ in shape: the Gauss--Newton error levels off after roughly $200$ iterations, while the Adam error is still decreasing slowly at the end of the budget.

\begin{figure}[H]
\centering
\includegraphics[width=1\linewidth]{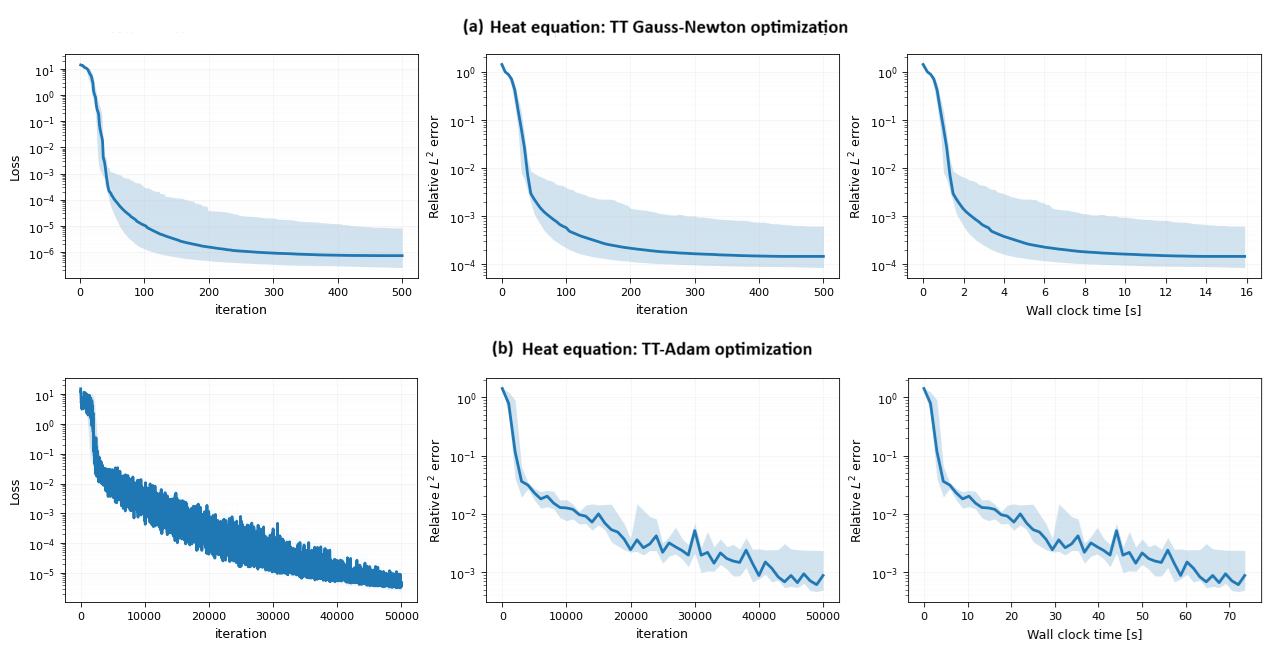}
\caption{Heat equation: (a) compressed
Gauss--Newton ($500$ iterations) and (b) TT-Adam ($50\,000$ iterations). Solid
curves show the median over ten seeds, shaded regions the interquartile range
between the 25th and 75th percentiles. From left to right: loss, relative $L^2$
error, and relative $L^2$ error versus wall-clock time; the first two are
plotted against the iteration count.}
\label{fig:ablation_heat_GN}
\end{figure}

\subsection{Parametric Darcy problem}

We consider the parametric Darcy problem
\begin{equation}\label{eq:parametric_diffusion}
-\nabla_x\cdot\bigl(\kappa(x,\xi)\nabla_x u(x,\xi)\bigr)
=f(x,\xi)
\quad\text{in }\Omega\times\Xi,
\qquad
u=0
\quad\text{on }\partial\Omega\times\Xi,
\end{equation}
where $\Omega=(0,1)^{d_x}$ and $\Xi=(-1,1)^{d_\xi}$ with $d_x = 3$ and $d_\xi=25$, respectively. We study two permeability coefficients of increasing parametric complexity. Before conducting the numerical ablation, we first analyze the additional
problem structure that can be exploited to reduce the computational costs
associated with both the assembly and the compressed Gauss--Newton solve.

\subsubsection{Affine permeability coefficient: analysis}

The first permeability benchmark is the affine coefficient; cf.~\cite{eigel2017adaptive}:
\begin{equation}\label{eq:affine_coefficient_general}
\kappa_{\rm aff}(x,\xi)
=
\bar\kappa
+
\varepsilon\sum_{j=1}^{\dpa}
\xi_j\sqrt{\lambda_j}
\prod_{m=1}^{\dsp}\cos(\pi jx_m),
\qquad
\lambda_j=j^{-2},
\end{equation}
with $\bar\kappa=1$, amplitude $\varepsilon=0.1$, and independent parameters
$\xi_j\sim\operatorname{Unif}(-1,1)$ for $j=1,\ldots,\dpa$. Writing
\begin{equation}\label{eq:affine_modes}
\kappa_{\rm aff}(x,\xi)
=
\sum_{j=0}^{\dpa}\kappa^{(j)}(x,\xi),
\end{equation}
where $\kappa^{(0)}:=\bar\kappa$ and
\[
\kappa^{(j)}(x,\xi)
=
\varepsilon\sqrt{\lambda_j}\,\xi_j
\prod_{m=1}^{\dsp}\cos(\pi jx_m),
\qquad j=1,\ldots,\dpa,
\]
each mode is rank one over $z=(x,\xi)\in\R^d$, $d=\dsp+\dpa$:
\begin{equation}\label{eq:affine_mode_product}
\kappa^{(j)}(z)
=
\prod_{m=1}^{d}\kappa_{j,m}(z_m),
\qquad j=0,\ldots,\dpa,
\end{equation}
with, for $j\ge1$,
\begin{equation}\label{affine coefficient for TT}
\kappa_{j,m}(z_m)
=
\begin{cases}
\varepsilon\sqrt{\lambda_j}\cos(\pi jz_m),
& m=1,\\[1mm]
\cos(\pi jz_m),
& 2\le m\le\dsp,\\[1mm]
z_m,
& m=\dsp+j,\\[1mm]
1,
& \dsp<m\le d,\quad m\neq\dsp+j,
\end{cases}
\end{equation}
and $\kappa_{0,1}:=\bar\kappa$, $\kappa_{0,m}:=1$ for $m=2,\ldots,d$.
The following lemma verifies Assumption~\ref{ass:operator}.
\begin{lemma}\label{lem:affine_operator_separability}
Let $\A v:=-\nabla_x\cdot(\kappa_{\rm aff}\nabla_xv)$ and let the
univariate spaces satisfy the derivative assumptions of
Proposition~\ref{prop:lc_scope_finite_order}, with $q_s=2$ for
$s\le\dsp$ and $q_s=0$ otherwise. Then $\A$ has the separated representation
\begin{equation}\label{eq:affine_operator_decomposition}
\A=\sum_{j=0}^{\dpa}\sum_{q=1}^{\dsp}
\A_1^{(j,q)}\otimes\cdots\otimes\A_d^{(j,q)},
\qquad R_\A=\dsp(\dpa+1),
\end{equation}
where the operator index $\beta=1,\ldots,R_\A$ is identified
with the pair $(j,q)\in\{0,\ldots,\dpa\}\times\{1,\ldots,\dsp\}$.
The corresponding bounded univariate factors are
\begin{equation}\label{eq:affine_univariate_operators}
\A_s^{(j,q)}h:=
\begin{cases}
-\kappa_{j,s}'\,h'-\kappa_{j,s}\,h'',&s=q,\\[1mm]
\kappa_{j,s}\,h,&s\ne q.
\end{cases}
\end{equation}
The local differential-channel representation is given by
\begin{equation}\label{eq:darcy_local_channels}
\mathcal D_{s,\ell}=\partial_{z_s}^{\ell},
\qquad
\Lambda_s=
\begin{cases}
\{0,1,2\},&s\le\dsp,\\
\{0\},&s>\dsp,
\end{cases}
\end{equation}
with channel coefficients, for $s\le\dsp$,
\begin{equation}\label{eq:affine_local_channel_coefficients}
b_{s,0}^{(j,q)}=(1-\delta_{sq})\kappa_{j,s},\qquad
b_{s,1}^{(j,q)}=-\delta_{sq}\kappa_{j,s}',\qquad
b_{s,2}^{(j,q)}=-\delta_{sq}\kappa_{j,s}.
\end{equation}
Here $b_{s,\ell}^{(j,q)}$ denotes $b_{s,\ell}^{(\beta)}$ under the
identification $\beta\equiv(j,q)$. For $s>\dsp$, only
$b_{s,0}^{(j,q)}=\kappa_{j,s}$ is needed.
\end{lemma}

\begin{proof}
Expanding the divergence gives
\begin{equation}\label{eq:darcy_expanded_operator}
\A v
=
-\sum_{q=1}^{\dsp}
\left[
(\partial_{x_q}\kappa_{\rm aff})\partial_{x_q}v
+
\kappa_{\rm aff}\partial_{x_q}^2v
\right],
\end{equation}
where
\begin{equation*}
\kappa_{\rm aff}
=
\sum_{j=0}^{\dpa}\prod_{m=1}^{d}\kappa_{j,m},
\qquad
\partial_{x_q}\kappa_{\rm aff}
=
\sum_{j=0}^{\dpa}
\kappa_{j,q}'\prod_{m\ne q}\kappa_{j,m}.
\end{equation*}
Thus the coefficients of both $\partial_{x_q}v$ and
$\partial_{x_q}^2v$ admit finite separable representations with smooth
univariate factors, and Proposition~\ref{prop:lc_scope_finite_order}
applies.

For an elementary tensor $v=v_1\otimes\cdots\otimes v_d$, each pair
$(j,q)$ gives
\begin{equation}\label{eq:affine_separable_direction}
-\partial_{x_q}\bigl(\kappa^{(j)}\partial_{x_q}v\bigr)
=
\bigl[-\kappa_{j,q}'v_q'-\kappa_{j,q}v_q''\bigr]
\prod_{m\ne q}\kappa_{j,m}v_m
=
\bigl(\A_1^{(j,q)}\otimes\cdots\otimes\A_d^{(j,q)}\bigr)v.
\end{equation}
Summation over $(j,q)$ yields
\eqref{eq:affine_operator_decomposition}, with the pair $(j,q)$ playing
the role of the abstract index $\beta$ and
$R_\A=\dsp(\dpa+1)$.

For $s\le\dsp$, substituting
\eqref{eq:affine_local_channel_coefficients} gives
\begin{align*}
\sum_{\ell\in\Lambda_s}
b_{s,\ell}^{(j,q)}\mathcal D_{s,\ell}h
&=
(1-\delta_{sq})\kappa_{j,s}h
-\delta_{sq}\bigl(\kappa_{j,s}'h'+\kappa_{j,s}h''\bigr)=
\A_s^{(j,q)}h.
\end{align*}
For $s>\dsp$, we have $s\ne q$, and hence
\begin{equation*}
b_{s,0}^{(j,q)}\mathcal D_{s,0}h
=
\kappa_{j,s}h
=
\A_s^{(j,q)}h.
\end{equation*}
This proves the stated local differential-channel representation.
\end{proof}

For each coefficient mode, we collect the sum over spatial derivatives
into one block TT.
\begin{lemma}
\label{lem:affine_block_core}
Let $u_\theta$ be given by \eqref{eq:tt_ansatz}, and assume that
$G_m\in C^2(\D_m;\R^{r_{m-1}\times r_m})$ for $m\le\dsp$. For
$j=0,\ldots,\dpa$ and $m=1,\ldots,d$, define
\begin{equation}\label{eq:affine_operated_cores}
\begin{aligned}
A_m^{(j)}(z_m)
&:=\kappa_{j,m}(z_m)G_m(z_m), \qquad
B_m^{(j)}(z_m)
:=
\begin{cases}
-\partial_{z_m}\bigl(\kappa_{j,m}(z_m)G_m'(z_m)\bigr),
& m\le\dsp,\\[1mm]
0,
& m>\dsp.
\end{cases}
\end{aligned}
\end{equation}
where differentiation is understood entrywise.
Both $A_m^{(j)}$ and $B_m^{(j)}$ take values in
$\R^{r_{m-1}\times r_m}$. Introduce the block cores
\begin{equation}\label{eq:affine_block_cores}
\mathcal B_m^{(j)}(z_m)
:=
\begin{bmatrix}
A_m^{(j)}(z_m)&B_m^{(j)}(z_m)\\
0&A_m^{(j)}(z_m)
\end{bmatrix}
\in\R^{2r_{m-1}\times2r_m}.
\end{equation}
Then
\begin{equation}\label{eq:affine_block_chain}
-\nabla_x\cdot
\bigl(\kappa^{(j)}\nabla_xu_\theta\bigr)(z)
=
\begin{bmatrix}1&0\end{bmatrix}
\mathcal B_1^{(j)}(z_1)\cdots\mathcal B_d^{(j)}(z_d)
\begin{bmatrix}0\\1\end{bmatrix}.
\end{equation}
The internal TT ranks of this representation are bounded by $2r_m$,
$m=1,\ldots,d-1$. In particular, if $u_\theta$ has maximal TT-rank $r$,
then \eqref{eq:affine_block_chain} has maximal TT-rank at most $2r$.
\end{lemma}

\begin{proof}
Fix $j$ and $q\le\dsp$. Differentiating \eqref{eq:tt_ansatz} and
distributing the factors in \eqref{eq:affine_mode_product}
gives
\begin{equation}\label{eq:affine_tt_flux_factorization}
\kappa^{(j)}\partial_{x_q}u_\theta
=
A_1^{(j)}\cdots A_{q-1}^{(j)}
\bigl(\kappa_{j,q}G_q'\bigr)
A_{q+1}^{(j)}\cdots A_d^{(j)}.
\end{equation}
Since all factors outside the $q$-th position are independent of $x_q$,
\begin{equation*}
\begin{aligned}
-\partial_{x_q}\bigl(\kappa^{(j)}\partial_{x_q}u_\theta\bigr)
&=
A_1^{(j)}\cdots A_{q-1}^{(j)}
\left[-\partial_{x_q}\bigl(\kappa_{j,q}G_q'\bigr)\right]
A_{q+1}^{(j)}\cdots A_d^{(j)}
\\
&=
A_1^{(j)}\cdots A_{q-1}^{(j)}
B_q^{(j)}
A_{q+1}^{(j)}\cdots A_d^{(j)}.
\end{aligned}
\end{equation*}
Summing over $q$ yields
\begin{equation}\label{eq:affine_tt_insertion_sum}
-\nabla_x\cdot\bigl(\kappa^{(j)}\nabla_xu_\theta\bigr)
=
\sum_{q=1}^{\dsp}
A_1^{(j)}\cdots A_{q-1}^{(j)}
B_q^{(j)}
A_{q+1}^{(j)}\cdots A_d^{(j)}.
\end{equation}
Ordinary block multiplication gives
\begin{equation}\label{eq:affine_partial_block_product}
\mathcal B_1^{(j)}\cdots\mathcal B_d^{(j)}
=
\begin{bmatrix}
A_1^{(j)}\cdots A_d^{(j)}
&
\displaystyle
\sum_{q=1}^{d}
A_1^{(j)}\cdots A_{q-1}^{(j)}
B_q^{(j)}
A_{q+1}^{(j)}\cdots A_d^{(j)}
\\[2mm]
0&
A_1^{(j)}\cdots A_d^{(j)}
\end{bmatrix}.
\end{equation}
Since $B_q^{(j)}=0$ for $q>\dsp$, the boundary unit vectors extract
exactly the sum in \eqref{eq:affine_tt_insertion_sum}, proving
\eqref{eq:affine_block_chain}. Each block core has size
$2r_{m-1}\times2r_m$. Absorbing the boundary vectors into the first and
last cores gives boundary ranks one and internal TT ranks at most
$2r_m$, hence maximal TT rank at most $2r$.
\end{proof}
\noindent
Combining \eqref{eq:affine_modes} and Lemma~\ref{lem:affine_block_core},
we obtain the representation
\begin{equation}\label{eq:affine_full_block_representation}
\A u_\theta(z)
=
\sum_{j=0}^{\dpa}
\begin{bmatrix}1&0\end{bmatrix}
\mathcal B_1^{(j)}(z_1)\cdots
\mathcal B_d^{(j)}(z_d)
\begin{bmatrix}0\\1\end{bmatrix}.
\end{equation}
Hence, $\A u_\theta$ is represented by
$R_\kappa=\dpa+1$ TT trains, each with maximal TT-rank at most $2r$.
We emphasize that $R_\kappa$ differs from the operator-term count
$R_{\A}=\dsp(\dpa+1)$ in \eqref{eq:affine_operator_decomposition}.

We proceed with the manufactured solution
\begin{equation}\label{eq:affine_exact_solution}
u_{\rm aff}^\star(x,\xi)
=
\prod_{m=1}^{\dsp}\sin(\pi x_m)
\big(1+\sum_{j=1}^{\dpa}\xi_j\big),
\end{equation}
which has maximal TT-rank two, and define
\[
f
=
-\nabla_x\cdot
\left(
\kappa_{\rm aff}\nabla_x u_{\rm aff}^\star
\right).
\]
The boundary conditions are imposed via \eqref{eq:SDF num}. The
interior residual is given by
\begin{equation*}
\Res_{\rm aff}(\theta)
=
-\nabla_x\cdot
\left(
\kappa_{\rm aff}
\nabla_x(u_\theta-u_{\rm aff}^\star)
\right)
=
\sum_{j=0}^{\dpa}\Res^{(j)}(\theta), \quad \Res^{(j)}(\theta)
:=
-\nabla_x\cdot
(
\kappa^{(j)}
\nabla_x(u_\theta-u_{\rm aff}^\star))
\end{equation*}
Define the partial contractions
\begin{equation}\label{eq:affine_partial_contractions}
\mathcal Q_{<s}^{(j)}:=\begin{bmatrix}1&0\end{bmatrix}
\prod_{m<s}\mathcal B_m^{(j)}\in\R^{1\times2r_{s-1}},
\qquad
\mathcal Q_{>s}^{(j)}:=\left(\prod_{m>s}\mathcal B_m^{(j)}\right)
\begin{bmatrix}0\\1\end{bmatrix}\in\R^{2r_s\times1}.
\end{equation}
We have the following algebraic characterization of the factors in the Jacobian. 
\begin{proposition}[Local-channel matrices for affine Darcy]
\label{prop:affine_darcy_compression}
Let the setting of Lemmas~\ref{lem:affine_operator_separability}
and~\ref{lem:affine_block_core} hold, and use the factor grids
\eqref{eq:factor_grids}. Assume that the TT-core parametrizations are
differentiable with values in $C^2(\overline{\D_s})$ entrywise for
$s\le\dsp$ and in $C(\overline{\D_s})$ otherwise.
For $k_s=r_{s-1}r_s$, $a_s=\overline{\alpha_{s-1}\alpha_s}$, and the
channels \eqref{eq:darcy_local_channels}, the sensitivity factor in
Theorem~\ref{prop:local_channel_compression} is
\begin{equation}\label{eq: active jacobian channels affine}
\Cmat_s^{\mathrm{ch}}:=
\begin{cases}
\operatorname{col}(\mathsf H_{s,0},\mathsf H_{s,1},\mathsf H_{s,2}),&s\le\dsp,\\[1mm]
\mathsf H_{s,0},&s>\dsp,
\end{cases}
\qquad
N_s^{\mathrm{ch}}=\begin{cases}3n_sk_s,&s\le\dsp,\\n_sk_s,&s>\dsp,\end{cases}
\end{equation}
where $\mathsf H_{s,\ell}$ is given by
\eqref{eq:numerical_sensitivity_matrices} and
$\Cmat_s^{\mathrm{ch}}\in\R^{N_s^{\mathrm{ch}}\times P_s}$.
Let $\mathsf M_{a_s}$ be the TT matrix unit as in \eqref{eq:tt_matrix_units}. For
$j=0,\ldots,\dpa$ and $\ell\in\Lambda_s$, define
\begin{equation}\label{eq:affine_derivative_insertions}
\mathsf P_{s,\ell,a_s}^{(j)}(z_s):=
\begin{cases}
\displaystyle\kappa_{j,s}(z_s)
\begin{bmatrix}\mathsf M_{a_s}&0\\0&\mathsf M_{a_s}\end{bmatrix},&\ell=0,\\[3mm]
\displaystyle-\kappa_{j,s}'(z_s)
\begin{bmatrix}0&\mathsf M_{a_s}\\0&0\end{bmatrix},&\ell=1,\\[3mm]
\displaystyle-\kappa_{j,s}(z_s)
\begin{bmatrix}0&\mathsf M_{a_s}\\0&0\end{bmatrix},&\ell=2,
\end{cases}
\end{equation}
with size $2r_{s-1}\times2r_s$. The direction matrices are
\begin{equation}\label{eq:affine_reduced_direction_columns}
\Psi_{s,\ell}^{\mathrm{ch}}\bigl(i,\overline{j_sa_s}\bigr)
=\delta_{i_sj_s}\sum_{j=0}^{\dpa}
\mathcal Q_{<s}^{(j)}(z_{<s,i_{<s}})
\mathsf P_{s,\ell,a_s}^{(j)}(z_{s,j_s})
\mathcal Q_{>s}^{(j)}(z_{>s,i_{>s}}),
\end{equation}
where $\Psi_{s,\ell}^{\mathrm{ch}}\in\R^{N_\D\times n_sk_s}$, and
\begin{equation}\label{eq:darcy_reduced_channel_stacking}
\Psi_s^{\mathrm{ch}}:=
\begin{cases}
\begin{bmatrix}\Psi_{s,0}^{\mathrm{ch}}&\Psi_{s,1}^{\mathrm{ch}}&
\Psi_{s,2}^{\mathrm{ch}}\end{bmatrix},&s\le\dsp,\\[1mm]
\Psi_{s,0}^{\mathrm{ch}},&s>\dsp,
\end{cases}
\qquad \Psi_s^{\mathrm{ch}}\in\R^{N_\D\times N_s^{\mathrm{ch}}}.
\end{equation}
The resulting compressed dimension is given by
\begin{equation}\label{eq:darcy_reduced_dimension}
\Neff^{\mathrm{ch}}=3\sum_{s=1}^{\dsp}n_sr_{s-1}r_s
+\sum_{s=\dsp+1}^{d}n_sr_{s-1}r_s.
\end{equation}
\end{proposition}
\begin{proof}
For $s\le\dsp$, differentiating the $s$-th block core
\eqref{eq:affine_block_cores}
with respect to $\theta_p^s$ gives
\begin{equation*}
\partial_{\theta_p^s}\mathcal B_s^{(j)}
=
\begin{bmatrix}
\kappa_{j,s}\partial_{\theta_p^s}G_s&
-\kappa_{j,s}'\partial_{z_s}\partial_{\theta_p^s}G_s
-\kappa_{j,s}\partial_{z_s}^2\partial_{\theta_p^s}G_s\\
0&\kappa_{j,s}\partial_{\theta_p^s}G_s
\end{bmatrix}.
\end{equation*}
Expanding the matrix-valued derivatives in the TT matrix units gives
\begin{equation*}
\partial_{\theta_p^s}\mathcal B_s^{(j)}(z_s)
=
\sum_{\ell=0}^{2}\sum_{a_s=1}^{k_s}
\mathsf P_{s,\ell,a_s}^{(j)}(z_s)
\bigl(\partial_{z_s}^{\ell}\partial_{\theta_p^s}G_s\bigr)(z_s)
(\alpha_{s-1},\alpha_s),
\qquad s\le\dsp.
\end{equation*}
For $s>\dsp$, differentiating the same block core gives
\begin{equation*}
\partial_{\theta_p^s}\mathcal B_s^{(j)}(z_s)
=
\sum_{a_s=1}^{k_s}
\mathsf P_{s,0,a_s}^{(j)}(z_s)
\bigl(\partial_{\theta_p^s}G_s\bigr)(z_s)
(\alpha_{s-1},\alpha_s).
\end{equation*}
Next, differentiating the block-TT representation
\eqref{eq:affine_block_chain} with respect to $\theta_p^s$ gives
\begin{equation*}
\partial_{\theta_p^s}
\left[
-\nabla_x\cdot
\bigl(\kappa^{(j)}\nabla_xu_\theta\bigr)
\right](z)
=
\mathcal Q_{<s}^{(j)}(z_{<s})
\partial_{\theta_p^s}\mathcal B_s^{(j)}(z_s)
\mathcal Q_{>s}^{(j)}(z_{>s}).
\end{equation*}
Substituting the expansions above, summing over $j$, and collocating on
the factor grid yields
\begin{equation*}
\Psi_{s,\ell}^{\mathrm{ch}}
\bigl(i,\overline{j_sa_s}\bigr)
=
\delta_{i_sj_s}
\sum_{j=0}^{\dpa}
\mathcal Q_{<s}^{(j)}(z_{<s,i_{<s}})
\mathsf P_{s,\ell,a_s}^{(j)}(z_{s,j_s})
\mathcal Q_{>s}^{(j)}(z_{>s,i_{>s}}),
\end{equation*}
for $\ell=0,1,2$ if $s\le\dsp$ and $\ell=0$ otherwise. Hence
\eqref{eq:affine_reduced_direction_columns} realizes
\eqref{eq:local_channel_direction_columns}.

Finally, $L_s=3$ for $s\le\dsp$ and $L_s=1$ for $s>\dsp$, so the  dimension formulas follow from
Theorem~\ref{prop:local_channel_compression} trivially.
\end{proof}

For convenience, we adapt Algorithm~\ref{alg:tt_assembly} to the above instance of the affine Darcy problem. This yields Algorithm~\ref{alg:affine_darcy_assembly}, which summarizes the above analysis and provides the corresponding TT assembly.
\begin{algorithm}[H]
\caption{Channel-compressed TT assembly for affine parametric Darcy}
\label{alg:affine_darcy_assembly}
\small
\begin{algorithmic}[1]
\setlength{\itemsep}{0.25em}
\REQUIRE Current parameters $\theta$; factor grids $\X_s$; TT ranks
$(r_0,\ldots,r_d)$; coefficient factors $\kappa_{j,m}$ from
\eqref{affine coefficient for TT}.

\STATE Set $\Lambda_s$ by \eqref{eq:darcy_local_channels}; evaluate
$\mathcal D_{s,\ell}G_s$ for $\ell\in\Lambda_s$, compute
$\mathsf H_{s,\ell}$ from \eqref{eq:numerical_sensitivity_matrices} and
assemble $\Cmat_s^{\mathrm{ch}}$ according to
\eqref{eq: active jacobian channels affine}.

\STATE Form 
$\mathcal B_m^{(j)}$ from \eqref{eq:affine_block_cores} and $\mathsf P_{s,\ell,a_s}^{(j)}$ from
\eqref{eq:affine_derivative_insertions} for $j=0,\ldots,\dpa$.

\STATE Form the residual
$\Res_{\rm aff}=\sum_{j=0}^{\dpa}\Res^{(j)}$
using the block-TT representation
\eqref{eq:affine_full_block_representation}.

\STATE Compute the model/model interfaces for all mode pairs $(j,j')$ by
\eqref{eq:TT_recursions_interfaces} and the corresponding
model/residual interfaces by \eqref{eq:tt_residual_interfaces}.

\FOR{$s=1,\ldots,d$}
  \STATE Assemble $\Upsilon_{ss}^{\mathrm{ch}}$ by the diagonal
  contraction \eqref{eq:tt_gramian_diagonal_final}, replacing the active
  matrix units by $\mathsf P_{s,\ell,a_s}^{(j)}$ and summing over
  $j,j'=0,\ldots,\dpa$.

  \STATE Assemble $\rho_{\mathrm c,s}^{\mathrm{ch}}$ by
  \eqref{eq:tt_gradient_local_matrix}--\eqref{eq:tt_compressed_rhs}
  using $\mathsf P_{s,\ell,a_s}^{(j)}$ at the active coordinate and the
  finite-sum rule of Remark~\ref{rem:tt_residual_sum}.

  \STATE For every $(\ell,i_s,a_s)$ initialize the bridge family
  $\mathsf Z^{(j,j')}$ as in Algorithm~\ref{alg:tt_assembly} with
  $\mathsf M_{a_s}$ replaced by
  $\mathsf P_{s,\ell,a_s}^{(j)}$.

  \FOR{$t=s+1,\ldots,d$}
    \STATE Assemble $\Upsilon_{st}^{\mathrm{ch}}$ by
    \eqref{eq:tt_gramian_offdiagonal_final} using
    $\mathsf P_{t,\ell',a_t'}^{(j')}$ at the second active coordinate and
    summing over $(j,j')$; set
    $\Upsilon_{ts}^{\mathrm{ch}}
    \leftarrow(\Upsilon_{st}^{\mathrm{ch}})^\top$.

    \IF{$t<d$}
      \STATE Advance $\mathsf Z^{(j,j')}$ through
      $(\mathcal B_t^{(j)},\mathcal B_t^{(j')})$ by the forward transfer.
    \ENDIF
  \ENDFOR
\ENDFOR

\STATE Stack the channel blocks in the order $(s,\ell,i_s,a_s)$,
compute
$\nabla\widehat{\mathcal L}
=(\Cmat^{\mathrm{ch}})^\top\rho_{\mathrm c}^{\mathrm{ch}}$
as in \eqref{eq:compressed_gradient}, and evaluate the loss by applying
\eqref{eq:tt_loss_recursion}--\eqref{eq:tt_loss_contraction}
to all pairs of residual trains.

\ENSURE
$\{\Cmat_s^{\mathrm{ch}}\}_{s=1}^d$,
$\Upsilon^{\mathrm{ch}}$,
$\rho_{\mathrm c}^{\mathrm{ch}}$,
$\nabla\widehat{\mathcal L}(\theta)$, and
$\widehat{\mathcal L}(\theta)$ for Algorithm~\ref{alg:woodbury}, with
$\Neff^{\mathrm{ch}}$ given by \eqref{eq:darcy_reduced_dimension}.
\end{algorithmic}
\end{algorithm}

\subsubsection{Quadratic permeability coefficient: analysis}

The second coefficient introduces quadratic parameter interactions:
\begin{equation}\label{eq:quadratic_coefficient}
\kappa_{\rm quad}(x,\xi)
=\big(1+\varepsilon\,\psi(x)\sum_{j=1}^{\dpa}w_j\xi_j\big)^2,\qquad
\psi(x)=\prod_{m=1}^{\dsp}\cos(\pi x_m),
\end{equation}
with $w_j:=j^{-2}/\sum_{k=1}^{\dpa}k^{-2}$, amplitude $\varepsilon=0.5$, and
parameter vectors $\xi\sim\operatorname{Unif}(-1,1)^{\dpa}$. Its explicit
expansion contains $1+2\dpa+\binom{\dpa}{2}$ separable coefficient terms, whose
factors and first spatial derivatives are bounded.
For this problem, Assumption~\ref{ass:operator} can be established along the lines of
Lemma~\ref{lem:affine_operator_separability} under the same regularity
requirements and yields the tensorized factorization of $\mathcal{A}$ with
$R_{\A}=\dsp\bigl(1+2\dpa+\binom{\dpa}{2}\bigr)$ terms.

The explicit expansion of \eqref{eq:quadratic_coefficient} contains $351$ terms for $\dpa=25$, making a direct treatment computationally burdensome. Instead, we exploit a compact TT factorization. 
\begin{lemma}
\label{lem:quadratic_coefficient_tt}
Let $\chi_m(x_m):=\cos(\pi x_m)$ and $\sigma_j(\xi_j):=w_j\xi_j$.
For $m=1,\ldots,\dsp$ and $j=1,\ldots,\dpa$, define
\begin{equation}\label{eq:quadratic_state_matrices}
K_m(x_m)
:=
\operatorname{diag}\bigl(1,\chi_m(x_m),\chi_m(x_m)^2\bigr),
\qquad
M_j(\xi_j)
:=
\begin{bmatrix}
1&0&0\\
\sigma_j(\xi_j)&1&0\\
\sigma_j(\xi_j)^2&2\sigma_j(\xi_j)&1
\end{bmatrix},
\end{equation}
together with
$\ell_\kappa:=\begin{bmatrix}1&2\varepsilon&\varepsilon^2\end{bmatrix}$
and $r_\kappa:=\begin{bmatrix}1&0&0\end{bmatrix}^{\top}$.
Set
\begin{equation*}
\varrho_0=\varrho_d=1,
\qquad
\varrho_m=3,
\qquad
1\le m<d.
\end{equation*}
Define the coefficient TT cores
$C_m:\D_m\to\R^{\varrho_{m-1}\times\varrho_m}$ by
\begin{equation}\label{eq:quadratic_coefficient_cores}
C_m(z_m)
:=
\begin{cases}
\ell_\kappa K_1(x_1),
& m=1,\\[1mm]
K_m(x_m),
& 2\le m\le\dsp,\\[1mm]
M_j(\xi_j),
& m=\dsp+j,\quad 1\le j<\dpa,\\[1mm]
M_{\dpa}(\xi_{\dpa})r_\kappa,
& m=d.
\end{cases}
\end{equation}
Then
\begin{equation}\label{eq:quadratic_coefficient_tt}
\kappa_{\rm quad}(x,\xi)
=
C_1(z_1)\cdots C_d(z_d)
\end{equation}
is an exact TT representation of \eqref{eq:quadratic_coefficient} with maximal TT-rank at most three.
\end{lemma}

\begin{proof}
Expanding the brackets in \eqref{eq:quadratic_coefficient} gives
\begin{equation*}
\kappa_{\rm quad}(x,\xi)
=
1+2\varepsilon\psi(x)S(\xi)
+\varepsilon^2\psi(x)^2S(\xi)^2,
\qquad
S(\xi):=\sum_{j=1}^{\dpa}\sigma_j(\xi_j).
\end{equation*}
Fix $\xi$ and define the partial sums
\[
t_j:=\sum_{k=j}^{\dpa}\sigma_k(\xi_k),
\quad j=1,\ldots,\dpa,
\]
with $t_{\dpa+1}:=0$. Since $t_j=t_{j+1}+\sigma_j(\xi_j)$, observe that expanding the square yields
\begin{equation}\label{eq:quadratic_state_update}
\begin{bmatrix}
1\\t_j\\t_j^2
\end{bmatrix}
=
\begin{bmatrix}
1\\
t_{j+1}+\sigma_j(\xi_j)\\
t_{j+1}^2+2\sigma_j(\xi_j)t_{j+1}+\sigma_j(\xi_j)^2
\end{bmatrix}
=
M_j(\xi_j)
\begin{bmatrix}
1\\t_{j+1}\\t_{j+1}^2
\end{bmatrix}.
\end{equation}
Starting from $r_\kappa=(1,0,0)^\top$ and using $t_1=S(\xi)$, we obtain
\[
M_1(\xi_1)\cdots M_{\dpa}(\xi_{\dpa})r_\kappa
=
\begin{bmatrix}
1\\S(\xi)\\S(\xi)^2
\end{bmatrix}.
\]
The spatial matrices satisfy
$K_1(x_1)\cdots K_{\dsp}(x_{\dsp})
=\operatorname{diag}(1,\psi(x),\psi(x)^2)$.
Consequently,
\begin{align*}
C_1(z_1)\cdots C_d(z_d)
&=
\ell_\kappa
K_1(x_1)\cdots K_{\dsp}(x_{\dsp})
M_1(\xi_1)\cdots M_{\dpa}(\xi_{\dpa})r_\kappa\\
&=
\begin{bmatrix}
1&2\varepsilon\psi(x)&\varepsilon^2\psi(x)^2
\end{bmatrix}
\begin{bmatrix}
1\\S(\xi)\\S(\xi)^2
\end{bmatrix}
=
\kappa_{\rm quad}(x,\xi).
\end{align*}
The first and last cores have sizes $1\times3$ and $3\times1$,
respectively, while all interior cores have size $3\times3$.
Hence the maximal TT-rank is at most three.
\end{proof}

We use the manufactured solution
\begin{equation*}
\begin{aligned}
u_{\rm quad}^\star(x,\xi)
=
\prod_{m=1}^{\dsp}\sin(\pi x_m)
\prod_{j=1}^{\dpa}\bigl(1+b\,w_j\xi_j\bigr)+
\zeta\prod_{m=1}^{\dsp}\sin(2\pi x_m)
\prod_{j=1}^{\dpa}\bigl(1-b\,w_j\xi_j\bigr),
\end{aligned}
\end{equation*}
with $b=0.4$ and $\zeta=0.25$. It vanishes on $\partial\Omega\times\Xi$ and is a sum of two elementary tensors, with maximal TT rank at most two. 
We set $f:=-\nabla_x\cdot(\kappa_{\rm quad}\nabla_xu_{\rm quad}^\star)$
and impose the boundary conditions via \eqref{eq:SDF num}.
The interior residual is then given by
\begin{equation}\label{eq:quadratic_darcy_residual}
\Res_{\rm quad}(\theta)
=
-\nabla_x\cdot\left(
\kappa_{\rm quad}\nabla_x(u_\theta-u_{\rm quad}^\star)
\right).
\end{equation}
To organize the operator contractions, define
\begin{equation}\label{eq:quadratic_operated_blocks}
A_m:=C_m\otimes G_m,\quad
B_m:=\begin{cases}
-C_m'\otimes G_m'-C_m\otimes G_m'',&m\le\dsp,\\[1mm]
0,&m>\dsp,
\end{cases}
\quad
\mathcal B_m:=\begin{bmatrix}A_m&B_m\\0&A_m\end{bmatrix}.
\end{equation}
Here $\mathcal B_m$ has size
$2\varrho_{m-1}r_{m-1}\times2\varrho_mr_m$.
The mixed-product identity for Kronecker products gives
\begin{equation*}
\kappa_{\rm quad}\partial_{x_q}u_\theta
=A_1\cdots A_{q-1}(C_q\otimes G_q')A_{q+1}\cdots A_d.
\end{equation*}
Differentiating at $q$ and summing the resulting insertions as in
Lemma~\ref{lem:affine_block_core} yields
\begin{equation}\label{eq:quadratic_block_chain}
-\nabla_x\cdot(\kappa_{\rm quad}\nabla_xu_\theta)
=\begin{bmatrix}1&0\end{bmatrix}
\mathcal B_1\cdots\mathcal B_d\begin{bmatrix}0\\1\end{bmatrix}.
\end{equation}
The internal ranks are at most $2\varrho_mr_m=6r_m$.
Define the partial contractions
\begin{equation}\label{eq:quadratic_partial_contractions}
\mathcal Q_{<s}:=\begin{bmatrix}1&0\end{bmatrix}\prod_{m<s}\mathcal B_m,
\qquad
\mathcal Q_{>s}:=\left(\prod_{m>s}\mathcal B_m\right)
\begin{bmatrix}0\\1\end{bmatrix},
\end{equation}
with sizes $1\times2\varrho_{s-1}r_{s-1}$ and
$2\varrho_sr_s\times1$, respectively.

In the following proposition, we keep the univariate spaces, TT-core regularity, factor grids, and
notation of Proposition~\ref{prop:affine_darcy_compression}.
By Proposition~\ref{prop:lc_scope_finite_order}, the Darcy operator with the quadratic coefficient
has the same local differential channels \eqref{eq:darcy_local_channels} as the Darcy operator with the affine one. Hence, the matrices
$\mathsf H_{s,\ell}$ and $\Cmat_s^{\mathrm{ch}}$ remain those in
\eqref{eq:numerical_sensitivity_matrices} and
\eqref{eq: active jacobian channels affine}.

\begin{proposition}[Local-channel matrices for quadratic Darcy]
\label{prop:quadratic_darcy_compression}
For $\ell\in\Lambda_s$, define
\begin{equation}\label{eq:quadratic_derivative_insertions}
\mathsf P_{s,\ell,a_s}(z_s):=
\begin{cases}
\displaystyle
\begin{bmatrix}
C_s(z_s)\otimes\mathsf M_{a_s}&0\\
0&C_s(z_s)\otimes\mathsf M_{a_s}
\end{bmatrix},&\ell=0,\\[3mm]
\displaystyle
\begin{bmatrix}
0&-C_s'(z_s)\otimes\mathsf M_{a_s}\\
0&0
\end{bmatrix},&\ell=1,\\[3mm]
\displaystyle
\begin{bmatrix}
0&-C_s(z_s)\otimes\mathsf M_{a_s}\\
0&0
\end{bmatrix},&\ell=2,
\end{cases}
\end{equation}
with
$\mathsf P_{s,\ell,a_s}\in
\R^{2\varrho_{s-1}r_{s-1}\times2\varrho_sr_s}$.
Then the direction matrices of
Theorem~\ref{prop:local_channel_compression} are given by
\begin{equation}\label{eq:quadratic_reduced_direction_columns}
\Psi_{s,\ell,\mathrm{quad}}^{\mathrm{ch}}
\bigl(i,\overline{j_sa_s}\bigr)
=
\delta_{i_sj_s}\,
\mathcal Q_{<s}(z_{<s,i_{<s}})
\mathsf P_{s,\ell,a_s}(z_{s,j_s})
\mathcal Q_{>s}(z_{>s,i_{>s}}),
\end{equation}
where
$\Psi_{s,\ell,\mathrm{quad}}^{\mathrm{ch}}
\in\R^{N_\D\times n_sk_s}$.
Collecting the channels gives
\begin{equation}\label{eq:quadratic_reduced_channel_stacking}
\Psi_{s,\mathrm{quad}}^{\mathrm{ch}}
:=
\begin{cases}
\begin{bmatrix}
\Psi_{s,0,\mathrm{quad}}^{\mathrm{ch}}&
\Psi_{s,1,\mathrm{quad}}^{\mathrm{ch}}&
\Psi_{s,2,\mathrm{quad}}^{\mathrm{ch}}
\end{bmatrix},&s\le\dsp,\\[1mm]
\Psi_{s,0,\mathrm{quad}}^{\mathrm{ch}},&s>\dsp.
\end{cases}
\end{equation}
Consequently, the effective compressed dimension is given by
\begin{equation}\label{eq:quadratic_reduced_dimension}
\Neff^{\mathrm{ch}}
=
3\sum_{s=1}^{\dsp}n_sr_{s-1}r_s
+
\sum_{s=\dsp+1}^{d}n_sr_{s-1}r_s.
\end{equation}
\end{proposition}

\begin{proof}
For $s\le\dsp$, differentiating the $s$-th
block core in \eqref{eq:quadratic_operated_blocks} with respect to
$\theta_p^s$ gives
\begin{equation*}
\partial_{\theta_p^s}\mathcal B_s
=
\begin{bmatrix}
C_s\otimes\partial_{\theta_p^s}G_s&
-C_s'\otimes\partial_{z_s}\partial_{\theta_p^s}G_s
-C_s\otimes\partial_{z_s}^2\partial_{\theta_p^s}G_s\\
0&C_s\otimes\partial_{\theta_p^s}G_s
\end{bmatrix}.
\end{equation*}
For $s>\dsp$, the upper-right block is zero. Expanding in the
matrix units gives
\begin{equation}\label{eq:quadratic_block_channel_expansion}
\partial_{\theta_p^s}\mathcal B_s(z_s)
=
\sum_{\ell\in\Lambda_s}\sum_{a_s=1}^{k_s}
\mathsf P_{s,\ell,a_s}(z_s)
\bigl(\partial_{z_s}^{\ell}\partial_{\theta_p^s}G_s\bigr)(z_s)
(\alpha_{s-1},\alpha_s).
\end{equation}
Substitution into the derivative of \eqref{eq:quadratic_block_chain},
followed by \eqref{eq:quadratic_partial_contractions} and collocation
yields \eqref{eq:quadratic_reduced_direction_columns}.
The remaining assertions follow from
Theorem~\ref{prop:local_channel_compression}, since $\Lambda_s$ and
$k_s$ are unchanged.
\end{proof}
\noindent
\begin{remark}\label{remark on two parametric problems}
The affine and quadratic Darcy problems have the same matrix
$\Cmat^{\mathrm{ch}}$ in their respective factorizations and, hence, the same
effective compressed dimension; see \eqref{eq:darcy_reduced_dimension} and
\eqref{eq:quadratic_reduced_dimension}. Their assembly costs, however, differ.
In the affine case, \eqref{eq:affine_full_block_representation} contains an
explicit sum over the $R_\kappa=\dpa+1$ coefficient modes, so the compressed
Gramian assembly scales with $R_\kappa^2$. In the quadratic case, the
coefficient dependence is represented by the single coefficient TT
\eqref{eq:quadratic_coefficient_tt}, where the cost is governed by the coefficient TT bond
dimensions.
\end{remark}

Algorithm~\ref{alg:quadratic_darcy_assembly} provides the TT assembly for the quadratic parametric Darcy problem. As in the affine case, we note that this is merely an adaptation of Algorithm~\ref{alg:tt_assembly}.
\begin{algorithm}[H]
\caption{Channel-compressed TT assembly for quadratic parametric Darcy}
\label{alg:quadratic_darcy_assembly}
\small
\begin{algorithmic}[1]
\setlength{\itemsep}{0.25em}
\REQUIRE Current parameters $\theta$; factor grids $\X_s$; TT ranks
$(r_0,\ldots,r_d)$; coefficient cores $C_m$ from
\eqref{eq:quadratic_coefficient_cores}. 

\STATE Set $\Lambda_s$ by \eqref{eq:darcy_local_channels}; evaluate
$\mathcal D_{s,\ell}G_s$ for $\ell\in\Lambda_s$, compute
$\mathsf H_{s,\ell}$ from \eqref{eq:numerical_sensitivity_matrices}, and set
$\Cmat_s^{\mathrm{ch}}
=\operatorname{col}_{\ell\in\Lambda_s}\mathsf H_{s,\ell}$.

\STATE Form $\mathcal B_m$ from
\eqref{eq:quadratic_operated_blocks} and 
$\mathsf P_{s,\ell,a_s}$ from
\eqref{eq:quadratic_derivative_insertions}.

\STATE Form the residual \eqref{eq:quadratic_darcy_residual} in the
block-TT form \eqref{eq:quadratic_block_chain}.

\STATE Compute the model/model interfaces by
\eqref{eq:TT_recursions_interfaces} and the model/residual interfaces by
\eqref{eq:tt_residual_interfaces}.

\FOR{$s=1,\ldots,d$}
  \STATE Assemble $\Upsilon_{ss}^{\mathrm{ch}}$ by the diagonal
  contraction \eqref{eq:tt_gramian_diagonal_final}, replacing the active
  matrix units by $\mathsf P_{s,\ell,a_s}$.

  \STATE Assemble $\rho_{\mathrm c,s}^{\mathrm{ch}}$ by
    \eqref{eq:tt_gradient_local_matrix}--\eqref{eq:tt_compressed_rhs}
    using $\mathsf P_{s,\ell,a_s}$ at the active coordinate.

  \STATE For every $(\ell,i_s,a_s)$ initialize the bridge as in
  Algorithm~\ref{alg:tt_assembly} with $\mathsf M_{a_s}$ replaced by
  $\mathsf P_{s,\ell,a_s}$.

  \FOR{$t=s+1,\ldots,d$}
    \STATE Assemble $\Upsilon_{st}^{\mathrm{ch}}$ by
    \eqref{eq:tt_gramian_offdiagonal_final} using
    $\mathsf P_{t,\ell',a_t'}$ at the second active coordinate, and set
    $\Upsilon_{ts}^{\mathrm{ch}}
    \leftarrow(\Upsilon_{st}^{\mathrm{ch}})^\top$.

    \IF{$t<d$}
      \STATE Advance the bridge through
      $(\mathcal B_t,\mathcal B_t)$ by the forward transfer.
    \ENDIF
  \ENDFOR
\ENDFOR

\STATE Stack the channel blocks in the order $(s,\ell,i_s,a_s)$,
compute
$\nabla\widehat{\mathcal L}
=(\Cmat^{\mathrm{ch}})^\top\rho_{\mathrm c}^{\mathrm{ch}}$
as in \eqref{eq:compressed_gradient}, and evaluate the loss by
double-layer contractions over all signed residual-train pairs.

\ENSURE
$\{\Cmat_s^{\mathrm{ch}}\}_{s=1}^d$,
$\Upsilon^{\mathrm{ch}}$,
$\rho_{\mathrm c}^{\mathrm{ch}}$,
$\nabla\widehat{\mathcal L}(\theta)$, and
$\widehat{\mathcal L}(\theta)$ for Algorithm~\ref{alg:woodbury}, with
$\Neff^{\mathrm{ch}}$ given by
\eqref{eq:quadratic_reduced_dimension}.
\end{algorithmic}
\end{algorithm}

\subsubsection{Parametric Darcy problem: numerical results}
We proceed with the numerical ablation study. Both the affine and the
quadratic coefficient are tested with the same configuration. We use a TT-PINN
with boundary ranks $r_0=r_d=1$ and uniform interior rank $r=2$. Each univariate
TT core is a fully connected $\tanh$ network
$1\to128\to128\to128\to r_{m-1}r_m$, with $33\,796$ parameters for the interior
cores ($4$ outputs) and $33\,538$ for the boundary cores ($2$ outputs), so that
$P=26\times 33\,796+2\times 33\,538=945\,772$ for $d=d_x+d_\xi=28$. Per
coordinate we draw $n=64$ interior points independently from the uniform
distribution, fixed once for the whole run. For both problems
the effective compressed dimension is given by
\begin{align*}
\Neff=3\sum_{s=1}^{d_x}n\,r_{s-1}r_s+\sum_{s=d_x+1}^{d}n\,r_{s-1}r_s=8192.
\end{align*}
Observe that $\Neff \ll P$ in this setting, as well. Therefore, the compressed Gauss--Newton solve is used with
$\mu=10^{-4}$. The ablation consists of $200$ Gauss--Newton iterations for each of $10$ random seeds. Relative $H^1$ errors are computed along the training process on deterministic midpoint factor grids with $n_{\rm err}=256$ points per coordinate.

The convergence curves in Fig.~\ref{fig:ablation_darcy} show the median over the
ten seeds, with the shaded regions indicating the specified interquartile range. We observe that both problems attain high accuracy, the affine case reaching a relative error about
one order of magnitude below the quadratic case, as expected from the greater
parametric complexity of the latter. Learning either map takes less than a
minute, which is highly competitive, while high accuracy is the one expected from the Gauss--Newton approach. The remaining
difference in wall-clock time for these two problems is due to the assembly costs rather than to the
size of the compressed solve; cf.
Remark~\ref{remark on two parametric problems}. We display several solution instances for different parameter values in Fig.~\ref{fig:collage_darcy_affine} for the affine coefficient and Fig.~\ref{fig:collage_darcy_quadratic} for the quadratic-interaction coefficient.

\begin{figure}[H]
\centering
\includegraphics[width=1\linewidth]{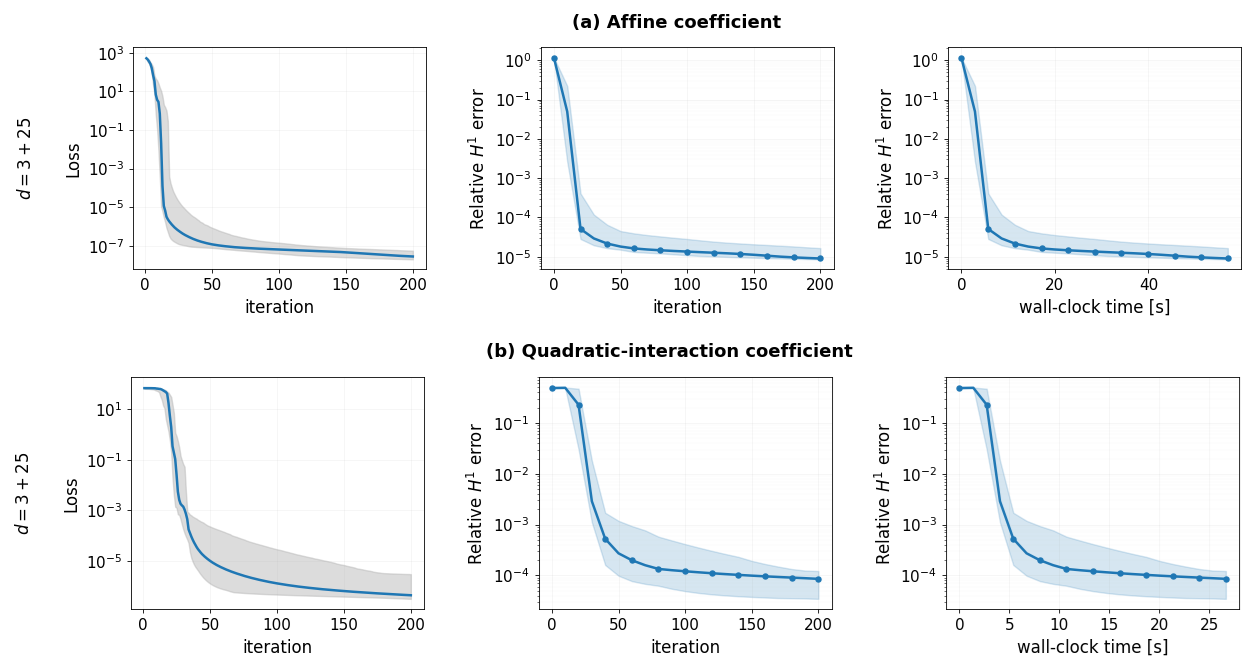}
\caption{Darcy problem: compressed Gauss--Newton for (a) the affine coefficient and (b) the quadratic-interaction
coefficient. Solid curves show the median over ten seeds, shaded regions the
interquartile range between the 25th and 75th percentiles. From left to right:
loss, relative $H^1$ error, and relative $H^1$ error versus wall-clock time; the
first two are plotted against the Gauss--Newton iteration count.}
\label{fig:ablation_darcy}
\end{figure}

\begin{figure}[H]
\centering
\includegraphics[width=0.75\linewidth]{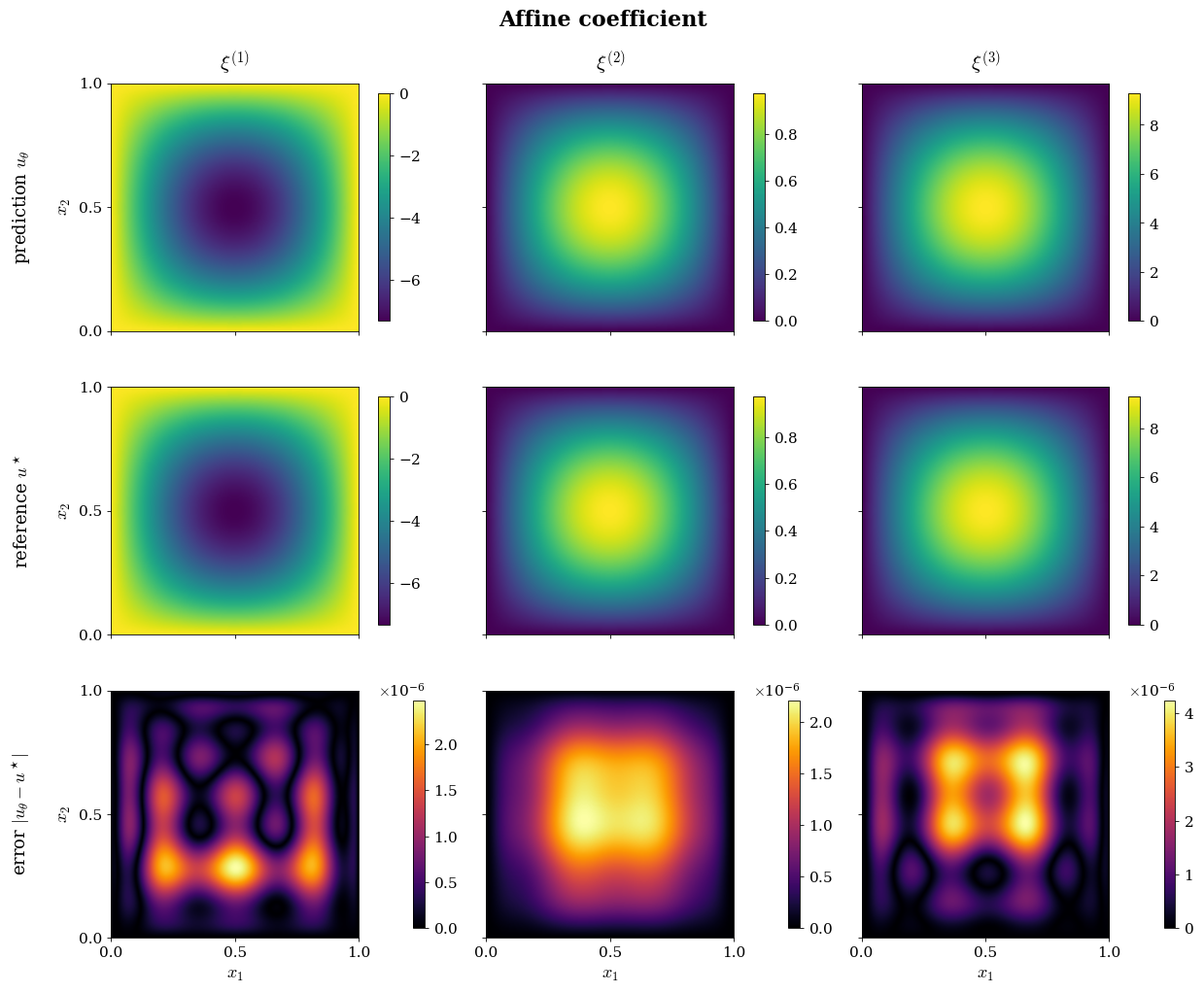}
\caption{Affine coefficient: Gauss--Newton TT-PINN prediction (top), reference solution
(middle) and pointwise absolute error (bottom) on the slice $x_3=0.5$, for three
random parameters $\xi^{(k)}\sim\operatorname{Unif}(-1,1)^{\dpa}$.}
\label{fig:collage_darcy_affine}
\end{figure}

\begin{figure}[H]
\centering
\includegraphics[width=0.75\linewidth]{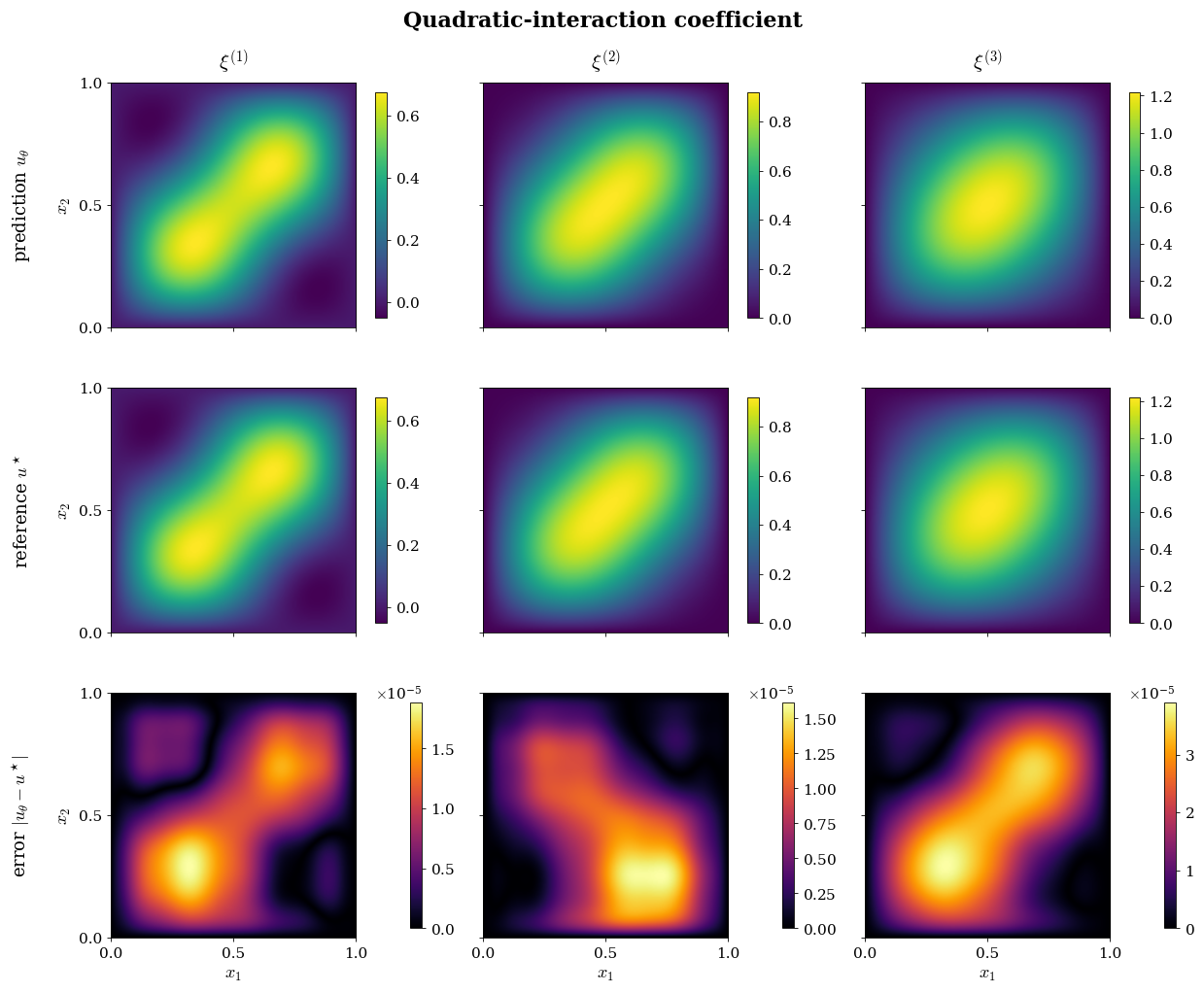}
\caption{Quadratic-interaction coefficient: Gauss--Newton TT-PINN prediction (top), reference
solution (middle) and pointwise absolute error (bottom) on the slice
$x_3=0.35$, for three random parameters
$\xi^{(k)}\sim\operatorname{Unif}(-1,1)^{\dpa}$. }
\label{fig:collage_darcy_quadratic}
\end{figure}

For comparison, we train the same TT-PINN with Adam. The Adam setup here is the same as in the heat equation example. The corresponding curves are shown in Fig.~\ref{fig:ablation_darcy_adam}. For both coefficients, Adam remains near the initial error level for about $30,000$ iterations before entering a rapid decay phase, reaching median relative $H^1$ errors of about $6\cdot10^{-3}$ for the affine and $1.6\cdot10^{-2}$ for the quadratic case. In contrast, compressed Gauss--Newton reaches $10^{-5}$ and $8\cdot10^{-5}$ error accuracy after $200$ iterations, respectively, while requiring five to eight times less wall-clock time ($57 \,\mathrm{s}$ and $26 \,\mathrm{s}$ versus $310 \,\mathrm{s}$ and $220 \,\mathrm{s}$). Therefore, the TT-compressed Gauss--Newton approach strikingly outperforms TT-Adam on these two problems.

\begin{figure}[H]
\centering
\includegraphics[width=1\linewidth]{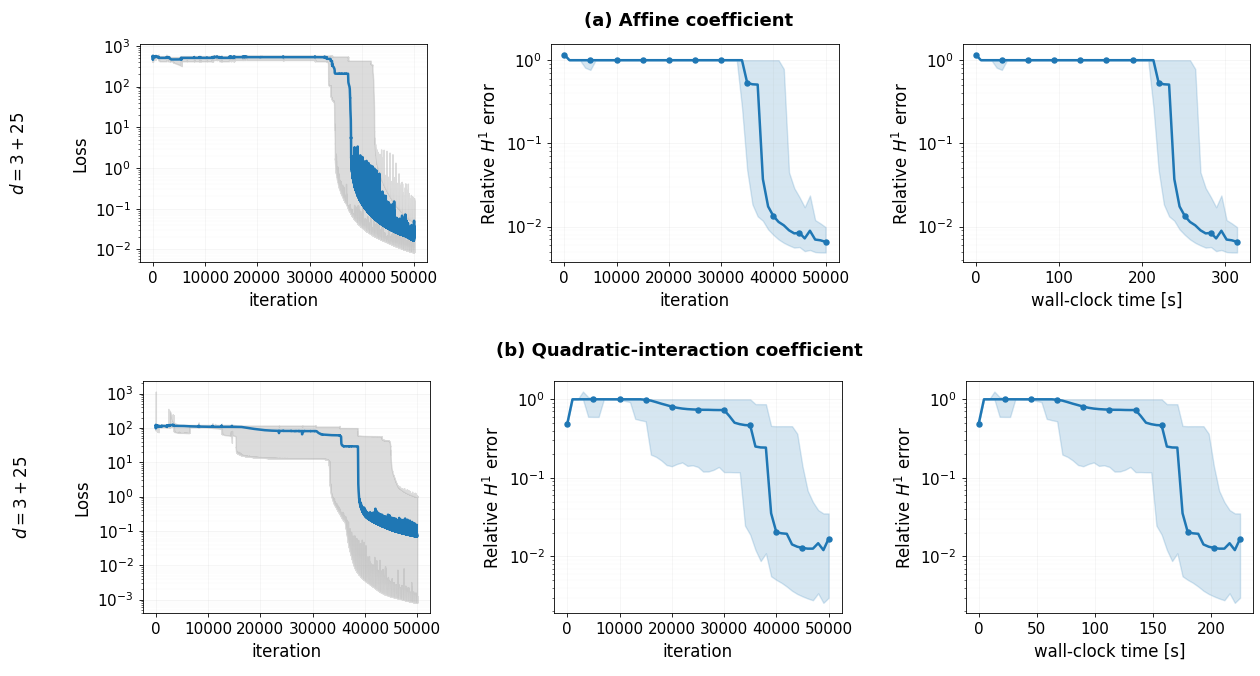}
\caption{Darcy problem: TT-Adam baseline for (a) the affine coefficient and (b) the
quadratic-interaction coefficient. Solid curves show the median over ten seeds,
shaded regions the interquartile range between the 25th and 75th percentiles.
From left to right: loss, relative $H^1$ error, and relative $H^1$ error versus
wall-clock time; the first two are plotted against the Adam iteration count.}
\label{fig:ablation_darcy_adam}
\end{figure}

\section{Conclusion and outlook}
\label{sec:conclusion}

In this work, we developed an abstract framework for tensor-structured
curvature-aware optimization of PINNs and, as an application of this framework,
derived full algebraic characterizations for the canonical polyadic and tensor train formats. The proposed Gauss--Newton solve is performed in the effective compressed
residual space, whose dimension is much smaller than that of the neural
parameter space in many applications of interest. Numerical experiments demonstrate
the high accuracy of the Gauss--Newton method and, thanks to tensor train compression,
which is the preferred format in this work, training times measured in seconds rather
than minutes on our high-dimensional benchmarks.


Our methodology assumes that the underlying high-dimensional parametric
PDE operator is separable. Within our problem setting, it will therefore be important to develop new or leverage existing low-rank approximation techniques to obtain separable approximations of non-separable operators with appropriate error control.  Further exploitation of the resulting nested Gramian
representations will also be of interest, given the wide range of
compression and structured matrix algebra techniques available in tensor
methods. Finally, joint adaptivity in tensor ranks and collocation grids
will aim to balance accuracy and computational cost while keeping the
effective compressed dimension small. These directions, together with
applications to more challenging uncertainty quantification problems,
will be the focus of our future work in this area.

\section*{Data availability}
The supporting code is available at \url{https://github.com/detzer0/tt-compressed-pinns}.

\section*{Acknowledgements} D. Korolev acknowledges the support of the Federal Ministry of Education and Research, Germany (funding reference: 01IS24081) under project HybridSolver. M. Eigel acknowledges funding from the German Federal Ministry of Education and Research, grant number FKZ 13N17160, ``Verbundprojekt: Quantum Read-Once-Memory -- Verwandlung von klassischen Daten zu Quantenzuständen -- Teilvorhaben: Quantenschaltkreis-Optimierung von Quantenzuständen durch Tensor-Netzwerke (QOQ-tn)'' and from the Deutsche Forschungsgemeinschaft (DFG, German Research Foundation) under Germany's Excellence Strategy -- The Berlin Mathematics Research Center MATH+ (EXC-2046/2, project ID: 390685689, Project PaA-7).

\section{Appendix: channel-compressed TT boundary terms}
\label{app:tt_boundary_contractions}

In this appendix, we propose an efficient way of assembling the Gauss--Newton system associated with the boundary term in \eqref{eq:continuous_loss}. To this end, we decompose the continuous Dirichlet loss into face contributions, discretize them on the grids \eqref{eq:boundary_face_grid}, and derive the corresponding TT contractions in a common sensitivity space.

\subsection{Boundary loss discretization}
By Remark~\ref{rem:trace}, the lateral boundary $\Lat$ consists of the
$2\dsp$ faces $\Lat_b^\sigma$, where $b=1,\ldots,\dsp$ and
$\sigma\in\{-,+\}$. For each pair $(b,\sigma)$, define
\begin{equation}\label{eq:continuous_face_residual}
\begin{aligned}
\mathcal L_{b,\sigma}(\theta)
:=\frac12\|r_{b,\sigma}(\theta)\|_{L^2(\Lat_b^\sigma)}^2, \qquad
r_{b,\sigma}(\theta)
:=\sqrt{\lambda_{\mathrm b}}
\left(
\operatorname{tr}_{\Lat_b^\sigma}u_\theta
-g|_{\Lat_b^\sigma}
\right).
\end{aligned}
\end{equation}
The boundary term in \eqref{eq:continuous_loss} then satisfies
\begin{equation}\label{eq:boundary_loss_face_decomposition}
\frac{\lambda_{\mathrm b}}2
\|\tr u_\theta-g\|_{L^2(\Lat)}^2
=
\sum_{b=1}^{\dsp}\sum_{\sigma\in\{-,+\}}
\mathcal L_{b,\sigma}(\theta).
\end{equation}
It therefore suffices to derive the contribution of one fixed pair
$(b,\sigma)$ and then sum over all pairs. For this face, the product
structure in Remark~\ref{rem:trace} gives
\begin{equation}\label{eq:continuous_face_integral}
\begin{aligned}
\mathcal L_{b,\sigma}(\theta)
=
\frac12
\int_{\D_1}\cdots
\int_{\D_{b-1}}
\int_{\D_{b+1}}\cdots
\int_{\D_d}
&\bigl|
r_{b,\sigma}(\theta)
(z_1,\ldots,z_{b-1},z_b^\sigma,z_{b+1},\ldots,z_d)
\bigr|^2\\
&\hspace{10mm}\mathrm dz_d\cdots\mathrm dz_{b+1}
\,\mathrm dz_{b-1}\cdots\mathrm dz_1 .
\end{aligned}
\end{equation}
We retain the factor grids $\X_m$ of \eqref{eq:factor_grids}.
and the boundary grids  in \eqref{eq:boundary_face_grid}, and we use equal-weight collocation averages as in
\eqref{eq:discrete_loss}. Applying these averages coordinate by coordinate to \eqref{eq:continuous_face_integral} gives
\begin{equation}\label{eq:discrete_face_quadrature}
\widehat{\mathcal L}_{b,\sigma}(\theta)
:=\frac{1}{2N_{b,\sigma}}
\sum_{i_{-b}}
\bigl|r_{b,\sigma}(\theta)
(z_{1,i_1},\ldots,z_b^\sigma,\ldots,z_{d,i_d})\bigr|^2,
\end{equation}
where we use the shortcut notation
\begin{equation}\label{eq:face_multiindex_sum}
\sum_{i_{-b}}
:=
\sum_{i_1=1}^{n_1}\cdots
\sum_{i_{b-1}=1}^{n_{b-1}}
\sum_{i_{b+1}=1}^{n_{b+1}}\cdots
\sum_{i_d=1}^{n_d}.
\end{equation}
Collecting the residual values in the lexicographic order of
\eqref{eq:multiindex}, we obtain
\begin{equation}\label{eq:boundary_residual_loss}
\begin{aligned}
\widehat{\mathcal L}_{b,\sigma}(\theta)
=\frac{1}{2N_{b,\sigma}}\|\rvec_{b,\sigma}(\theta)\|_2^2 , \qquad \rvec_{b,\sigma}(\theta)
:=\underset{z_i\in\X_{\Lat_b^\sigma}}{\operatorname{col}}
\bigl(r_{b,\sigma}(\theta)(z_i)\bigr)
\in\R^{N_{b,\sigma}}, \quad .
\end{aligned}
\end{equation}
Define
$\J_{b,\sigma}:=\diff_\theta\rvec_{b,\sigma}(\theta)
\in\R^{N_{b,\sigma}\times P}$. Then
\begin{equation}\label{eq:face_gradient_gramian}
\begin{aligned}
\nabla\widehat{\mathcal L}_{b,\sigma}
&=\frac{1}{N_{b,\sigma}}\J_{b,\sigma}^\top\rvec_{b,\sigma}
\in\R^P, \qquad
\DGram_{b,\sigma}
:=\frac{1}{N_{b,\sigma}}\J_{b,\sigma}^\top\J_{b,\sigma}
\in\R^{P\times P}.
\end{aligned}
\end{equation}
Following the channel factorization approach of
Theorem~\ref{prop:local_channel_compression}, one could factorize the
interior PDE and face Jacobians separately as
\begin{align*}
\J_{\rvec}
=\Psi_\D^{\mathrm{ch}}\Cmat_\D^{\mathrm{ch}},
\qquad
\J_{b,\sigma}
=\Psi_{b,\sigma}^{\mathrm{ch}}\Cmat_{b,\sigma}^{\mathrm{ch}},
\end{align*}
where
$\Psi_\D^{\mathrm{ch}}\in
\R^{N_\D\times\Neff^{\mathrm{ch},\D}}$,
$\Cmat_\D^{\mathrm{ch}}\in
\R^{\Neff^{\mathrm{ch},\D}\times P}$,
$\Psi_{b,\sigma}^{\mathrm{ch}}\in
\R^{N_{b,\sigma}\times\Neff^{\mathrm{ch},b,\sigma}}$, and
$\Cmat_{b,\sigma}^{\mathrm{ch}}\in
\R^{\Neff^{\mathrm{ch},b,\sigma}\times P}$. To compute the Gauss--Newton Gramian of the total discrete loss, we first
form the Jacobian of the vertically concatenated interior and face residuals,
whose separate channel factorizations give
\begin{align*}
\begin{bmatrix}
\J_{\rvec}\\
\J_{1,-}\\
\J_{1,+}\\
\vdots\\
\J_{\dsp,+}
\end{bmatrix}
=
\begin{bmatrix}
\Psi_\D^{\mathrm{ch}} & 0 & 0 & \cdots & 0\\
0 & \Psi_{1,-}^{\mathrm{ch}} & 0 & \cdots & 0\\
0 & 0 & \Psi_{1,+}^{\mathrm{ch}} & \cdots & 0\\
\vdots & \vdots & \vdots & \ddots & \vdots\\
0 & 0 & 0 & \cdots & \Psi_{\dsp,+}^{\mathrm{ch}}
\end{bmatrix}
\begin{bmatrix}
\Cmat_\D^{\mathrm{ch}}\\
\Cmat_{1,-}^{\mathrm{ch}}\\
\Cmat_{1,+}^{\mathrm{ch}}\\
\vdots\\
\Cmat_{\dsp,+}^{\mathrm{ch}}
\end{bmatrix}.
\end{align*}
Defining
\begin{align*}
\J^{\mathrm{ch,sep}}
&:=
\operatorname{col}
\bigl(
\J_{\rvec},
\J_{1,-},
\J_{1,+},
\ldots,
\J_{\dsp,+}
\bigr), \qquad
\Cmat^{\mathrm{ch,sep}}
:=
\operatorname{col}
\bigl(
\Cmat_\D^{\mathrm{ch}},
\Cmat_{1,-}^{\mathrm{ch}},
\Cmat_{1,+}^{\mathrm{ch}},
\ldots,
\Cmat_{\dsp,+}^{\mathrm{ch}}
\bigr),
\end{align*}
the corresponding compressed dimension is then given by
\begin{align*}
\Neff^{\mathrm{ch,sep}}
=
\Neff^{\mathrm{ch},\D}
+
\sum_{b=1}^{\dsp}\sum_{\sigma\in\{-,+\}}
\Neff^{\mathrm{ch},b,\sigma}.
\end{align*}
However, this dimension grows additively with the number of faces, although the
individual channel factorizations share many local sensitivities. Indeed,
for the face $\Lat_b^\sigma$ and any $s\ne b$, the same factor-grid
sensitivities $\mathsf H_{s,0}$ already occur in the interior
factorization, while only the $b$-th coordinate additionally requires the
endpoint sensitivity at $z_b^\sigma$. Thus, the above construction counts
common sensitivity blocks repeatedly, increasing the compressed dimension
unnecessarily. To avoid this duplication, we seek a common matrix
$\Cmat^{\mathrm{ch}}\in\R^{\Neff^{\mathrm{ch}}\times P}$ such that
\begin{align}\label{eq:common_channel_fact}
\J_{\rvec}
=\Psi^{\mathrm{ch},\D}\Cmat^{\mathrm{ch}},
\qquad
\J_{b,\sigma}
=\Psi^{\mathrm{ch},b,\sigma}\Cmat^{\mathrm{ch}}.
\end{align}
This construction will be outlined in the next section. 

\subsection{Efficient boundary Gramian factorization}

We now outline the construction of 
\eqref{eq:common_channel_fact}. To account for the endpoint evaluations
required by the face residuals, we introduce the following factor evaluation
grids
\begin{equation}\label{eq:boundary_face_appendix}
\widehat\X_m
:=
\X_m\cup\mathcal E_m
=
\{z_{m,i_m}\}_{i_m=1}^{\widehat n_m},
\qquad
\widehat n_m:=|\widehat\X_m|,
\end{equation}
where $\mathcal E_m$ contains the endpoints required by the traces. Here,
the first $n_m$ points are indexed according to \eqref{eq:factor_grids}. For $b\le\dsp$ and $\sigma\in\{-,+\}$, let $e_b^\sigma$ denote the index of $z_b^\sigma$ in $\widehat\X_b$, i.e., $z_{b,e_b^\sigma}=z_b^\sigma$.

Since a Dirichlet trace residual in \eqref{eq:continuous_face_residual} involves the function itself, we include
$\mathcal D_{s,0}=I$ among the local operators. If $0\notin\Lambda_s$,
we augment $\Lambda_s$ by this operator and set
$b_{s,0}^{(\beta)}=0$ in
\eqref{eq:local_differential_channel_representation}, so that the
representation of the PDE operator $\A$ remains unchanged.
For the TT ansatz \eqref{eq:tt_ansatz}, set $r_0=r_d=1$,
$k_s=r_{s-1}r_s$, $a_s=\overline{\alpha_{s-1}\alpha_s}$, and
$L_s:=|\Lambda_s|$. The sensitivities
\eqref{eq:local_channel_sensitivities}, specialized as in
\eqref{eq:numerical_sensitivity_matrices}, are evaluated on
$\widehat\X_s$:
\begin{equation}\label{eq:boundary_face_sensitivities}
\begin{aligned}
\mathsf H_{s,\ell}(\overline{i_sa_s},p)
:=
\bigl(
\mathcal D_{s,\ell}\partial_{\theta_p^s}G_s
\bigr)(z_{s,i_s})(\alpha_{s-1},\alpha_s), \quad
\mathsf H_{s,\ell}
\in\R^{\widehat n_s k_s\times P_s},
\end{aligned}
\end{equation}
where $\ell\in\Lambda_s$ and $1\le i_s\le\widehat n_s$. Following \eqref{eq:local_channel_stacking}, define
\begin{equation}\label{eq:boundary_common_sensitivities}
\begin{aligned}
\Cmat_s^{\mathrm{ch}}
&:=
\operatorname{col}_{\ell\in\Lambda_s}\mathsf H_{s,\ell}
\in\R^{N_s^{\mathrm{ch}}\times P_s},
&N_s^{\mathrm{ch}}&:=L_s\widehat n_s k_s,\\
\Cmat^{\mathrm{ch}}
&:=
\diag(\Cmat_1^{\mathrm{ch}},\ldots,\Cmat_d^{\mathrm{ch}})
\in\R^{\Neff^{\mathrm{ch}}\times P},
&\Neff^{\mathrm{ch}}&:=\sum_{s=1}^dN_s^{\mathrm{ch}}.
\end{aligned}
\end{equation}
Thus, the common compressed dimension is
$\Neff^{\mathrm{ch}}=\sum_sL_s\widehat n_sk_s$, compared with
$\sum_sL_sn_sk_s$ for the interior residual alone, where the increase
accounts for the endpoint evaluations required by the face residuals.

For the face $\Lat_b^\sigma$, define
\begin{align*}
\mathcal I_{b,\sigma}
:=
\bigl\{
i=(i_1,\ldots,i_d):
i_b=e_b^\sigma,\;
1\le i_m\le n_m,\ m\ne b
\bigr\}.
\end{align*}
Let $\mathsf M_{a_s}\in\R^{r_{s-1}\times r_s}$ be the matrix units defined in
\eqref{eq:tt_matrix_units}. Define
$\Psi_{s,\ell}^{\mathrm{ch},b,\sigma}
\in\R^{N_{b,\sigma}\times\widehat n_s k_s}$ by
\begin{equation}\label{eq:boundary_channel_directions}
\begin{aligned}
\Psi_{s,\ell}^{\mathrm{ch},b,\sigma}
\bigl(i,\overline{j_sa_s}\bigr)
:={}&
\sqrt{\lambda_{\mathrm b}}\,
\delta_{\ell0} \,\delta_{i_sj_s}\,
\mathsf G_1(i_1)\cdots
\mathsf G_{s-1}(i_{s-1})
\mathsf M_{a_s}
\mathsf G_{s+1}(i_{s+1})\cdots
\mathsf G_d(i_d),
\end{aligned}
\end{equation}
where $i\in\mathcal I_{b,\sigma}$ and $1\le j_s\le\widehat n_s$. Here $\delta_{\ell0}$ makes all non-identity channel blocks vanish, while
$\delta_{i_sj_s}$ selects the corresponding factor-grid point, which for
$s=b$ is the endpoint $j_b=e_b^\sigma$. Stacking as in
\eqref{eq:local_channel_stacking}, set
\begin{align*}
\Psi_s^{\mathrm{ch},b,\sigma}
:=
[\Psi_{s,\ell}^{\mathrm{ch},b,\sigma}]_{\ell\in\Lambda_s}
\in\R^{N_{b,\sigma}\times N_s^{\mathrm{ch}}}, \qquad
\Psi^{\mathrm{ch},b,\sigma}
:=
[\Psi_1^{\mathrm{ch},b,\sigma}\ \cdots\
\Psi_d^{\mathrm{ch},b,\sigma}]
\in\R^{N_{b,\sigma}\times\Neff^{\mathrm{ch}}}.
\end{align*}
Consequently, we obtain the factorization
\begin{equation}\label{eq:boundary_jacobian_factorization}
\J_{b,\sigma}
=
\Psi^{\mathrm{ch},b,\sigma}\Cmat^{\mathrm{ch}}
\in\R^{N_{b,\sigma}\times P}.
\end{equation}
For the interior Jacobian, the columns of
$\Psi^{\mathrm{ch},\D}\in\R^{N_\D\times\Neff^{\mathrm{ch}}}$
corresponding to the additional endpoint evaluations are zero. Hence the
interior and all face Jacobians use the same matrix
$\Cmat^{\mathrm{ch}}$. In view of \eqref{eq:face_gradient_gramian}, define
\begin{equation}\label{eq:boundary_compressed_quantities}
\begin{aligned}
\Upsilon^{\mathrm{ch},b,\sigma}
:=
\frac{1}{N_{b,\sigma}}
(\Psi^{\mathrm{ch},b,\sigma})^\top
\Psi^{\mathrm{ch},b,\sigma}
\in\R^{\Neff^{\mathrm{ch}}\times\Neff^{\mathrm{ch}}}, \quad
\rho_{\mathrm c}^{\mathrm{ch},b,\sigma}
:=
\frac{1}{N_{b,\sigma}}
(\Psi^{\mathrm{ch},b,\sigma})^\top
\rvec_{b,\sigma}(\theta)
\in\R^{\Neff^{\mathrm{ch}}}.
\end{aligned}
\end{equation}
Then the discrete boundary Gramian and the boundary gradient term are given by
\begin{align*}
\DGram_{b,\sigma}
=
(\Cmat^{\mathrm{ch}})^\top
\Upsilon^{\mathrm{ch},b,\sigma}
\Cmat^{\mathrm{ch}}, \qquad
\nabla\widehat{\mathcal L}_{b,\sigma}
=
(\Cmat^{\mathrm{ch}})^\top
\rho_{\mathrm c}^{\mathrm{ch},b,\sigma}.
\end{align*}
To evaluate these quantities without forming the direction matrices, we
define the forward and backward face transfers
\begin{align*}
\Tfwd_m^{b,\sigma}
&:
\R^{r_{m-1}\times r_{m-1}}
\to
\R^{r_m\times r_m},
&
\Tbwd_m^{b,\sigma}
&:
\R^{r_m\times r_m}
\to
\R^{r_{m-1}\times r_{m-1}}
\end{align*}
by the following formulas
\begin{equation}\label{eq:boundary_weighted_transfers}
\begin{aligned}
\Tfwd_m^{b,\sigma}(\mathsf X)
&:=
\begin{cases}
\displaystyle
\frac1{n_m}\sum_{j=1}^{n_m}
\mathsf G_m(j)^\top\mathsf X\mathsf G_m(j),
&m\ne b,\\[2mm]
G_b(z_b^\sigma)^\top\mathsf XG_b(z_b^\sigma),
&m=b,
\end{cases}\\
\Tbwd_m^{b,\sigma}(\mathsf Y)
&:=
\begin{cases}
\displaystyle
\frac1{n_m}\sum_{j=1}^{n_m}
\mathsf G_m(j)\mathsf Y\mathsf G_m(j)^\top,
&m\ne b,\\[2mm]
G_b(z_b^\sigma)\mathsf YG_b(z_b^\sigma)^\top,
&m=b,
\end{cases}
\end{aligned}
\end{equation}
which are the face counterparts of
\eqref{eq:forward_transfer_operator}--\eqref{eq:backward_transfer_operator}. Define
\begin{equation}\label{eq:boundary_interfaces}
\begin{aligned}
\mathsf L_{<1}^{b,\sigma}&=1,
&
\mathsf L_{<m+1}^{b,\sigma}
&=
\Tfwd_m^{b,\sigma}(\mathsf L_{<m}^{b,\sigma}),
&&m=1,\ldots,d-1,\\
\mathsf R_{>d}^{b,\sigma}&=1,
&
\mathsf R_{>m-1}^{b,\sigma}
&=
\Tbwd_m^{b,\sigma}(\mathsf R_{>m}^{b,\sigma}),
&&m=d,\ldots,2.
\end{aligned}
\end{equation}
Then
$\mathsf L_{<s}^{b,\sigma}\in
\R^{r_{s-1}\times r_{s-1}}$ and
$\mathsf R_{>s}^{b,\sigma}\in\R^{r_s\times r_s}$.

Compared with the interior TT contractions, the only additional distinction
is whether the active coordinate $s$ coincides with the boundary coordinate
$b$. If $s\ne b$, the active coordinate is averaged over its original factor
grid. If $s=b$, it is fixed at the prescribed endpoint
$e_b^\sigma$. We encode these two cases by
\begin{equation}\label{eq:boundary_average_selector}
\mathsf D_s^{b,\sigma}
:=
\begin{cases}
\displaystyle
\frac1{n_s}
\diag\bigl(I_{n_s},0_{\widehat n_s-n_s}\bigr),
&s\ne b,\\[2mm]
\mathsf e_{e_b^\sigma}\mathsf e_{e_b^\sigma}^\top,
&s=b,
\end{cases}
\qquad
\mathsf D_s^{b,\sigma}
\in\R^{\widehat n_s\times\widehat n_s},
\end{equation}
where $\mathsf e_j\in\R^{\widehat n_s}$ denotes the $j$-th Euclidean
basis vector.

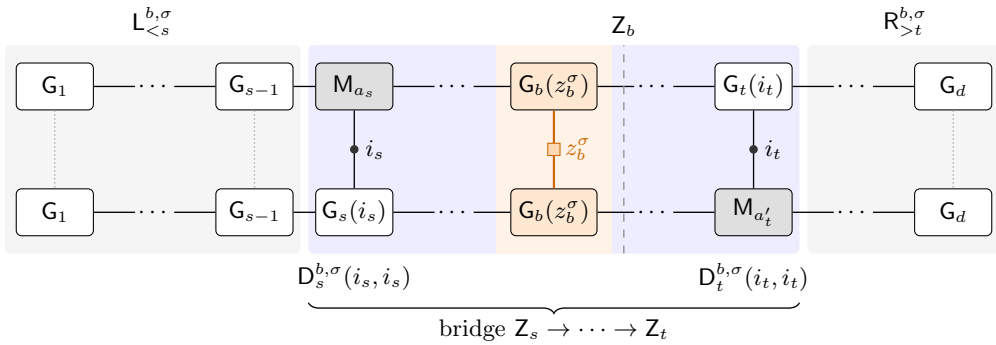
\begin{figure}[H]
\centering
\begin{tikzpicture}[
  x=1.32cm,y=1.30cm,
  core/.style={draw,rounded corners=2pt,minimum width=1.02cm,
               minimum height=0.62cm,inner sep=1pt,font=\footnotesize,fill=white},
  mu/.style={core,fill=gray!25},
  bcore/.style={core,fill=orange!18,minimum width=1.14cm,inner sep=0.5pt},
  dots/.style={font=\footnotesize,inner sep=1pt},
  bond/.style={line width=0.5pt},
  share/.style={densely dotted,gray!70,line width=0.5pt},
  bline/.style={draw=orange!80!black,line width=0.8pt},
  cut/.style={dashed,gray!75,line width=0.6pt},
  reg/.style={fill=gray!8,rounded corners=2pt},
  bridge/.style={fill=blue!7,rounded corners=2pt},
  breg/.style={fill=orange!10},
  pin/.style={circle,fill=black!80,inner sep=0pt,minimum size=3.2pt},
  bpin/.style={rectangle,draw=orange!80!black,fill=orange!30,
               minimum size=4.6pt,inner sep=0pt},
  lab/.style={font=\footnotesize,inner sep=1pt}
]
\fill[reg]    (0.50,-0.42) rectangle (3.46,1.70);
\fill[bridge] (3.54,-0.42) rectangle (8.46,1.70);
\fill[breg]   (5.42,-0.42) rectangle (6.58,1.70);
\fill[reg]    (8.54,-0.42) rectangle (10.50,1.70);

\node[core]  (t1)  at (1,1.28)  {$\mathsf G_1$};
\node[dots]  (t2)  at (2,1.28)  {$\cdots$};
\node[core]  (t3)  at (3,1.28)  {$\mathsf G_{s-1}$};
\node[mu]    (t4)  at (4,1.28)  {$\mathsf M_{a_s}$};
\node[dots]  (t5)  at (5,1.28)  {$\cdots$};
\node[bcore] (t6)  at (6,1.28)  {$\mathsf G_b(z_b^\sigma)$};
\node[dots]  (t7)  at (7,1.28)  {$\cdots$};
\node[core]  (t8)  at (8,1.28)  {$\mathsf G_t(i_t)$};
\node[dots]  (t9)  at (9,1.28)  {$\cdots$};
\node[core]  (t10) at (10,1.28) {$\mathsf G_d$};

\node[core]  (b1)  at (1,0)  {$\mathsf G_1$};
\node[dots]  (b2)  at (2,0)  {$\cdots$};
\node[core]  (b3)  at (3,0)  {$\mathsf G_{s-1}$};
\node[core]  (b4)  at (4,0)  {$\mathsf G_s(i_s)$};
\node[dots]  (b5)  at (5,0)  {$\cdots$};
\node[bcore] (b6)  at (6,0)  {$\mathsf G_b(z_b^\sigma)$};
\node[dots]  (b7)  at (7,0)  {$\cdots$};
\node[mu]    (b8)  at (8,0)  {$\mathsf M_{a_t'}$};
\node[dots]  (b9)  at (9,0)  {$\cdots$};
\node[core]  (b10) at (10,0) {$\mathsf G_d$};

\foreach \a/\bb in {1/2,2/3,3/4,4/5,5/6,6/7,7/8,8/9,9/10} {
  \draw[bond] (t\a.east)--(t\bb.west);
  \draw[bond] (b\a.east)--(b\bb.west);
}

\foreach \k in {1,3,10} \draw[share] (t\k.south)--(b\k.north);

\draw[bond] (t4.south)--(b4.north);
\node[pin] at (4,0.64) {};
\node[lab,anchor=west] at (4.08,0.64) {$i_s$};

\draw[bline] (t6.south)--(b6.north);
\node[bpin] at (6,0.64) {};
\node[lab,anchor=west,text=orange!80!black] at (6.08,0.64) {$z_b^\sigma$};

\draw[bond] (t8.south)--(b8.north);
\node[pin] at (8,0.64) {};
\node[lab,anchor=west] at (8.08,0.64) {$i_t$};

\draw[cut] (6.70,-0.42)--(6.70,1.70);

\node[lab,anchor=south] at (1.98,1.76) {$\mathsf L_{<s}^{b,\sigma}$};
\node[lab,anchor=south] at (6.70,1.76) {$\mathsf Z_b$};
\node[lab,anchor=south] at (9.52,1.76) {$\mathsf R_{>t}^{b,\sigma}$};

\node[lab,anchor=north] at (4,-0.48) {$\mathsf D_s^{b,\sigma}(i_s,i_s)$};
\node[lab,anchor=north] at (8,-0.48) {$\mathsf D_t^{b,\sigma}(i_t,i_t)$};
\draw[decorate,decoration={brace,mirror,amplitude=4pt},line width=0.5pt]
  (3.54,-0.92)--(8.46,-0.92)
  node[lab,midway,below=5pt] {bridge $\mathsf Z_s\to\cdots\to\mathsf Z_t$};

\end{tikzpicture}
\caption{Double-layer contraction diagram for an off-diagonal boundary Gramian block with $s<b<t$.}
\label{fig:boundary_offdiag_contraction}
\end{figure}
\begin{proposition}[Algebraic characterization of the channel-compressed TT boundary terms]
\label{prop:tt_boundary_channel_contractions}
Fix $b\in\{1,\ldots,\dsp\}$ and $\sigma\in\{-,+\}$. The following statements hold.
\begin{enumerate}[label=\textup{(\roman*)}]

\item
For every $s$ and $\ell,\ell'\in\Lambda_s$, the diagonal channel blocks of
$\Upsilon^{\mathrm{ch},b,\sigma}$ satisfy
\begin{equation}\label{eq:boundary_gramian_diagonal}
\Upsilon_{ss}^{\mathrm{ch},b,\sigma,(\ell,\ell')}
=
\lambda_{\mathrm b}\,
\delta_{\ell0}\delta_{\ell'0}\,
\mathsf D_s^{b,\sigma}
\otimes
\bigl(
\mathsf L_{<s}^{b,\sigma}
\otimes
\mathsf R_{>s}^{b,\sigma}
\bigr)
\in
\R^{\widehat n_s k_s\times\widehat n_s k_s}.
\end{equation}

\item
Let $s<t$, and fix $1\le i_s\le\widehat n_s$,
$1\le i_t\le\widehat n_t$ and
$a_s=\overline{\alpha_{s-1}\alpha_s}$,
$a_t'=\overline{\alpha_{t-1}'\alpha_t'}$. Define
\begin{align}
\mathsf Z_s
&:=
\mathsf D_s^{b,\sigma}(i_s,i_s)
\mathsf M_{a_s}^\top
\mathsf L_{<s}^{b,\sigma}
\mathsf G_s(i_s)
\in\R^{r_s\times r_s},
\label{eq:boundary_bridge_initialization}\\
\mathsf Z_m
&:=
\Tfwd_m^{b,\sigma}(\mathsf Z_{m-1})
\in\R^{r_m\times r_m},
\qquad s<m<t,
\label{eq:boundary_bridge_propagation}\\
\mathsf Z_t
&:=
\mathsf D_t^{b,\sigma}(i_t,i_t)
\mathsf G_t(i_t)^\top
\mathsf Z_{t-1}
\mathsf M_{a_t'}
\in\R^{r_t\times r_t}.
\label{eq:boundary_bridge_termination}
\end{align}
Then
\begin{equation}\label{eq:boundary_gramian_offdiagonal}
\Upsilon_{st}^{\mathrm{ch},b,\sigma,(\ell,\ell')}
\bigl(\overline{i_sa_s},\overline{i_ta_t'}\bigr)
=
\lambda_{\mathrm b}\,
\delta_{\ell0}\delta_{\ell'0}
\bigl\langle
\mathsf Z_t,
\mathsf R_{>t}^{b,\sigma}
\bigr\rangle_{\mathrm F}.
\end{equation}
The propagation step is absent when $t=s+1$, and
$\Upsilon_{ts}^{\mathrm{ch},b,\sigma,(\ell',\ell)}
=(\Upsilon_{st}^{\mathrm{ch},b,\sigma,(\ell,\ell')})^\top$.

\item
Suppose additionally that the discrete face residual admits the TT
representation
\begin{equation}\label{eq:boundary_residual_tt}
\rvec_{b,\sigma}(\theta)(i)
=
\mathsf Y_1^{b,\sigma}(i_1)\cdots
\mathsf Y_d^{b,\sigma}(i_d),
\qquad i\in\mathcal I_{b,\sigma},
\end{equation}
where
$\mathsf Y_m^{b,\sigma}(i_m)\in\R^{q_{m-1}\times q_m}$ and $q_0=q_d=1$.
Define the mixed interfaces by
\begin{equation}\label{eq:boundary_mixed_interfaces}
\begin{aligned}
\mathsf L_{<1}^{b,\sigma,\rho}
&=1,
&
\mathsf L_{<m+1}^{b,\sigma,\rho}
&:=
\begin{cases}
\displaystyle
\frac1{n_m}\sum_{j=1}^{n_m}
\mathsf G_m(j)^\top
\mathsf L_{<m}^{b,\sigma,\rho}
\mathsf Y_m^{b,\sigma}(j),
&m\ne b,\\[2mm]
G_b(z_b^\sigma)^\top
\mathsf L_{<b}^{b,\sigma,\rho}
\mathsf Y_b^{b,\sigma}(e_b^\sigma),
&m=b,
\end{cases}
\\
\mathsf R_{>d}^{b,\sigma,\rho}
&=1,
&
\mathsf R_{>m-1}^{b,\sigma,\rho}
&:=
\begin{cases}
\displaystyle
\frac1{n_m}\sum_{j=1}^{n_m}
\mathsf G_m(j)
\mathsf R_{>m}^{b,\sigma,\rho}
\mathsf Y_m^{b,\sigma}(j)^\top,
&m\ne b,\\[2mm]
G_b(z_b^\sigma)
\mathsf R_{>b}^{b,\sigma,\rho}
\mathsf Y_b^{b,\sigma}(e_b^\sigma)^\top,
&m=b.
\end{cases}
\end{aligned}
\end{equation}
Set
\begin{equation}\label{eq:boundary_residual_local_matrix}
\mathsf P_s^{b,\sigma}(i_s)
:=
\mathsf D_s^{b,\sigma}(i_s,i_s)
\mathsf L_{<s}^{b,\sigma,\rho}
\mathsf Y_s^{b,\sigma}(i_s)
\bigl(\mathsf R_{>s}^{b,\sigma,\rho}\bigr)^\top.
\end{equation}
Then
\begin{equation}\label{eq:boundary_compressed_residual}
\rho_{\mathrm c,s}^{\mathrm{ch},b,\sigma,(\ell)}
\bigl(\overline{i_s\alpha_{s-1}\alpha_s}\bigr)
=
\sqrt{\lambda_{\mathrm b}}\,
\delta_{\ell0}\,
\mathsf P_s^{b,\sigma}(i_s)(\alpha_{s-1},\alpha_s).
\end{equation}
Finally, the face loss $\widehat{\mathcal L}_{b,\sigma}
=\frac12\,\mathsf S_d^{b,\sigma}$ is obtained by the same contraction with both layers replaced by the residual train:
\begin{align*}
\mathsf S_0^{b,\sigma}
=1, \qquad
\mathsf S_m^{b,\sigma}
=
\begin{cases}
\displaystyle
\frac1{n_m}\sum_{j_m=1}^{n_m}
\bigl(\mathsf Y_m^{b,\sigma}(j_m)\bigr)^\top
\mathsf S_{m-1}^{b,\sigma}
\mathsf Y_m^{b,\sigma}(j_m),
&m\ne b,\\[2mm]
\bigl(\mathsf Y_b^{b,\sigma}(e_b^\sigma)\bigr)^\top
\mathsf S_{b-1}^{b,\sigma}
\mathsf Y_b^{b,\sigma}(e_b^\sigma),
&m=b,
\end{cases}
\end{align*}
\end{enumerate}
\end{proposition}

\begin{proof}
The proof follows the double-layer TT contraction arguments of
Theorem~\ref{thm:tt_gramian_forward_backward} and
Proposition~\ref{prop:tt_compressed_gradient}, but now using the face transfers
\eqref{eq:boundary_weighted_transfers}. The distinction at the active
coordinate is encoded by $\mathsf D_s^{b,\sigma}$ in
\eqref{eq:boundary_average_selector}.

For a diagonal block, contracting over the coordinates $m<s$ and $m>s$
gives the interfaces $\mathsf L_{<s}^{b,\sigma}$ and
$\mathsf R_{>s}^{b,\sigma}$ from \eqref{eq:boundary_interfaces}.
At the active coordinate, the face contraction of the factors
$\delta_{i_sj_s}$ in \eqref{eq:boundary_channel_directions} gives
$\mathsf D_s^{b,\sigma}$ from \eqref{eq:boundary_average_selector}.
By the same Kronecker-product argument as in the interior case, 
\begin{align*}
\frac{1}{N_{b,\sigma}}
\bigl(\Psi_{s,\ell}^{\mathrm{ch},b,\sigma}\bigr)^\top
\Psi_{s,\ell'}^{\mathrm{ch},b,\sigma}
&=
\lambda_{\mathrm b}\,
\delta_{\ell0}\delta_{\ell'0}\,
\mathsf D_s^{b,\sigma}
\otimes
\bigl(
\mathsf L_{<s}^{b,\sigma}
\otimes
\mathsf R_{>s}^{b,\sigma}
\bigr),
\end{align*}
which is \eqref{eq:boundary_gramian_diagonal}.

We next consider an off-diagonal block with $s<t$ and follow the
double-layer contraction in
Fig.~\ref{fig:boundary_offdiag_contraction} from left to right.

\medskip
\noindent\emph{Step 1:}
Contracting all coordinates $m<s$ by the face transfers
\eqref{eq:boundary_weighted_transfers} gives
\begin{align*}
\mathsf L_{<s}^{b,\sigma}
=
\Tfwd_{s-1}^{b,\sigma}\circ\cdots\circ
\Tfwd_1^{b,\sigma}
\bigl(\mathsf L_{<1}^{b,\sigma}\bigr).
\end{align*}
If $b<s$, the corresponding transfer already evaluates the $b$-th
coordinate at $z_b^\sigma$.

\medskip
\noindent\emph{Step 2:}
At the first active coordinate $s$, the pinned update
\eqref{eq:pinned_update} from Lemma~\ref{lem:averaged_tt_recursion}
contributes $1/n_s$ if $s\ne b$, while for $s=b$ the coordinate is evaluated
only at the endpoint $z_b^\sigma$ with index $e_b^\sigma$. Thus,
\begin{align*}
\mathsf Z_s
=
\mathsf D_s^{b,\sigma}(i_s,i_s)
\mathsf M_{a_s}^{\top}
\mathsf L_{<s}^{b,\sigma}
\mathsf G_s(i_s),
\end{align*}
which is \eqref{eq:boundary_bridge_initialization}.

\medskip
\noindent\emph{Step 3:}
For $s<m<t$, the bridge is propagated by $\mathsf Z_m
=
\Tfwd_m^{b,\sigma}(\mathsf Z_{m-1}).$ If the bridge crosses the boundary coordinate $b$, the usual average is
replaced by $\mathsf Z_b
=
G_b(z_b^\sigma)^\top
\mathsf Z_{b-1}
G_b(z_b^\sigma)$, which is the $m=b$ case of
\eqref{eq:boundary_bridge_propagation}.

\medskip
\noindent\emph{Step 4:}
At the second active coordinate $t$, we proceed as in Step~2, with the two
layers interchanged. Thus,
\[
\mathsf Z_t
=
\mathsf D_t^{b,\sigma}(i_t,i_t)
\mathsf G_t(i_t)^\top
\mathsf Z_{t-1}
\mathsf M_{a_t'},
\]
which is \eqref{eq:boundary_bridge_termination}.

\medskip
\noindent\emph{Step 5:}
Finally, the adjoint identity for the face transfers gives
\begin{align*}
\Big\langle
\Tfwd_d^{b,\sigma}\circ\cdots\circ
\Tfwd_{t+1}^{b,\sigma}(\mathsf Z_t),
\mathsf R_{>d}^{b,\sigma}
\Big\rangle_{\mathrm F}
=
\big\langle
\mathsf Z_t,
\mathsf R_{>t}^{b,\sigma}
\big\rangle_{\mathrm F},
\end{align*}
which yields \eqref{eq:boundary_gramian_offdiagonal} upon multiplication by
$\lambda_{\mathrm b}\delta_{\ell0}\delta_{\ell'0}$. Figure~\ref{fig:boundary_offdiag_contraction} illustrates the case $s<b<t$. Other positions of $b$ are handled by the same face transfers \eqref{eq:boundary_weighted_transfers}: endpoint evaluation enters the left or right interface when $b<s$ or $b>t$, and the selectors when $b=s$ or $b=t$.

For the compressed residual, replacing the second model layer by the residual
cores $\mathsf Y_m^{b,\sigma}$ gives the mixed interfaces
\eqref{eq:boundary_mixed_interfaces} and, at the active coordinate,
\begin{align*}
\rho_{\mathrm c,s}^{\mathrm{ch},b,\sigma,(\ell)}
\bigl(\overline{i_s\alpha_{s-1}\alpha_s}\bigr)
&=
\sqrt{\lambda_{\mathrm b}}\,
\delta_{\ell0}\,
\mathsf D_s^{b,\sigma}(i_s,i_s)
\Big[
\mathsf L_{<s}^{b,\sigma,\rho}
\mathsf Y_s^{b,\sigma}(i_s)
\bigl(\mathsf R_{>s}^{b,\sigma,\rho}\bigr)^\top
\Big]_{\alpha_{s-1},\alpha_s}\\
&=
\sqrt{\lambda_{\mathrm b}}\,
\delta_{\ell0}\,
\mathsf P_s^{b,\sigma}(i_s)
(\alpha_{s-1},\alpha_s),
\end{align*}
which proves \eqref{eq:boundary_compressed_residual}.
Replacing both layers by the residual train yields the face-loss
contraction, completing the proof.
\end{proof}

We note that for a finite sum of TT residual trains, the compressed residuals are added
trainwise and the loss is evaluated over all train pairs. Stacking
\eqref{eq:boundary_compressed_residual} over $s$ and $\ell$ recovers
$\rho_{\mathrm c}^{\mathrm{ch},b,\sigma}$ defined in
\eqref{eq:boundary_compressed_quantities}. Summing the interior loss
\eqref{eq:discrete_loss} and the face losses
\eqref{eq:discrete_face_quadrature} gives
\begin{equation}\label{eq:boundary_total_loss}
\widehat{\mathcal L}(\theta)
=
\widehat{\mathcal L}_\D(\theta)
+
\sum_{b=1}^{\dsp}
\sum_{\sigma\in\{-,+\}}
\widehat{\mathcal L}_{b,\sigma}(\theta).
\end{equation}
On the common sensitivity space of \eqref{eq:boundary_common_sensitivities},
define the interior quantities
\begin{align}\label{eq:boundary_interior_compressed_quantities}
\Upsilon^{\mathrm{ch},\D}
:=
\frac1{N_\D}
(\Psi^{\mathrm{ch},\D})^\top
\Psi^{\mathrm{ch},\D}, \qquad
\rho_{\mathrm c}^{\mathrm{ch},\D}
:=
\frac1{N_\D}
(\Psi^{\mathrm{ch},\D})^\top
\rvec_\D.
\end{align}
By \eqref{eq:common_channel_fact}, the interior and face Jacobians share the
same $\Cmat^{\mathrm{ch}}$. Hence, using
\eqref{eq:boundary_interior_compressed_quantities} and
\eqref{eq:boundary_compressed_quantities}, their compressed contributions
can be added directly:
\begin{equation}\label{eq:boundary_total_compressed}
\begin{aligned}
\Upsilon^{\mathrm{ch}}
:=
\Upsilon^{\mathrm{ch},\D}
+
\sum_{b=1}^{\dsp}
\sum_{\sigma\in\{-,+\}}
\Upsilon^{\mathrm{ch},b,\sigma}, \qquad
\rho_{\mathrm c}^{\mathrm{ch}}
:=
\rho_{\mathrm c}^{\mathrm{ch},\D}
+
\sum_{b=1}^{\dsp}
\sum_{\sigma\in\{-,+\}}
\rho_{\mathrm c}^{\mathrm{ch},b,\sigma}.
\end{aligned}
\end{equation}
Consequently, by \eqref{eq:common_channel_fact} and
\eqref{eq:boundary_total_compressed},
$\DGram
=
(\Cmat^{\mathrm{ch}})^\top
\Upsilon^{\mathrm{ch}}
\Cmat^{\mathrm{ch}}$ and $\nabla\widehat{\mathcal L}
=
(\Cmat^{\mathrm{ch}})^\top
\rho_{\mathrm c}^{\mathrm{ch}}$, consistently with
Theorem~\ref{prop:local_channel_compression}. Hence, the compressed solver
of Section~\ref{sec:woodbury} applies unchanged.

\bibliographystyle{plain} 
\bibliography{bib}
\end{document}